\documentclass[reqno, 12pt]{amsart}
\usepackage{amsmath,amssymb,mathrsfs,amsthm,amsfonts,mathtools}
\usepackage[inline]{enumitem} 
\usepackage[usenames,dvipsnames]{xcolor}
\usepackage{graphicx} 
\usepackage[paper=letterpaper,margin=1in]{geometry}

\definecolor{hypercolor}{HTML}{003399}
\usepackage{hyperref}
 	
\hypersetup{colorlinks=true, linkcolor=hypercolor,citecolor=hypercolor}

\usepackage{newtxtext}

\numberwithin{equation}{section}
\allowdisplaybreaks

\newtheorem{thm}{Theorem}[section]
\newtheorem{lem}[thm]{Lemma}
\newtheorem{prop}[thm]{Proposition}
\newtheorem{innercustomprop}{Proposition}
\newenvironment{customprop}[1]
{\renewcommand\theinnercustomprop{#1}\innercustomprop}%
{\endinnercustomprop}
\newtheorem{cor}[thm]{Corollary}
\newtheorem{innercustomcor}{Corollary}
\newenvironment{customcor}[1]
{\renewcommand\theinnercustomcor{#1}\innercustomcor}%
{\endinnercustomcor}
\theoremstyle{definition}

\theoremstyle{remark}
\newtheorem{ex}[thm]{Example}
\newtheorem{rmk}[thm]{Remark}
\newtheorem{assu}[thm]{Assumption}

\newcommand{\e}{\varepsilon}

\newcommand{\sff}{\mathsf{f}}
\newcommand{\hatf}{\hat{f}}
\newcommand{\hatg}{\hat{g}}
\newcommand{\barg}{\overline{g}}

\renewcommand{\aa}{a}							

\newcommand{\calN}{\mathcal{N}}

\newcommand{\conc}{\mathrm{conc}}
\newcommand{\last}{\mathrm{last}}
\newcommand{\rest}{\mathrm{rest}}
\newcommand{\HS}{\mathrm{HS}}

\newcommand{\Msp}{\mathcal{M}_+} 				
\newcommand{\PM}{M}								
\newcommand{\PMa}{\mathbf{Q}}					
\newcommand{\PMw}{\overline{M}}						
\newcommand{\PMm}{M'}							
\newcommand{\bridge}{\operatorname{bridge}}		
\newcommand{\measure}{\mu}						
\newcommand{\measurew}{\overline{\mu}}				
\newcommand{\measuree}{\nu}						
\newcommand{\pmeasure}{\mu}						
\newcommand{\margin}[1]{|_{#1}{}}
\newcommand{\Path}{\gamma}						
\newcommand{\Pathh}{\eta}						
\newcommand{\vecpath}{\vec{\gamma}}				
\newcommand{\Psp}{\Gamma}						

\newcommand{\gmc}{\mathcal{G}_{0}}			
\newcommand{\gmcw}{\mathcal{G}}		            
\newcommand{\inters}{\tau}						
\newcommand{\intersop}{\mathsf{T}}				
\newcommand{\noisee}{\zeta}						
\newcommand{\noise}{\xi}						
\newcommand{\noiseu}{\xi_{\hilsp}}				
\newcommand{\noisesp}{\Delta}					

\newcommand{\filt}{\mathscr{F}}
\newcommand{\filtt}{\mathscr{G}}
\newcommand{\filttt}{\mathscr{H}}

\renewcommand{\ker}{\mathsf{k}}					
\newcommand{\kerop}{\mathsf{K}}					
\newcommand{\op}{V}								
\newcommand{\hsker}{v}							
\newcommand{\eigenval}{\lambda}					
\newcommand{\eigenfn}{u}						
\newcommand{\ortho}{\phi}						
\newcommand{\Yop}{Y}							
\newcommand{\Yopu}{\hat{Y}}						
\newcommand{\Yst}{Y_{\star}}					
\newcommand{\Ystt}[1]{Y_{\star\,#1}}			
\newcommand{\shift}{\phi}						
\newcommand{\isom}{O}							
\newcommand{\isomu}{U}					        
\newcommand{\proj}{\Pi}	              			
\newcommand{\wfn}{g_\mathrm{w}}					
\newcommand{\expfn}{g_\mathrm{exp}}				

\newcommand{\id}{\mathbf{1}}					
\newcommand{\sg}{\mathcal{Q}}					
\newcommand{\sgsum}{\mathcal{R}}				
\newcommand{\heatsg}{\mathcal{P}}				
\newcommand{\heatsger}[2]{\mathcal{P}^{\hphantom{#1}\varepsilon}_{#1#2}}
\newcommand{\heatsgel}[2]{\mathcal{P}^{\varepsilon\hphantom{#2}}_{#1#2}}
\newcommand{\Jop}{\mathcal{J}}					
\newcommand{\jfn}{\mathsf{j}}
\newcommand{\Sop}{S}							
\newcommand{\Top}{\mathcal{T}}					

\newcommand{\Cop}{\mathcal{C}}					
\newcommand{\Dop}{\mathcal{D}}					
\newcommand{\Eop}{\mathcal{E}}					
\newcommand{\Fop}{\mathcal{F}}					
\newcommand{\Uop}{\mathcal{U}}					
\newcommand{\Vop}{\mathcal{V}}					

\newcommand{\pair}{\mathrm{Pair}}
\newcommand{\intv}[1]{[\![#1]\!]}
\newcommand{\dgm}{\mathrm{Dgm}}
\newcommand{\vecalpha}{\vec{\alpha}}			
\newcommand{\vecsigma}{\vec{\sigma}}			
\newcommand{\vecu}{\vec{u}}						

\newcommand{\textc}{\mathrm{c}}
\newcommand{\textr}{\mathrm{r}}

\newcommand{\yc}{y_\mathrm{c}}

\newcommand{\yr}{y_\mathrm{r}}

\newcommand{\flow}{Z}							
\newcommand{\floww}{Z'}							
\newcommand{\wien}{\operatorname{Wein}}			
\newcommand{\hk}{\mathsf{p}}					
\newcommand{\embed}{\iota}						
\newcommand{\embedd}{\jmath}					
\newcommand{\unit}{q}							
\newcommand{\startt}{r}							
\newcommand{\refine}{\operatorname{Major}}

\newcommand{\opdot}{\circ}						
\newcommand{\opdott}{\raisebox{0.1ex}{\scalebox{.75}{$\odot$}}} 
\newcommand{\cldot}{\bullet}					

\newcommand{\N}{\mathbb{N}}
\newcommand{\R}{\mathbb{R}}
\newcommand{\Z}{\mathbb{Z}}
\newcommand{\Csp}{C}
\newcommand{\Cbsp}{C_\mathrm{b}}
\newcommand{\Cloc}{\widehat{\mathcal{C}}}
\newcommand{\Cfin}{C_\mathrm{fd}}

\newcommand{\Lsp}{L}
\newcommand{\BMsp}{L}							
\newcommand{\lsp}{\ell}
\newcommand{\hilsp}{\mathcal{H}}				
\newcommand{\msp}{\mathfrak{X}} 				
\newcommand{\pt}{\mathbf{x}}					
\newcommand{\vecpt}{\vec{\mathbf{x}}}			

\newcommand{\D}{\mathbb{D}}						
\newcommand{\I}{\mathbb{I}}						
\newcommand{\J}{\mathbb{J}}						
\newcommand{\intvs}{\mathcal{I}}				
\newcommand{\intvss}{\mathcal{J}}				
\newcommand{\Intvs}{\mathscr{I}}				

\newcommand{\norm}[1]{\Vert #1\Vert}
\newcommand{\Norm}[1]{\big\Vert #1\big\Vert}

\newcommand{\normop}[1]{\Vert #1\Vert}
\newcommand{\Normop}[1]{\big\Vert #1\big\Vert}
\newcommand{\NOrmop}[1]{\Big\Vert #1\Big\Vert}
\newcommand{\ip}[1]{\langle #1\rangle}
\newcommand{\Ip}[1]{\big\langle #1\big\rangle}
\newcommand{\IP}[1]{\Big\langle #1\Big\rangle}
\newcommand{\ipp}[1]{(#1)}
\newcommand{\Ipp}[1]{\big( #1\big)}

\newcommand{\cdott}{\,\cdot\,}

\newcommand{\E}{\mathbf{E}}
\newcommand{\EE}{\mathbb{E}}
\renewcommand{\P}{\mathbf{P}}
\newcommand{\PP}{\mathbb{P}}
\newcommand{\Var}{\mathbf{Var}}
\newcommand{\ind}{1}					
\renewcommand{\d}{\mathrm{d}}		

\newcommand{\til}{\widetilde}
\newcommand{\supp}{\operatorname{supp}}
\newcommand{\nullsp}{\operatorname{null}}		
\newcommand{\range}{\operatorname{range}}		

\title{Conditional GMC within the stochastic heat flow}
\author{Jeremy Clark}
\address[Jeremy Clark]{Department of Mathematics, University of Mississippi}

\author{Li-Cheng Tsai}
\address[Li-Cheng Tsai]{Department of Mathematics, University of Utah}

\begin{document}
\begin{abstract} 
We establish that the family of polymer measures $\PM^{\theta}_{[s,t]}$ associated with the Stochastic Heat Flow (SHF), indexed by $\theta\in\R$, has a conditional Gaussian Multiplicative Chaos (GMC) structure. Namely, taking the random measure $\PM^{\theta}_{[s,t]}$ as the reference measure, we construct the path-space GMC with noise strength $\aa > 0$ and prove that the resulting random measure is equal in law to $\PM^{\theta+\aa}_{[s,t]}$. As two applications, we prove that the polymer measure and SHF tested against general nonnegative functions are almost surely strictly positive and that the SHF converges to $0$ as $\theta\to\infty$.
\end{abstract}

\maketitle

\section{Introduction}
\label{s.intro}

In this paper, we study the polymer measure associated with the Stochastic Heat Flow (SHF).
The SHF is the scaling limit of directed polymers and the noise-mollified Stochastic Heat Equation (SHE) in the critical dimension $1+2$ and around the critical temperature.
The SHF behaves dramatically differently from the solution of the $1+1$ dimensional SHE, and its construction and characterization already pose significant challenges.
The work \cite{caravenna2023critical} derived the finite-dimensional distributions (in time) of the SHF as a universal limit of discrete polymer models.
The SHF with parameter $\theta\in\R$ is a two-parameter stochastic process in $s\leq t\in\R$ that takes values in the set of positive measures on $\R^4$
\begin{align*}
	\flow^{\theta}_{s,t} = \flow^{\theta}_{s,t}(\d x, \d x') \in \Msp(\R^4)\ ,
	\quad
	x,x'\in\R^2\ .
\end{align*}
The parameter $\theta$ can be viewed as a renormalized coupling constant, with $\theta\to-\infty$ and $\theta\to+\infty$ respectively representing the weak- and strong-disorder regimes.
Recently, \cite{tsai2024stochastic} provided an axiomatic formulation of the SHF as a continuous process and established the uniqueness and existence under this formulation.
Around the same time, \cite{clark2024continuum} constructed the (unnormalized) polymer measure $\PM^{\theta}=\PM^{\theta}_{[s,t]}$ associated with the SHF.
In particular, $Z^{\theta}_{s,t}$ is the initial-final time marginal of $\PM^{\theta}_{[s,t]}$.
Just like the SHF, $\PM^{\theta}$ is expected to behave qualitatively differently from its $1+1$ dimensional counterpart constructed in \cite{alberts2014intermediate}.

A question of interest is whether the polymer measure $\PM$ is a path-space Gaussian Multiplicative Chaos (GMC).
At an informal level, a GMC is a random measure of the form
\begin{align*}
    \gmcw(\d\Path)
    =
    \measure(\d\Path)
    \,
    \exp\big(\noisee(\Path)-\tfrac{1}{2}\EE\noisee(\Path)^2\big)  \ ,
\end{align*}
where $\measure$ is a deterministic measure, called the \emph{reference measure}, and $\noisee$ is a Gaussian field having correlations $\EE[ \noisee(\Path)\noisee(\Pathh)  ]=\ker(\Path,\Pathh)$ for a given kernel $\ker$.
Often $\noisee$ is not function-valued, and hence the informal expression above does not make sense merely as a Radon-Nikodym relation between $\measure$ and $\gmc$.
Still, the theory of GMC, which commenced with the article \cite{kahane1985} and is reviewed in \cite{rhodes2014gaussian}, provides methods for rigorously defining random measures of this type under suitable assumptions.
We will consider GMCs on the \emph{path space} 
\begin{align}
	\label{e.pathsp}
	\Path \in \Psp_{[s,t]} := \Csp([s,t],\R^d)\ ,
	\quad
	d=2\ .
\end{align}
Pre-limiting Gaussian polymer measures are path-space GMCs.
Take, for instance, the noise-mollified SHE.
It is not difficult to check that its polymer measure is a GMC with reference measure being the Wiener measure and Gaussian field $\noisee(\Path):=\beta_\e\int_s^t\d u\, \noise_\e(u,\Path(u))$, where $\noise_\e$ is obtained by spatially mollifying the spacetime white noise at a fixed scale $\e>0$, and $\beta_\e>0$.
One may expect the limiting polymer measure to be a GMC also, which the article \cite{quastel2022kpz} demonstrated to be the case in dimensions $1+1$. However, in $1+2$ dimensions, the measure $\PM^{\theta}$ \emph{cannot} be a GMC, as was explained in \cite[Section~2.6]{clark2024continuum}.
We mention that \cite{caravenna2023gmc} studied what can be understood as the final-time marginal of $\PM^{\theta}$, showing that it is not a GMC.
We emphasize that our notation of GMC concerns the path space, as opposed to the final-time marginal.
The two do not have direct relations to our knowledge.


Even though the random measures $\PM^{\theta}=\PM^{\theta}_{[s,t]}$ cannot be constructed as a GMC, the family $(\PM^{\theta})_{\theta\in \R}$  possesses a conditional GMC structure.  Before making this claim more precise, let us explain the idea.
First, the expectation $\E\PM^{\theta}$ is a Wiener measure (see Section~\ref{s.moment.annealed}), implying that if $\PM^{\theta}$ were a GMC, its reference measure could be taken to be said Wiener measure. 
Second, this Wiener measure is the $\theta'\to-\infty$ limit of $\PM^{\theta'}$ (see \eqref{e.infiniteT.lim}), that is, $\PM^{-\infty}=\E\PM^{\theta}$.
Consequently, we can understand the impossibility of constructing $\PM^{\theta}$ as a GMC in terms of the need for an infinitely strong  ``coupling strength'' to go from parameter $-\infty$ to $\theta$.
However, if instead of taking the deterministic Wiener measure $\PM^{-\infty}$ as the reference measure, we take $\PM^{\theta_0}$ for some $\theta_0\in(-\infty,\theta)$---making the difference $\theta-\theta_0$ finite---, we can expect that $\PM^{\theta}$ is a GMC wrt $\PM^{\theta_0}$. 
Since $\PM^{\theta_0}$ is random, we view this GMC as being conditional on $\PM^{\theta_0}$, and hence the name \emph{conditional} GMC.
This gives rise to a structure governing the interrelationship between the laws of the random measures $\PM^{\theta}$, $\theta\in \R$.

We prove this conditional GMC structure in the present work.
Stating the main result precisely requires some GMC-related definitions, which we defer to Section~\ref{s.tools}, as indicated below.
Given any Polish space $\msp$, call a positive measure on $\msp$ \textbf{locally finite} if it is finite on bounded subsets of $\msp$, and let $\Msp(\msp)$ be the set of locally finite positive Borel measures.
In Section~\ref{s.tools.gmc}, we will use the formulation in \cite{shamov2016gaussian} to define the GMC 
\begin{align}
	\label{e.gmcw.show}
	\gmcw=\gmcw[\pmeasure,\Yop,\noise] \in \Msp(\msp)
\end{align}
that depends on a reference measure $\pmeasure\in\Msp(\msp)$, an operator $\Yop$ that factorizes the correlation kernel $\ker$ (see \eqref{e.Yop}), and a Gaussian vector $\noise=(\noise_1,\noise_2,\ldots)$, where $\noise_1,\noise_2,\ldots$ are iid, mean-zero, $\R$-valued Gaussian.
It is common to assume that $\EE\noise_1^2=1$, but we choose to keep $\EE\noise_1^2:=\aa\in[0,\infty)$ flexible.
For deterministic $\pmeasure$ and $\Yop$, the GMC \eqref{e.gmcw.show} is random only through its dependence on $\noise$.
In our case, with $\msp$ being the path space $\Psp_{[s,t]}$ equipped with the uniform metric, we take the reference measure $\pmeasure$ to be the polymer measure $\PM^{\theta}_{[s,t]}\in\Msp(\Psp_{[s,t]})$ (see Section~\ref{s.tools.PM}), which is random and taken to be \emph{independent} of $\noise$, and we will use $\PM^{\theta}_{[s,t]}$ to construct an operator $\Ystt{[s,t]}^{\theta}$ in \eqref{e.Ystt}, which is hence $\PM^{\theta}_{[s,t]}$-measurably random.
Here is the main result of this paper.

\begin{thm}\label{t.main}
With $s<t$, $\theta\in\R$, and $\EE\noise_1^2:=\aa$,
\begin{align}
	\label{e.t.main}
	\gmcw\big[ \PM^{\theta}_{[s,t]}, \Ystt{[s,t]}^{\theta}, \noise \big]
	\text{ is equal in law to }
	\PM^{\theta+\aa}_{[s,t]}\ .
\end{align}
\end{thm}
\noindent{}%
The pair of articles \cite{clark2022continuum,clark2023conditional} proved a similar conditional GMC structure for polymer measures defined on  diamond hierarchical lattices.
Based on these studies, the first listed author of the current paper conjectured that the same structure holds in the $1+2$ dimensional Euclidean space, which we now settle in Theorem~\ref{t.main}.

Much of the difficulty in proving Theorem~\ref{t.main} stems from the (expected) extremely localized behavior of polymer measures in critical dimensions.
More precisely, we conjecture that paths sampled from $\PM^{\theta}_{[s,t]}$ would \emph{only visit a random fractal} of zero Lebesgue measure in spacetime.
This conjecture is supported by \cite[Theorem~1.1]{caravenna2025singularity}, which asserts that the time-$t$ marginal of $\PM^{\theta}_{[s,t]}$ is singular wrt the Lebesgue measure.
In physics terms, when the localized behavior is present, the conditional GMC needs to be constructed from the spacetime white noise ``restricted'' to a random fractal, which is in stark contrast to subcritical systems such as the $1+1$ dimensional continuum directed polymer.
Given this, even the construction of the conditional GMC is subtle, and more so for proving a statement like Theorem~\ref{t.main}.
For the polymer measure on the hierarchical lattice, these difficulties were overcome in \cite{clark2022continuum,clark2023conditional} by using the hierarchical structure.
In our setting, no such structure is available, and we have to devise a completely different construction and proof, which we describe in more detail in Section~\ref{s.intro.singular}.


\subsection{Applications}
\label{s.intro.apps}
In addition to answering a question of interest, Theorem~\ref{t.main}  provides a \emph{tool} for studying the polymer measure and SHF.
The tool is a coupling of $\PM^{\theta}_{[s,t]}$ for different $\theta\in\R$.
First, by Theorem~\ref{t.main}, for every fixed $\theta_0$, the conditional GMC $\gmcw[\PM^{\theta_0}_{[s,t]},\Ystt{[s,t]}^{\theta_0},\noise]$ with $\EE\,\noise_1^2=\aa$ gives a coupling of $\PM^{\theta_0}_{[s,t]}$ and $\PM^{\theta_0+\aa}_{[s,t]}$.
The idea is to vary $\aa$ by taking each $\noise_i=\noise_i(\aa)$ to be a Brownian motion in $a$, and to take smaller and smaller $\theta_0$ tending to $-\infty$.
After proving some consistency properties, which we do in Section~\ref{s.apps.tool}, the idea leads to a coupling of $\PM^{\theta}_{[s,t]}$ for all $\theta\in\R$.
We state this as Proposition~\ref{p.gmc.coupling} and prove it in Section~\ref{s.apps.tool}.
Write $\noise(\aa)=(\noise_{1}(\aa),\noise_{2}(\aa),\ldots)$ for the Gaussian vector whose components are independent two-sided Brownian motions in $\aa\in\R$, and write $\noise(\theta,\theta'):=\noise(\theta')-\noise(\theta)$ for its increment.
Recall from before Theorem~\ref{t.main} that $\Ystt{[s,t]}^{\theta}$ is $\PM^{\theta}_{[s,t]}$ measurable.

\begin{prop}\label{p.gmc.coupling}
There exists a coupling of the family $(\PM^{\theta}_{[s,t]})_{ \theta\in\R}$ such that, with $\filttt_{\theta}:=\sigma(\PM^{\phi}_{[s,t]}\,|\,\phi\leq\theta)$ denoting its canonical filtration,
\begin{align}
\label{e.gmc.coupling}
\begin{split}
    \text{the law of } &(\PM^{\theta}_{[s,t]})_{\theta\geq\theta_0}
    \text{ conditioned on } \filttt_{\theta_0}
\\
    &\text{ is equal to the law of }
\big(\gmcw\big[\PM^{\theta}_{[s,t]},\Ystt{[s,t]}^{\theta},\noise(\theta_0,\theta)\big]\big)_{\theta\geq\theta_0}\ .
\end{split}
\end{align}
In particular, $\PM^{\theta}_{[s,t]}$ is a Markovian martingale in $\theta$.
\end{prop}
\noindent
We call this coupling the \textbf{GMC coupling}.

The GMC coupling leads to concrete results about the polymer measure and SHF.
We demonstrate this point by proving Corollaries~\ref{c.positivity}--\ref{c.superc}, stated below.
The proof of these results, which we carry out in Sections~\ref{s.apps.}--\ref{s.apps.moment}, is fairly simple given Proposition~\ref{p.gmc.coupling}.
Roughly speaking, the idea is to vary $\theta$ with the aid of some asymptotic properties of the polymer measures as $\theta\to-\infty$.

Next we prepare for the statement of Corollary~\ref{c.positivity}.
Let $\PP_{x,[s,t]}$ denote the law of a Brownian motion that starts from $x$ at time $s$ and runs until time $t$, and define the Wiener measure
\begin{align}
    \label{e.wien}
    \wien_{[s,t]}:= \int_{\R^{d} }\d x \, \PP_{x,[s,t]}\ ,
    \quad
    d = 2\ .
\end{align}
Given a metric space $\msp$, let $\BMsp(\msp)$ denote the space of all $\R$-valued Borel measurable functions on $\msp$ and $\Cloc(\msp)$ denote the set of continuous functions with bounded supports. We endow $\Msp(\msp)$ with the vague topology, that is, the coarsest topology such that $\measure\mapsto \measure f:=\int_{\msp}\measure(\d\pt)\,f(\pt)$ is continuous for all $f\in\Cloc(\msp)$.
By \cite{tsai2024stochastic}, the SHF $Z^{\theta}_{s,t}$ is an $\Msp(\R^4)$-valued continuous process in $s\leq t\in\R$ that is uniquely characterized by a list of axioms, which we recall in Section~\ref{s.intro.axioms}.

\begin{cor}\label{c.positivity}
For any deterministic $ 0\leq f\in\BMsp(\Psp_{[s,t]})$ that is not $\wien_{[s,t]}$-a.e.\ zero,
\begin{align}
	\label{e.c.positivity.PM}
	\PM^\theta_{[s,t]} f > 0
	\quad
	\text{a.s.}
\end{align}
In particular, for any deterministic $ 0\leq g,g'\in\BMsp(\R^2)$ that are not Lebesgue-a.e.\ zero,
\begin{align}
	\label{e.c.positivity}
	\flow^\theta_{s,t} \ g \otimes g'
	:=
	\int_{\R^4} \flow^{\theta}_{s,t}(\d x, \d x')\ g(x) \, g'(x') > 0
	\quad
	\text{a.s.}
\end{align}
\end{cor}

We now make a few remarks regarding the above corollary.
First, note that the assumption that $f$ is not $\wien_{[s,t]}$-a.e.\ zero is minimal, because $\E\PM^{\theta}_{[s,t]}=\wien_{[s,t]}$ (see Section~\ref{s.moment.annealed}). 
Next, in the polymer language, the left-hand side of \eqref{e.c.positivity} is the partition function.
Given \eqref{e.c.positivity}, one can now define the \emph{normalized} polymer measure
\begin{align}
	\frac{1}{\flow^\theta_{s,t} \ g \otimes g'} \, \PM^{\theta}_{[s,t]}(\d\Path)\, g(\Path(s)) \, g'(\Path(t)) \ ,
\end{align}
for $g,g'$ that make the denominator finite and satisfy the assumptions of Corollary~\ref{c.positivity}.
The result~\eqref{e.c.positivity} also answers the question about infinite speed of propagation in \cite[Section~11.2.2]{caravenna2024critical}.
Analogous positivity results in $1+1$ dimensions were obtained in \cite{mueller1991support,morenoflores2014positivity} by different means.

The second application shows that SHF($\theta$) converges to $0$ as $\theta\to\infty$.
\begin{cor}\label{c.superc}
As $\theta\to\infty$, $\flow^\theta_{s,t}\to 0$ vaguely a.s.
\end{cor}
\noindent{}%
We actually prove a slightly stronger mode of convergence than vague.
Stating it requires more notation, so we defer it to Corollary~\ref{c.superc.} in Section~\ref{s.apps.}.
A similar result was recently achieved in \cite[Theorem~1.4]{caravenna2025singularity} by different means.
See Remark~\ref{r.superc} for a comparison of the two results.

\subsection{Axiomatic characterization}
\label{s.intro.axioms}
First, recall from \cite{tsai2024stochastic} that the law of SHF($\theta$) is uniquely characterized by the following axioms.
The definition of the product $\bullet$ will be recalled in Section~\ref{s.tools.PM} and the definition of $\sg^{\intv{n},\theta}$ will be recalled in Section~\ref{s.moment.sg}.
\begin{enumerate}
\item \label{d.shf.conti}
$Z=Z^{\theta}_{s,t}$ is an $\Msp(\R^4)$-valued, continuous process in $s\leq t\in\R$.
\item \label{d.shf.ind}
$\flow^{\theta}_{s,t}$ and $\flow^{\theta}_{t,u}$ are independent for all $s< t< u$.
\item \label{d.shf.ck}
$\flow^{\theta}_{s,t}\bullet \flow^{\theta}_{t,u}=\flow^{\theta}_{s,u}$ for all $s< t< u$.
\item \label{d.shf.mome}
$
	\E \flow^{\theta}_{s,t}{}^{\otimes n}
	=
	\d x\, \d x' \, \sg^{\intv{n},\theta}(t-s,x,x')
$
for all $s<t$ and $n=1,2,3,4$.
\end{enumerate}

Our proof of Theorem~\ref{t.main} utilizes a version of the above axiomatic characterization.
Rename the $s,t$ in Theorem~\ref{t.main} to $\startt,\unit$ to avoid overloading notation here.
To prove Theorem~\ref{t.main}, since the law of $\PM^{\theta}_{[\startt,\unit]}$ is translation invariant in time (by \cite[Proposition 2.9]{clark2024continuum}), without loss of generality, we take $\startt=0$.
Let $\unit\D:=\cup_{m\geq 1} \unit 2^{-m}\Z$ be the set of $\unit$-dyadic numbers.
In much of this paper, to avoid unnecessary technical difficulties, we will only work with $s<t \in \unit\D$.
Accordingly, we will use the axiomatic characterization of the polymer measure stated in the proposition below, which we prove in Appendix~\ref{s.a.PM} using \cite[Theorem~1.9]{tsai2024stochastic}.
For $A\subset [s,t]$, write $\margin{A}:\Psp_{[s,t]}\to\Psp_{A}$ for the restriction map $\Path\mapsto\Path\margin{A}$,
When $A$ is finite, we interpret $\Psp_{A}$ as $ (\R^2)^{A}=\R^{2A}$ and view $\margin{A}:\Psp_{[s,t]}\to\R^{2A}$ as the evaluation map $\Path\mapsto(\Path(u))_{u\in A}$.
Accordingly, for $\pmeasure\in\Msp(\Psp_{[s,t]})$, we write $\pmeasure\circ\margin{A}^{-1}$ for the pushforward of $\pmeasure$ by $\margin{A}$ and call it the \textbf{time-$A$ marginal of $\pmeasure$}.
\begin{prop}
\label{p.axioms}
Take $\unit\in(0,\infty),\theta'\in\R$, and any probability space that houses random $\PMm_{[s,t]}\in\Msp(\Psp_{[s,t]})$ for $s<t\in \unit\D$. Write $\floww_{s,t}:=\PMm_{[s,t]}\circ\margin{\{s,t\}}^{-1}$ for the time-$\{s,t\}$ marginal.
If for all $t_0<\cdots<t_{\ell}\in\unit\D$ we have that
\begin{enumerate}
\setcounter{enumi}{1}
\item \label{p.axioms.ind}
$\PMm_{[t_0,t_1]},\ldots,\PMm_{[t_{\ell-1},t_{\ell}]}$ are independent, 
\item \label{p.axioms.coupling}
$\PMm_{[t_0,t_\ell]}\circ\margin{\{t_0,\ldots,t_\ell\}}^{-1} = \floww_{t_0,t_1}\opdot\cdots\opdot\floww_{t_{\ell-1},t_{\ell}}$\ , and
\item \label{p.axioms.mome}
$
	\E \floww_{s,t} {}^{\otimes n}
	=
	\d x\, \d x' \, \sg^{\intv{n},\theta'}(t-s,x,x')
$
for all $s<t \in \unit\D$ and $n=1,2,3,4$,
\end{enumerate}
then $\PMm_{[0,\unit ]}$ is equal in law to $\PM^{\theta'}_{[0,\unit ]}$.
\end{prop}

\subsection{Moment matching}
\label{s.intro.mome}
In addition to the difficulty caused by the localized behaviors mentioned after Theorem~\ref{t.main}, verifying the moment matching in Proposition~\ref{p.axioms}\eqref{p.axioms.mome} is a task that is by no means straightforward. 
The moment matching was achieved in \cite{clark2023planar} up to $n=2$ using martingale techniques, which rely on properties of the quenched polymer measures that seem available only for $n\leq 2$.
Here we take a different route through the operator techniques developed in \cite{gu2021moments}.
Using the operator techniques, however, requires taking certain local-time like limits and invoking combinatorial resummings.
We carry these out in Section~\ref{s.moment} and prove the matching for all $n\in\N:=\{1,2,\ldots\}$.

\subsection{Localized behavior, challenges, ideas}
\label{s.intro.singular}

Let us explain how the localized behavior mentioned after Theorem~\ref{t.main} poses a challenge to proving Theorem~\ref{t.main}.

For the purpose of comparison, let us briefly recall the GMC construction in $1+1$ dimensions from \cite[Sections~3.3, 7.1]{quastel2022kpz}.  We take $[0,1]$ as our time interval here and write $\wien=\wien_{[0,1]}$. Let $\noise_1,\noise_2,\ldots$  be as before Theorem~\ref{t.main} (recall that $\EE\noise_1^2=a$), and take any $\ortho_1,\ortho_2,\ldots\in\Csp([0,1]\times\R)$ that form an orthonormal basis wrt to $\d t \d x$.
The $1+1$ dimensional continuum polymer measure can be constructed as
\begin{align}
	\label{e.1+1}
	\lim_{\ell\to\infty} 
	\wien(\d\Path)\,
	\exp\sum_{i=1}^{\ell} 
        \Big(
		  \noise_i\int_{0}^1 \d t\, \ortho_{i}(t,\Path(t)) 
            - \frac{\aa}{2} \, (\text{same integral})^2
        \Big)
	\ .
\end{align}
The expression within the limit is a measure-valued martingale, and this fact together with known moment bounds ensures that the limit exists non-trivially.
Note that the $\noise_i$s and $\ortho_i$s realize the spacetime white noise $[0,1]\times\R$ through $\noisee:=\sum_{i} \noise_i \,\ip{ \ortho_i, \cdott }$, and, accordingly, \eqref{e.1+1} gives a rigorous construction of the non-rigorous expression 
\begin{align*}
	\wien(\d \Path)\,\exp\Big(\int_{0}^1 \d t\, \noisee(t,\Path(t))-\frac{1}{2} \,\EE\big[(\text{same integral})^2\big]\Big)\ .
\end{align*}

In $1+2$ dimensions, we would like to replace $\wien$ with $\PM=\PM^{\theta}_{[0,1]}$ and run a similar construction, but it will \emph{not} work if $\PM$ exhibits the anticipated localized behavior mentioned after Theorem~\ref{t.main}.
Observe that the construction \eqref{e.1+1} uses an orthonormal basis $\ortho_1,\ortho_2,\ldots$ wrt $\d t \d x$ to model the spacetime white noise.
For $\PM$, however, we expect the paths sampled from a realization of $\PM$ to only visit a random, fractal subset of $[0,1]\times\R^2$.
In this case, the basis needs to be chosen to model the white noise ``concentrated'' on a random Lebesgue null subset of $[0,1]\times\R^2$.


Except for the instance studied in \cite{clark2022continuum} in the diamond hierarchical lattice setting, we are not aware of any method for constructing the conditional GMCs that addresses the difficulty arising from the localized behavior. The aforementioned work achieved the construction through a factorization identity (Proposition 10.12 therein) whose proof crucially utilized the hierarchical structure. Proving its analog for $\PM$ remains open.

As a first step towards circumventing this issue, we construct the conditional GMC in the path space without going through the spacetime, following the spirit of the abstract GMC formulation of \cite{shamov2016gaussian}.  
In other terms, we avoid the localized behavior in spacetime by bypassing spacetime altogether.
To see how, refer back to \eqref{e.1+1}, set $v_{i}(\Path):=\int_{0}^{1} \d t \, \ortho_i(t,\Path(t))$, and note that the GMC \eqref{e.1+1} can be defined through just $v_{1},v_{2},\ldots$.
The latter are functions on the path space $\Psp_{[0,1]}$, unlike the $\ortho$s, which are functions on the spacetime $[0,1]\times\R^{d}$.
In general, given a correlation kernel on the path space (or any Polish space) that defines a positive Hilbert--Schmidt operator, under suitable assumptions, we can take the functions $v_{i}(\Path)$ to be the eigenfunctions of the operator and use them to construct the GMC.
In Section~\ref{s.tools.gmc}, we explain how this works and apply it to $\PM$.


The path-space construction alone does not solve the issue, and there are still challenges to overcome.
Since our proof relies on Proposition~\ref{p.axioms}, we need to couple the conditional GMCs over different time intervals, which amounts to coupling the Gaussian vectors (the $\noise$s) used to build the GMCs.
Take the intervals $[0,1],[1,2],[0,2]$ for example, and denote the respective Gaussian vectors by $\noise_{[0,1]},\noise_{[1,2]},\noise_{[0,2]}$.
For the GMC \eqref{e.1+1} in $1+1$ dimensions, since $\noise_{I}$ represents the spacetime white noise over $I\times\R$, the coupling is simply $\noise_{[0,2]}=\noise_{[0,1]}\oplus\noise_{[1,2]}$, where $\noise_{[0,1]}$ and $\noise_{[1,2]}$ are independent.
For our conditional GMC in $1+2$ dimensions, the coupling should be much more complicated due to the (expected) localized behavior.
We construct the coupling by analyzing the correlation kernels.

The analysis of the correlation kernels, which we carry out in Sections~\ref{s.embedding}--\ref{s.coupling}, encounters a likely anomalous behavior of $\PM$.
For a fixed $[s,t]$, the kernel acts on the $\Lsp^2$ space over $\Psp_{[s,t]}$ wrt the polymer measure $\PM^{\theta}_{[s,t]}$ (or, more precisely, a weighted version of it defined in Section~\ref{s.tools.inters}).
The issue is that, when different intervals are considered, the measures are likely mutually singular.
Take $[0,1],[0,2]$ for example, let $\PM_1:=\PM^{\theta}_{[0,1]}$, and let $\PM'$ be the time-$[0,1]$ marginal of $\PM^{\theta}_{[0,2]}$.
For reasons explained in Remark~\ref{r.singular}, we conjecture that $\PM_1$ and $\PM'$ are a.s.\ mutually singular.
This causes much difficulty since we need to compare functions in different $\Lsp^2$ spaces wherein the underlying measures are potentially mutually singular.
We overcome this obstacle by appealing to a martingale structure (Proposition~\ref{p.prob.space}\eqref{p.prob.space.2}--\eqref{p.prob.space.3}) that allows us to relate measures like $\PM_1$ and $\PM'$.

\subsection*{Acknowledgment}
We thank B\'{a}lint Vir\'{a}g and Mo Dick Wong for useful discussions.
The research of LCT is partially supported by the NSF through DMS-2243112 and the Alfred P.\ Sloan Foundation through the Sloan Research Fellowship FG-2022-19308.

\subsection*{Outline}
In Section~\ref{s.tools}, we recall the relevant background, prepare a few tools, and explain how to construct the conditional GMC over a fixed time interval.
In Section~\ref{s.moment}, we prove the moment matching required by Proposition~\ref{p.axioms}\eqref{p.axioms.mome}.
In Section~\ref{s.embedding}, we establish an operator embedding result to be used in Section~\ref{s.coupling}.
In Section~\ref{s.coupling}, we construct a coupling of the conditional GMCs over different time intervals and use the coupling to prove Theorem~\ref{t.main}.
In Section~\ref{s.apps}, we prove Corollaries~\ref{c.positivity}--\ref{c.superc} based on Theorem~\ref{t.main}.
To streamline the presentation, the proof of some technical results are placed in the appendices, with explicit references made in the main text.

\section{Background and tools}
\label{s.tools}

\subsection{Polymer measure}
\label{s.tools.PM}
Let us recall the construction of the polymer measure $\PM^{\theta}_{[s,t]}$ from \cite{clark2024continuum}.
The overall idea is similar to that of \cite{alberts2014intermediate}, but the required Chapman--Kolmogorov equation is more delicate.

Our construction begins with defining suitable products between measures on $\R^4$. For any $\measuree_1(\d x,\d x'), \measuree_2(\d x,\d x')\in\Msp(\R^4)$, define the \textbf{left-} and \textbf{right-indented $\opdot_r$} respectively as
\begin{align}
	\label{e.opendotl}
	\big( \measuree_{1} \opdot_r \measuree_{2} \big)(\d x,\d x', \d x'')
	&:=
	\int_{y\in\R^{2}}
	\measuree_{1} (\d x, \d y) \, \hk(r,y-x') \, \measuree_{2} (\d x', \d x'')\ ,
\\
	\label{e.opendotr}
	\big( \measuree_{1} \opdot_r \measuree_{2} \big)(\d x,\d x', \d x'')
	&:=
	\int_{y\in\R^{2}}
	\measuree_{1} (\d x, \d x') \, \hk(r,x'-y) \, \measuree_{2} (\d y, \d x'')\ ,
\end{align}
where $\hk(r,x):=\exp(-|x|^2/2r)/2\pi r$ denotes the heat kernel on $\R^2$. The idea in \eqref{e.opendotl}--\eqref{e.opendotr} is to bring $y$ close to $x'$ through an approximation of the delta function $f_r(y-x')$, which is chosen to be the heat kernel $f_r=\hk(r,\cdott)$. We use the term \textit{indented} in reference to these operations because we will apply them to the SHF with
\begin{align*}
    \measuree_1=\flow^{\theta}_{s,t-r}\ , \ 
    \measuree_2=\flow^{\theta}_{t,u} \text{ in (\ref{e.opendotl})}, 
    \quad \text{and}\quad
    \measuree_1=\flow^{\theta}_{s,t}\ , 
    \measuree_2=\flow^{\theta}_{t+r,u}
    \text{ in (\ref{e.opendotr}) }
\end{align*}
for $s<t<u$ and small $r>0$. These choices for $\measuree_1,\measuree_2$ yield a martingale structure since
$
	\E\flow^{\theta}_{s,t}(\d x, \d x')=\d x \d x' \hk(t-s,x-x')
$.
It was shown in \cite[Theorem~2.8]{clark2024continuum} that the limit
\begin{align}
	\label{e.opendot}
	\flow^{\theta}_{t_0,t_1} \opdot \cdots \opdot \flow^{\theta}_{t_{\ell-1},t_\ell}
	&:=
	\lim_{r\to 0} 
	\flow^{\theta}_{t_0 ,t_1} \opdot_r \cdots \opdot_r \flow^{\theta}_{t_{\ell-1},t_\ell}
	\in
	\Msp(\R^{2\{t_0 ,\ldots,t_{\ell}\}})
\end{align}
exists vaguely a.s.\ and vaguely in $\Lsp^2$, where the multiple product is defined similarly, each $\opdot_r$ can be left- or right-indented (or even more generally, which we will not use), and the limit does not depend on the ways of indentation.
Later \cite{tsai2024stochastic} showed that the limit exists vaguely in $\Lsp^n$ for all $n\in\N$, that the heat kernel can be replaced by a general class of functions, and that the limit does not depend on the choice of function.
In \eqref{e.opendot}, integrating out the middle variables leads to the product
\begin{align}
\label{e.closedot}
	\flow^{\theta}_{t_0 ,t_1} \cldot \flow^{\theta}_{t_{1},t_2}(\d x_0, \d x_{2})
	:=
	\flow^{\theta}_{t_0 ,t_1} \opdot \flow^{\theta}_{t_1,t_2} \big( \d x_0, \R^2, \d x_2)
\end{align}
for $\ell=2$ and similarly for $\ell>2$.

To construct the polymer measure, define the finite-dimensional marginals as
\begin{align}
	\label{e.PM.fdd}
	\PM^{\theta}_{t_0 ,\ldots,t_\ell} := \flow^{\theta}_{t_0 ,t_1} \circ \cdots \circ \flow^{\theta}_{t_{\ell-1},t_\ell}\ , 
	\quad
	t_0 <\cdots<t_\ell\ ,
\end{align}
and note that by Axiom~\eqref{d.shf.ck} these marginals enjoy the Kolmogorov consistency.
Given the consistency, the argument in \cite[Section~5.1]{clark2024continuum} shows that these marginals uniquely extend to a random measure on $\Psp_{[s,t]}$, which we call the \textbf{polymer measure} $\PM^{\theta}_{[s,t]}$. 
The proposition below states a useful conditional expectation property, which is phrased in terms of the family of polymer measures $\PM^{\theta}_{[s,t]}$ with $s<t\in \unit\D$ (the $\unit$-dyadic numbers, as in Proposition~\ref{p.axioms}).
The proof is placed in Appendix~\ref{s.a.PM}.
Let $\bridge_{s,x\to t,x'}\in\Msp(\Psp_{[s,t]})$ denote the \emph{unnormalized} Brownian bridge measure 
\begin{align}
	\label{e.bridge}
	\bridge_{s,x\to t,x'} \psi
	:=
	\EE_{x}\big[ \psi(B) \, \big| B(t) =x' \big]\, \hk(t-s,x'-x)\ ,
\end{align}
where $B$ denotes the Brownian motion with $B(s)=x$.
For $\pmeasure_{[s,t]}\in\Msp(\Psp_{[s,t]})$ and $t_1<\cdots<t_{2\ell}\in\unit\D$, put $I_{i}:=[t_{2i-1},t_{2i}]$, $r_i := t_{2i+1}-t_{2i}$, and define 
\begin{align}
	\label{e.opendot.path}
	\big( \pmeasure_{I_1} \,\opdott_{\,r_1}\, \cdots \,\opdott_{\,r_{\ell-1}}\, \pmeasure_{I_{\ell}} \big)(\d\Path)
	:=
	\bigotimes_{i=1}^{\ell} \pmeasure_{I_i}(\d\Path_i) \cdot
	\prod_{i=1}^{\ell-1} \bridge_{t_{2i},\Path_i(t_{2i})\to t_{2i+1},\Path_{i+1}(t_{2i+1})}(\d\Path'_i) \ ,
\end{align}
where $\Path\in\Psp_{[t_0,t_{2\ell}]}$, $\Path_i:=\Path|_{I_i}$ and $\Path'_{i}:=\Path|_{[t_{2i},t_{2i+1}]}$.
Take any probability space $(\Omega,\filt,\P)$ that houses the SHF $\flow^{\theta}_{s,t}$ for $s<t$, where $\filt$ is the sigma algebra generated by the full family $\flow^{\theta}_{s,t}$, $s<t\in\R$.
Fix any $\unit> 0$, and for a closed interval $I$ with endpoints in $\unit\D$, let $\filt_{I}$ be the sigma algebra generated by $\flow^{\theta}_{s,t}$ for all $s<t\in I$, and for disjoint closed intervals $I_1,\ldots,I_\ell$ with endpoints in $\unit\D$, let $\filt_{I_1\cup\cdots\cup I_\ell}$ be the sigma algebra generated by $\filt_{I_1},\ldots,\filt_{I_\ell}$.
\begin{prop}\label{p.prob.space}
Notation as above.
\begin{enumerate}
\item \label{p.prob.space.1}
The polymer measures $\PM^{\theta}_{[s,t]}$ for $s<t\in \unit\D$ can be constructed on $(\Omega,\filt,\P)$ with each $\PM^{\theta}_{[s,t]}$ being $\filt_{[s,t]}$-measurable.
\item \label{p.prob.space.2}
For $t_1<\cdots<t_{2\ell} \in \unit\D$, $I_{i}:=[t_{2i-1},t_{2i}]$, and $r_i:=t_{2i+1}-t_{2i}$\ ,
\begin{align}
	\label{e.martingale}
	\E\big[ 
		\PM^{\theta}_{[t_1,t_{2\ell}]} 
	\,\big|\,
		\filt_{I_1\cup\cdots\cup I_{\ell}}
	\big]
	=
	\PM^{\theta}_{I_1} \,\opdott_{\,r_1}\, \cdots \,\opdott_{\,r_{\ell-1}}\, \PM^{\theta}_{I_\ell}\ .
\end{align}
\item \label{p.prob.space.3}
Take the notation in \eqref{p.prob.space.2}, fix $t_1,t_{2\ell}$, and, for each $i=1,\ldots,\ell-1$, fix $t_{2i}$ or $t_{2i+1}$ and vary $r_i\in\unit\D_{>0}$ so that $t_{2i+1}:=t_{2i}+r_{i}$ or $t_{2i}:=t_{2i+1}-r_{i}$. 
We have
\begin{align}
    \label{e.p.prob.space.3}
	\bigvee_{r_1,\ldots,r_{\ell}\in\unit\D_{>0}} \filt_{I_1\cup\cdots\cup I_\ell}
    =
    \filt_{[t_1,t_{2\ell}]}\ .
\end{align}
\end{enumerate}
\end{prop}
\begin{ex}\label{ex.interval}
In Proposition~\ref{p.prob.space}\eqref{p.prob.space.3}, if we choose to fix $t_{2},t_{4},\ldots,t_{2\ell-2}$, then
\begin{align}
    I_1 = [t_1,t_2]\ ,
    &&
    I_2 = [t_2+r_1,t_4]\ ,
    &&
    I_3 = [t_4+r_2,t_6]\ ,
    &&
    \ldots,
    &&
    I_{\ell} = [t_{2\ell-2}+r_{\ell-1},t_{2\ell}]\ .
\end{align}
To reiterate, we fix the $t$s and vary the $r$s.
\end{ex}
\begin{rmk}\label{r.indep}
Observe that the random measures $\PM^\theta_{I_1},\ldots, \PM^\theta_{I_\ell}$ are independent for internally disjoint $I_1,\ldots,I_\ell$, in consequence of Proposition~\ref{p.prob.space}\eqref{p.prob.space.1} and Axiom~\eqref{d.shf.ind}.
\end{rmk}
\begin{rmk}\label{r.singular}
Let us point out a likely singular behavior of the polymer measure.
Call the time-$t_1$ marginals of $\flow^{\theta}_{t_0 ,t_1}$, $\flow^{\theta}_{t_1 ,t_2}$, and $\flow^{\theta}_{t_0 ,t_1} \opdot \flow^{\theta}_{t_1,t_2}$ respectively $\flow_1,\flow_2$, and $\flow_\textc\in\Msp(\R^2)$.
It was shown in \cite[Theorem~1.1]{caravenna2025singularity} that $\flow_1$ and $\flow_2$ are a.s.\ singular wrt the Lebesgue measure, and these two random measures are independent by Axiom~\eqref{d.shf.ind}.
These properties suggest that $\flow_1$ and $\flow_2$ likely assign positive masses to disjoint sets a.s.
Namely, the measures should at least partially ``miss'' each other, which, by \eqref{e.closedot}, would make $\flow_\textc$ and $\flow_1$ mutually singular and would therefore make $\PM^{\theta}_{[t_0,t_1]}$ and $\PM^{\theta}_{[t_0,t_2]}\circ\margin{[t_0,t_1]}^{-1}$ mutually singular (in the probability space in Proposition~\ref{p.prob.space}).
\end{rmk}

\subsection{Intersection local time}
\label{s.tools.inters}
Let us recall the construction of the intersection local time from \cite{clark2023planar}.
For $\Path,\Pathh\in\Psp_{[s,t]}$, consider 
\begin{align}
	\label{e.inters.e}
	\inters^{\Path,\Pathh}_{[s,t],\e}
	:=
	\d u \, \frac{1}{\e^{2}\,|\log \e|^2} \Phi\Big(\frac{\Path(u)-\Pathh(u)}{\e}\Big)
	\in
	\Msp[s,t]
	\ ,
	\quad
	\Phi(x) := \ind_{|x|\leq 1}\ ,
\end{align}
where $|\cdott|$ denotes the Euclidean norm on $\R^2$.
By \cite[Theorem~2.9]{clark2023planar}, there exists a measurable
$
	\inters_{[s,t]}=\inters^{\Path,\Pathh}_{[s,t]}: \Psp^2_{[s,t]} \to \Msp[s,t]
$
such that $\inters^{\Path,\Pathh}_{[s,t]}$ has no atoms and, for every $m<\infty$, as $\e\to 0$, 
\begin{align}
	\label{e.inters.L1conv}
	\int_{|\Path(s)|+|\Pathh(s)|\leq m} 
	\big( \E\PM^{\theta}_{[s,t]}{}^{\otimes 2}\big)(\d\Path, \d\Pathh)\,
	\sup_{u\in[s,t]}
	\big| \inters^{\Path,\Pathh}_{[s,t],\e}[s,u] - \inters^{\Path,\Pathh}_{[s,t]}[s,u] \big| 
	\longrightarrow
	0\ .
\end{align}
We call this $\inters^{\Path,\Pathh}_{[s,t]}$ the \textbf{intersection local time measure} for the pair of paths $\Path,\Pathh$.
It follows from Lemma~\ref{l.expansion} with $n=2$ that, for every $g,g'\in\Lsp^2(\R^4,\d x)$ and $\aa<\infty$,
\begin{align}
	\label{e.inters.e.finite}
	\lim_{\e\to 0}
	\int_{\Psp_{[s,t]}^2} 
	\big( \E\PM^{\theta}_{[s,t]}{}^{\otimes 2}\big)(\d\Path, \d\Pathh) \
	|g(\Path(s),\Pathh(t))\, g'(\Path(s),\Pathh(t))| 
	\,
	e^{\aa\inters^{\Path,\Pathh}_{[s,t],\e}[s,t]} 
	<
	\infty\ .
\end{align}
This together with \eqref{e.inters.L1conv} implies that
\begin{align}
	\label{e.inters.finite}
	\int_{\Psp_{[s,t]}^2} 
	\big( \E\PM^{\theta}_{[s,t]}{}^{\otimes 2}\big)(\d\Path, \d\Pathh) \
	|g(\Path(s),\Pathh(t))\, g'(\Path(s),\Pathh(t))| 
	\, 
	e^{\aa\inters^{\Path,\Pathh}_{[s,t]}[s,t]} 
	<
	\infty\ .
\end{align}

Let us introduce the intersection local time operator $\intersop$.
Since the bound \eqref{e.inters.finite} requires the ``localizing'' functions $g,g'$, we similarly localize the polymer measure by introducing a weight
\begin{align}
	\label{e.PMw}
	\PMw{}^{\theta}_{[s,t]}(\d\Path)
	:=
	\PM^{\theta}_{[s,t]}(\d\Path)\, \expfn(\Path(s)) \ ,
	\quad
	\expfn(x) := e^{-|x|}\ .
\end{align}
The choice of $\expfn$ is arbitrary and can be replaced by any strictly positive continuous function in $\Lsp^2(\R^2,\d x)$.
For Borel $A\subset[s,t]$, let $\intersop^{\theta,[s,t]}_{A}$ denote  the operator acting on $\Lsp^2(\Psp_{[s,t]},\PMw{}^{\theta}_{[s,t]})$ as
\begin{align}
	\label{e.intersop}
	\big( \intersop^{\theta,[s,t]}_{A} \psi \big)(\Path)
	:=
	\int_{\Psp_{[s,t]}} \PMw{}^{\theta}_{[s,t]}(\d\Pathh)\, 
	\inters^{\Path,\Pathh}_{[s,t]}(A)\,\psi(\Pathh)\ .
\end{align}
This operator is Hilbert--Schmidt by \eqref{e.inters.finite}, and we verify that it is positive in Lemma~\ref{l.positive}.
Call closed intervals \textbf{internally disjoint} if they intersect at most at their endpoints.
For any internally disjoint closed subintervals $I_1,\ldots,I_\ell$ of $[s,t]$, 
\begin{align}
	\label{e.intersop.additive}
	\intersop^{\theta,[s,t]}_{I_1\cup\cdots\cup I_\ell}
	=
	\intersop^{\theta,[s,t]}_{I_1} + \cdots + \intersop^{\theta,[s,t]}_{I_\ell}\ ,
\end{align}
since $\inters_{[s,t]}^{\Path,\Pathh}$ has no atoms.

\subsection{GMCs}
\label{s.tools.gmc}
Here, we make precise the GMC framework that we will be working with.
We will state the framework in a general setting and then specialize at the end of this subsection.

To set up the general setting, take a Polish space $\msp$, a reference measure $\measure\in\Msp(\msp)$, and a nonnegative symmetric function $\ker(\pt,\pt')\in\BMsp(\msp^2)$ as the correlation kernel.
For the purpose of generalizing the weighted polymer measure \eqref{e.PMw}, let us fix any strictly positive $\wfn\in\Csp(\msp)$ such that $\wfn$ and $1/\wfn$ are bounded on bounded subsets of $\msp$. Assume that the weighted measure 
$
	\measurew(\d\pt) := \measure(\d\pt)\, \wfn(\pt)
$
satisfies 
\begin{align}
	\label{e.gcm.assump}
	\int_{\msp^2} \measurew^{\otimes 2}(\d\pt,\d\pt')\, e^{ \aa' \ker(\pt,\pt')}  < \infty \quad \text{ for all } \aa' <\infty\ .
\end{align}
Assume further that the associated operator 
\begin{align}
	\kerop: \Lsp^2(\msp,\measurew) \to \Lsp^2(\msp,\measurew)\ ,
	\quad
	\big( \kerop \psi\big)(\pt) := \int_{\msp} \measurew(\d\pt')\, \ker(\pt,\pt')\, \psi(\pt')
\end{align}
is positive.
Note that it is Hilbert--Schmidt in consequence of \eqref{e.gcm.assump}. As stated before Theorem~\ref{t.main}, the randomness of the GMC will be built from iid Gaussian variables $\noise_1,\noise_2,\ldots$.
To streamline our subsequent notation, we adopt the equivalent (but more abstract) language of Gaussian Hilbert spaces to describe the $\noise_i$s.
Call a linear map $\noise:\lsp^2\to\Lsp^2(\noisesp,\filtt,\PP)$, for some sample space $\noisesp$ and sigma algebra $\filtt$, an \textbf{isotropic Gaussian vector with strength $\aa\geq 0$} when every $\noise v$ is a mean-zero Gaussian random variable and 
\begin{align}
	\EE[\noise v \cdot \noise v'] = \aa\, \ip{ v, v' }_{\lsp^2}\ ,
	\quad
	v,v'\in\lsp^2\ .
\end{align}
Write $e_i$ for the $i$th coordinate vector in $\lsp^2$.
Of course, $\noise_i:=\noise e_i$ for $i=1,2,\ldots$ provide the iid Gaussian variables stated before Theorem~\ref{t.main}.
We write $\PP,\EE$ for the law and expectation of $\noise$ to distinguish them from the law and expectation $\P,\E$ of the polymer measure, which will appear later when we specialize.

To build the GMC, following \cite{shamov2016gaussian} (but with a stronger assumption of being bounded in $\Lsp^2$), we consider any $\Yop$ that factorizes $\kerop$ as follows
\begin{align}
	\label{e.Yop}
	\Yop: \lsp^2 \to \Lsp^2(\Psp,\pmeasure)\ ,
	\text{ bounded linear, such that }
	\Yop\Yop^* = \kerop\ .
\end{align}
Such a $\Yop$ exists.
To see how, given a positive and Hilbert--Schmidt operator $\kerop$, enumerate with multiplicity its nonzero eigenvalues in decreasing order $\eigenval_1 \geq \eigenval_2 \geq \cdots>0$, and let $\eigenfn_1,\eigenfn_2,\ldots$ be corresponding orthonormal eigenvectors spanning $\nullsp \kerop\,{}^{\perp}$.  With $e_i$ denoting the $i$th coordinate vector of $\lsp^2$, the \textbf{standard Y operator of $\kerop$} defined by
\begin{align}
	\label{e.Yst}
	\Yst=\Yst(\kerop):\lsp^2 \to \Lsp^2(\Psp,\pmeasure)\ ,
	\quad
	\Yst\, e_{i} := \sqrt{\eigenval_{i}}\,\eigenfn_i
\end{align}
is readily seen to satisfy \eqref{e.Yop}.
By \cite[Corollary~5]{shamov2016gaussian}, the (primitive) GMC $\gmc[\measurew,\Yop,\noise]$ is the unique $\Msp(\msp)$-valued $\noise$-measurable function characterized by the conditions 
$
	\EE\,\gmc[\measurew,\Yop,\noise] = \measurew
$
and
$
	\gmc[\measurew,\Yop,\noise+\shift] = \gmc[\measurew,\Yop,\noise] \, e^{\Yop\shift}
$
for all $\shift\in\lsp^2$.
We call this GMC \emph{primitive} because we want the reference measure to be $\measure$ rather than $\measurew$.
The desired GMC is then defined as 
\begin{align}
    \gmcw[\measure,\Yop,\noise]:=\gmc[\measurew,\Yop,\noise]\, \wfn^{-1}=\gmc[\measure\,\wfn,\Yop,\noise]\, \wfn^{-1}\ .
\end{align}
The characterization of $\gmc$ translates into the following characterization of $\gmcw$ as a $\noise$-measurable function:
\begin{align}
	\label{e.shamov.1}
	\EE\,\gmcw[\measure,\Yop,\noise] &= \measure\ ,
\\
	\label{e.shamov.2}
	\gmcw[\measure,\Yop,\noise+\shift] &= \gmcw[\measure,\Yop,\noise] \, e^{\Yop\shift}\,
	\
	\text{for all } \shift \in \lsp^2\ .
\end{align}
Note that $\wfn$ does not appear explicitly in \eqref{e.shamov.1}--\eqref{e.shamov.2} but only implicitly through $\Yop$.
The weight is only a technical device that ensures a bound like \eqref{e.gcm.assump} and plays no other role.
Given \eqref{e.gcm.assump}, the GMC $\gmcw$ exists, is unique, and can be constructed from Kahane's martingale approximation with $\aa:=\EE(\noise e_1)^2$ as
\begin{align}
    \label{e.kahane.martingale}
    \gmcw[\measure,\Yop,\noise]
    =
    \lim_{\ell\to\infty} \measure \, \exp
    \sum_{i=1}^{\ell}
    \Big( \noise e_i\, \Yop e_i - \frac{\aa}{2} (\Yop e_i)^2\Big) \ .
\end{align}

The results from \cite{shamov2016gaussian} guarantee the uniqueness of $\gmcw[\measure,\Yop,\noise]$ for a fixed $\Yop$, but we will need to consider different $\Yop$s satisfying \eqref{e.Yop}.
Lemma~\ref{l.polar}\eqref{l.polar.2} below gives the needed consistency property.
Recall that a bounded linear $\isom$ between two (real) inner product spaces is a \textbf{partial isometry} if $\isom|_{\nullsp\isom\,{}^\perp}$ is an isometry.
\begin{lem}\label{l.polar}
Take any bounded $\Yop_j:\lsp^2\to\Lsp^2(\msp,\measure)$ with $\Yop_j\,\Yop_j^{*}=\kerop$ for $j=1,2$.
\begin{enumerate}
\item \label{l.polar.1}
The linear map defined by
\begin{align}
	\label{e.l.polar}
	\isom: \lsp^2 \longrightarrow \lsp^2 \ ,
	\qquad
	\begin{array}{l@{,}l}
		\Yop_1^{*}\eigenfn_i \longmapsto \Yop_2^{*}\eigenfn_i 
		\ & \ \ 
		i=1,2,\ldots
	\\
		\hspace{18pt}
		v \longmapsto 0  
		\ & \ \ 
		v \in \nullsp\Yop_1 = \range \Yop_1^{*}\,{}^{\perp}
	\end{array}	
\end{align}
is a partial isometry such that
\begin{align}
	\label{e.isom}
	\Yop_1 = \Yop_2 \isom\ ,
	\quad
	\nullsp \isom = \nullsp\Yop_1 = \range\Yop_1^*\,{}^\perp\ ,
\end{align}
and the properties in \eqref{e.isom} uniquely characterize the partial isometry $\isom$.
%
\item \label{l.polar.2}
The GMCs $\gmcw[\pmeasure,\Yop_1,\noise]$ and $\gmcw[\pmeasure,\Yop_2,\noise]$ have the same law.
\end{enumerate}
\end{lem}
\noindent{}%
Note that Lemma~\ref{l.polar}\eqref{l.polar.1} amounts to rephrasing a special case of the polar decomposition.
\begin{proof}
\eqref{l.polar.1}\ 
Let $\eigenfn_{i},\eigenval_{i}$ be the eigenfunctions and nonzero eigenvalues of $\kerop$ as above.
The assumption $\Yop_j\Yop_j^*=\kerop$ implies that $\Yop_j^{*}\eigenfn_1/\sqrt{\eigenval_1},\Yop_j^{*}\eigenfn_2/\sqrt{\eigenval_2},\ldots$ form an orthonormal basis for $\range\Yop_j^*$.
From this the conclusions of Part~\eqref{l.polar.1} are readily checked.

\eqref{l.polar.2}\ 
It follows from \eqref{e.isom} that  $\range \isom=\nullsp\Yop_{2}\,{}^\perp$. Note that $\isom \isom^*$ is the orthogonal projection onto $\nullsp\Yop_2\,{}^\perp$ and, in particular, $\Yop_2=\Yop_2 \isom \isom^*$.  The GMC characterization \eqref{e.shamov.1}--\eqref{e.shamov.2} implies that $\gmcw[\pmeasure,\Yop_2\isom,\noise\isom]$ is equal to $\gmcw[\pmeasure,\Yop_2,\noise]$.  Since $\Yop_1=\Yop_2\isom$ and the Gaussian vector $\noise\isom$ is equal in distribution to $\noise$, we conclude that the GMCs $\gmcw[\pmeasure,\Yop_1,\noise]$  and $\gmcw[\pmeasure,\Yop_2,\noise]$ are equal in distribution. 
\end{proof}


Next we recall Kahane's GMC moment formula. Let $\pair\intv{n}$ denote the collection of  unordered pairs of  $\intv{n}:=\{1,\ldots,n\}$.  We use  $\alpha=ij$ for $i<j\in\intv{n}$ to denote members of  $\pair\intv{n}$.
For $\vecpt:=(\pt_1,\ldots,\pt_n)\in\msp^n$, we write $\vecpt_\alpha := (\pt_{i},\pt_{j})$ for $\alpha=ij\in\pair\intv{n}$.
We have
\begin{align}
	\label{e.kahane.}
	\EE\,\gmcw[\measure,\Yop,\noise]^{\otimes n} 
	=
	\pmeasure^{\otimes n}(\d \vecpt)\, e^{\aa \sum_{\alpha\in\pair\intv{n}} \ker(\vecpt_{\alpha}) }
	\quad
	\text{in } \Msp(\msp^{n})\ ,
\end{align}
which follows by taking any nonnegative $f\in\Cloc(\msp)$ and applying \cite[Theorem~6]{kahane1985} to $\gmcw[\measure\cdot f,\Yop,\noise]$.

We now specialize with
\begin{align}
	\label{e.gmcw.specialization}
	\msp=:\Psp_{[s,t]}\ ,
	&&
	\measure=\PM=:\PM^{\theta}_{[s,t]}\ ,
	&&
	\ker=\inters(\cdott,\cdott)=:\inters^{\cdott,\cdott}_{[s,t]}[s,t] \ ,
	&&
	\kerop=\intersop=:\intersop^{\theta,[s,t]}_{[s,t]}\ .
\end{align}
Note that $\PM$ is random, and we take it to be independent of $\noise$.
Accordingly, $\kerop$ is random and $\PM$-measurable.
To simplify notation, we write the standard Y operator in the specialized case as
\begin{align}
	\label{e.Ystt}
	\Ystt{[s,t]}^{\theta}
	:=
        \Yst(\intersop)
        =
	\Yst\big(\intersop^{\theta,[s,t]}_{[s,t]}\big)\ .
\end{align}
More generally, we can take any bounded $\Yop:\lsp^2\to\Lsp^2(\Psp_{[s,t]},\PMw{}^{\theta}_{[s,t]})$ such that $\Yop\Yop^*=\intersop$, provided that the random $\Yop$ is $\PM$-measurable.
These specializations give the conditional GMC 
\begin{align}
    \label{e.gmcw}
    \gmcw=\gmcw[\PM,\Yop,\noise]=\gmcw[\PM^{\theta}_{[s,t]},\Yop,\noise]
\end{align}
over a fixed interval $[s,t]$.
The coupling of them over different intervals will be carried out in Section~\ref{s.coupling}.
Recall that $\P,\E$ denote the law and expectation of $\PM$.
Under the current specialization, taking the time-$\{s,t\}$ marginal of \eqref{e.kahane.} and taking $\E$ on both sides yields
\begin{align}
\label{e.kahane}
	\E\,\EE \big( \gmcw \circ\margin{\{s,t\}}^{-1}\big)^{\otimes n} 
	=
	\E\PM^{\otimes n}\cdot e^{\aa\sum_{\alpha\in\pair\intv{n}} \inters(\vecpath_{\alpha}) } \circ(\margin{\{s,t\}}^{\otimes n})^{-1} 
	\quad
	\text{in } \Msp\big(\R^{2\intv{n}}\times\R^{2\intv{n}}\big)\ ,
\end{align}
where $\PM$, $\inters$, and $\gmcw$ are as in \eqref{e.gmcw.specialization} and \eqref{e.gmcw}.  
This moment formula has a role in the next section.

\section{Moment matching}
\label{s.moment}
The goal of this section is to prove the following result, which verifies the required moment matching in Proposition~\ref{p.axioms}\eqref{p.axioms.mome}.
We do so by explicitly evaluating the right-hand side of \eqref{e.kahane} using the operator techniques developed in \cite{gu2021moments}. 
\begin{prop}\label{p.mome}
For $n\in\N$ and the conditional GMC $\gmcw$ in \eqref{e.gmcw} with $\EE(\noise e_1)^2=\aa\geq 0$,
\begin{align}
	\label{e.moment}
	\E\,\EE \big(\gmcw\circ \margin{\{s,t\}}^{-1}\big)^{\otimes n} 
	=
	\d x\, \d x' \ \sg^{\intv{n},\theta+\aa}(t-s,x,x')
	\quad
	\text{in } \Msp\big(\R^{2\intv{n}}\times\R^{2\intv{n}}\big)\ .
\end{align}
\end{prop}
\noindent{}%
In Section~\ref{s.moment.sg}, we recall the delta-Bose semigroup and its properties.
In Section~\ref{s.moment.annealed}, we define the annealed measures and prepare a few properties.  
Then, in Section~\ref{s.moment.}, we prove Proposition~\ref{p.mome}.

\subsection{Delta-Bose semigroup}
\label{s.moment.sg}
Let us recall the delta-Bose semigroup $\sg^{\intv{n},\theta}$ and its properties.
This semigroup is obtained in \cite{gu2021moments} based on \cite{rajeev99,dimock04}.
To simplify notation, we will often drop the $n$ dependence in scripts,  writing $\sg^{\intv{n},\theta}=\sg^{\theta}$, etc.
Recall the notation for pairs from before \eqref{e.kahane.}.
For $\alpha\in\pair\intv{n}$, the relevant operators map between functions on 
\begin{align}
	\label{e.xsp}
	\R^{2\intv{n}} 
	:=
	(\R^{2})^{\intv{n}}
	&:= 
	\big\{ (x_i)_{i\in\intv{n}} \, \big| \, x_i\in\R^2 \big\}
	\ ,
\\
	\label{e.ysp}
	\R^{2}\times\R^{2\intv{n}\setminus\alpha} 
	:= 
	\R^{2}\times\R^{2(\intv{n}\setminus\alpha)} 
	&:=
	\big\{ y=(\yc, (y_i)_{i\in\intv{n}\setminus\alpha}) \, \big| \, \yc,y_i\in\R^2 \big\}
	\ ,
\end{align}
where we index the first coordinate in \eqref{e.ysp} by ``c'' for ``center of mass''.
Consider
\begin{align}
	\label{e.Sop}
	\Sop_{\alpha}: \R^{2}\times\R^{2\intv{n}\setminus\alpha}  \to \R^{2\intv{n}},
	\qquad
	\big( \Sop_{\alpha}y \big)_{i}
	:=
	\begin{cases}
		\yc & \text{when } i\in\alpha \, ,
	\\
		y_i & \text{when } i\in\intv{n}\setminus\alpha\ .
	\end{cases}
\end{align}
Recall that $\hk(t,x)=\exp(-|x|^2/2t)/2\pi t$ is the heat kernel on $\R^2$, and let $\heatsg(t,x):=\prod_{i\in\intv{n}}\hk(t,x_i)$ denote the heat kernel on $\R^{2\intv{n}}$. We also define 
\begin{align}
	\label{e.jfn}
	\jfn^{\theta}(t)
	:=
	\int_{0}^{\infty} \d v \ \frac{t^{v-1}e^{\theta v}}{\Gamma(v)}\ .
\end{align}
For $\alpha\neq\alpha'\in\pair\intv{n}$, define the integral operators $\heatsg_{\alpha}(t)$, $\heatsg_{\alpha}(t)^*$, $\heatsg_{\alpha\alpha'}(t)$, $\Jop^{\theta}_{\alpha}(t)$ through their kernels as
\begin{subequations}
\label{e.ops}
\begin{align}
	\label{e.incoming}
	\heatsg_{\alpha}(t,y,x)
	&:=
	\heatsg\big(t, \Sop_{\alpha}y - x \big)
	=:
	\big(\heatsg_{\alpha}\big)^*(t,x,y)
	\ ,
\\
	\label{e.swapping}
	\heatsg_{\alpha\alpha'}(t,y,y')
	&:=
	\heatsg\big(t, \Sop_{\alpha}y - \Sop_{\alpha'}y' \big)\ ,
\\
	\label{e.Jop}
	\Jop^{\theta}_{\alpha}(t,y, y')
	&:=
	4\pi\,\jfn^{\theta}(t) \, \hk(\tfrac{t}{2},\yc-\yc') \, \prod_{i\in\intv{n}\setminus\alpha} \hk(t,y_i-y'_i)\ ,
\end{align}
\end{subequations}
where $x\in\R^{2\intv{n}}$ and $y,y'\in\R^{2}\times\R^{2\intv{n}\setminus\alpha}$ in \eqref{e.incoming} and \eqref{e.Jop}, and $y\in\R^{2}\times\R^{2\intv{n}\setminus\alpha}$ and $y'\in\R^{2\intv{n}\setminus\alpha'}$ in \eqref{e.swapping}.
Next, let
\begin{align}
	\label{e.dgm}
	\dgm\intv{n}
	&:=	
	\big\{ \vecalpha=(\alpha_k)_{k=1}^m\in\pair\intv{n}^m \, \big| \, m\in\N, \alpha_{k}\neq\alpha_{k+1}, k=1,\ldots, m-1 \big\}\ .
\end{align}
This set indexes certain diagrams, and hence the abbreviation $\dgm$; see \cite[Section~2]{gu2021moments}.
Write $|\vecalpha|:=m$ for the length of $\vecalpha\in\dgm$.
For a function $f=f(u,u',\ldots)$  depending on finitely many nonnegative $u$s, write $\int_{\Sigma(t)}\d \vecu\, f = \int_{u+u'+\cdots=t} \d \vecu\, f$ for the convolution-like integral.
For $\vecalpha\in\dgm\intv{n}$, define the operator
\begin{align}
	\label{e.sgsum}
	\sgsum^{\theta}_{\vecalpha}(t)
	&:=
	\int_{\Sigma(t)} \d \vecu \
	\heatsg_{\alpha_{1}}(u_{\frac{1}{2}})^* \,
	\prod_{k=1}^{|\vecalpha|-1} \Jop^{\theta}_{\alpha_{k}}(u_{k}) \, \heatsg_{\alpha_{k}\alpha_{k+1}}(u_{k+\frac{1}{2}}) \cdot
	\Jop^{\theta}_{\alpha_{|\vecalpha|}}(u_{|\vecalpha|}) \, \heatsg_{\alpha_{|\vecalpha|}}(u_{|\vecalpha|+\frac{1}{2}})\ .
\end{align}
Hereafter, products of operators are understood in the indexed order going from left to right, meaning that $\prod_{\ell=1}^{k}\Top_\ell := \Top_1 \Top_2 \cdots \Top_k$.
The delta-Bose semigroup on $\R^{2\intv{n}}$ is
\begin{align}
\label{e.sg}
	\sg^{\intv{n},\theta}(t)
	=
	\sg^{\theta}(t)
	:=
	\heatsg(t)
	+
	\sum_{\vecalpha\in\dgm\intv{n}} 
	\sgsum^{\theta}_{\vecalpha}(t)\ .
\end{align}

By \cite[Section~8]{gu2021moments}, the series \eqref{e.sg} converges under the $\Lsp^2$ operator norm.
To state the convergence, let us prepare some notation and tools.
For a bounded operator $\Top:\Lsp^2(\R^{d},\d x)\to\Lsp^2(\R^{d'},\d x)$, let $\normop{\Top}$ denote its operator norm.
Note that the kernels in \eqref{e.ops} are positive.
For operators $\Top(t)$  parameterized by $t>0$ with nonnegative kernels $\Top(t,z',z)$, $z\in\R^{d},z'\in\R^{d'}$, a bounded operator $ \int_{0}^b \d t \, \Top(t):\Lsp^2(\R^{d},\d x)\to\Lsp^2(\R^{d'},\d x)$ is defined by the kernel $\int_{0}^b \d t \, \Top(t,z',z)$ provided that
\begin{align}
	\sup_{\substack{ \norm{f}_2\leq 1 \\ \norm{f'}_2\leq1 } } \ \int_{0}^b \d t \int_{\R^{d'+d}} \d z' \d z \ |f'(z')| \, \Top(t,z',z) \, |f(z)| 
\end{align}
is finite, and the above is the operator norm of $ \int_{0}^b \d t \, \Top(t) $. Hereafter, $c=c(v,v',\ldots)$ denotes a deterministic, finite, nonnegative constant that may vary from place to place but depends only on the designated variables $v,v',\ldots$.
By \cite[Lemmas~8.1, 8.2, 8.5]{gu2021moments} (the second of which uses the proof of Lemma~3.1 in \cite{dellantonio1994hamiltonians}), for $c=c(\theta,n)$ and $\alpha\neq\alpha'\in\pair\intv{n}$,
\begin{subequations}
\label{e.opbds}
\begin{align}
	\label{e.opbd.incoming}
	\Normop{ \heatsg_{\alpha}(t) }
	&\leq
	c \, t^{-1/2}\ ,
	\quad
	\Normop{ \heatsg_{\alpha}(t)^* }
	\leq
	c \, t^{-1/2}\ ,
\\
	\label{e.opbd.swapping}
	\Normop{ \heatsg_{\alpha\alpha'}(t) }
	&\leq
	c \, t^{-1}\ ,
\\
	\label{e.opbd.swapping.int}
	\NOrmop{ \int_{0}^\infty \d t\, \heatsg_{\alpha\alpha'}(t) }
	&\leq
	c\ ,
\\
	\label{e.opbd.Jop}
	\Normop{ \Jop^{\theta}_{\alpha}(t) }
	&\leq
	c \, t^{-1} \big|\log\big(t\wedge \tfrac{1}{2}\big)\big|^{-2}\, e^{c\,t} \ .
\end{align}
\end{subequations}
These bounds are combined in \cite[Lemma~8.10]{gu2021moments} to show the convergence of \eqref{e.sg}.
The following is a slightly generalized version of that lemma, and we include a proof in Appendix~\ref{s.a.ops}.
\begin{lem}
\label{l.sum}
For $m\in\N$ and $i\in\intv{m}$, let $\Top_{i}(t):\Lsp^2(\R^{d_i},\d x) \to \Lsp^2(\R^{d_{i-1}},\d x)$ be a $t$-parameterized operator with a nonnegative kernel.
Assume that, for $c', c_1, \ldots, c_{m}\in[0,\infty)$, $b_i\in(0,1)$,  and for a partition $A_1,A_2,A_3$ of $\intv{m}$, the following bounds hold for all $t>0$.
\begin{subequations}
\label{e.l.sum.bd}
\begin{align}
	\label{e.l.sum.bd.1}
	&\text{When } i \in A_1,
	\quad
	\Normop{ \Top_{i}(t) }
	\leq
	c_{i} \, e^{c't} \, t^{-b_i}\ .
\\
	\label{e.l.sum.bd.2}
	&\text{When } i \in A_2,
	\quad
	\Normop{ \Top_{i}(t) }
	\leq
	c_{i} \, e^{c't} \,t^{-1} \, \big|\log(\tfrac{1}{2}\wedge t)\big|^{-2} \ .
\\
	\label{e.l.sum.bd.3}
	&\text{When } i \in A_3,
	\quad
	\Normop{ \Top_{i}(t) }
	\leq
	c_{i} \, e^{c't} \,t^{-1}
	\text{ and }
	\NOrmop{\int_{0}^\infty \d t\ e^{-c' t} \Top_{i}(t) } \leq c_i\ .
\end{align}
\end{subequations}
There exists a universal $c>0$ such that for all $m\in\N$, $\ell=0,\dots,|A_2|-1$, $t>0$, and $\lambda \geq c'+2$
\begin{align}
\label{e.l.sum}
\begin{split}
	\NOrmop{ \int_{\Sigma(t)} &\d \vecu \, \prod_{i\in\intv{m}} \Top_{i}(u_{i}) }
	\leq
	c^m\, m^{2} \, t^{|A_1|-1} \,  e^{\lambda t} \,
	\prod_{i\in\intv{m}} c_i 
	\cdot 
	\prod_{j\in A_1}\frac{t^{-b_j}}{1-b_j}
\\
	&\cdot
	\begin{cases}
	\displaystyle
	\frac{1}{ |\log(\lambda-c'-1)|^{|A_2|-1-\ell} \, |\log(t\wedge\frac{1}{2})|^{\ell+1} } & \text{ when } |A_2|>0\ ,
	\\
	\ 1 &\text{ when } |A_2|=0\ .
	\end{cases}
\end{split}
\end{align}
\end{lem}
\noindent{}%
Combining \eqref{e.opbds} with Lemma~\ref{l.sum} for $\ell=0$ and a large enough $\lambda$ shows that \eqref{e.sg} converges absolutely in operator norm. 

\subsection{Annealed polymer measures}
\label{s.moment.annealed}
Define the \textbf{$n$th annealed measure} as
\begin{align}
	\label{e.PMa}
	\PMa^{n,\theta}_{[s,t]}
	:=
	\E\PM^{\theta\ \otimes n}_{[s,t]}\in\Msp(\Psp_{[s,t]}^{n})\ .
\end{align}
The annealed measure has finite-dimensional marginals given by the delta-Bose semigroup, that is, for $s=t_0<\cdots<t_{\ell}=t$ we have
\begin{align}
	\label{e.PMa.marginals}
	\PMa^{n,\theta}_{[s,t]}\circ\margin{\{t_0,\ldots,t_\ell\}}^{-1}
	=
	\prod_{i=0}^{\ell-1} \d x^{i} \, \sg^{\intv{n},\theta}\big(t_{i+1}-t_{i}, x^{i}, x^{i+1}\big) \cdot \d x^{\ell}\ ,
\end{align}
Moreover, the annealed polymer measures have a concatenation consistency that can be expressed in terms of an operation $\opdott_{r}^{n}$ defined analogously to \eqref{e.opendot.path}  for measures $\pmeasure_I\in\Msp(\Psp_I^n)$, namely that
\begin{align}
	\label{e.PMa.ck}
	\PMa^{n,\theta}_{I_1} 
	\,\opdott_{r_1}^{n}\, 
	\cdots 
	\,\opdott_{r_{\ell-1}}^{n}\,
	\PMa^{n,\theta}_{I_{\ell}}
	\longrightarrow
	\PMa^{n,\theta}_{[t_1,t_{2\ell}]}
	\quad
	\text{vaguely as } r_1,\ldots,r_{\ell-1}\to 0\ ,
\end{align}
where $I_i:=[t_{2i-1},t_{2i}]$ and $r_i:=t_{2i+1}-t_{2i}$, as before. These properties are proven in \cite[Proposition~2.12]{clark2024continuum} for $n=1,2$, and the same proof generalizes to $n>2$.
When $n=1$, the delta-Bose semigroup is the heat semigroup $\sg^{\intv{1},\theta}(t)=\hk(t)$ (see \eqref{e.sg} and note that $\dgm\intv{1}=\emptyset$), so by \eqref{e.PMa.marginals} we have $\PMa^{1,\theta}_{[s,t]}=\wien_{[s,t]}$.
In general, the kernel of $\sg^{\intv{n},\theta}(t)$ is larger than the heat kernel (see \eqref{e.sgsum}--\eqref{e.sg}), and it follows from \eqref{e.PMa.marginals} that
\begin{align}
	\label{e.PMa.>=heat}
	\wien_{[s,t]}^{\otimes n} 
	\leq
	\PMa^{n,\theta}_{[s,t]}\ .
\end{align}
Next, by \cite[Equation~(2.20)]{tsai2024stochastic} we have
\begin{align}
    \label{e.annealed.centered}
    \E\big((\PM^{\theta}_{[s,t]}-\wien_{[s,t]}) \circ \margin{s,t}^{-1}\big)^{\otimes n}
    =
    \d x\, \d x'\,\sum_{\vecalpha\in\dgm_*\intv{n}} \sgsum^{\intv{n},\theta}_{\vecalpha}(t-s,x,x')\ ,
\end{align}
where $\dgm_*\intv{n}:=\{\vecalpha\in\dgm\intv{n}\,|\, \alpha_1\cup\cdots\cup\alpha_{\last}=\intv{n}\}$.
Note that the right-hand side of \eqref{e.annealed.centered} is a positive measure.
This fact together with \eqref{e.PMa.ck} shows that $\E(\PM^{\theta}_{[s,t]}-\wien_{[s,t]})^{\otimes n}$ is also a positive measure.

\subsection{Moment matching}
\label{s.moment.}
As the first step toward proving Proposition~\ref{p.mome}, we develop a series expansion for the right-hand side of \eqref{e.kahane}, and we achieve this through an $\e$ approximation.
Take $g,g'\in\Lsp^{2}(\R^{2\intv{n}},\d x)$, write $\inters_{\e}(\Path,\Pathh)=\inters^{\Path,\Pathh}_{[s,t],\e}[s,t]$ to simplify notation, and consider the $\e$ approximation of the right-hand side of \eqref{e.kahane} tested against $g\otimes g'$:
\begin{align}
	\label{e.p.expansion.}
	\int_{\Psp_{[s,t]}^n} \PMa_{[s,t]}^{n,\theta}(\d\vecpath)\ 
	e^{\aa\sum_{\alpha\in\pair\intv{n}} \inters_{\e}(\vecpath_{\alpha}) }\,
	g(\vecpath(s)) \, g'(\vecpath(t))\ .
\end{align}
In Lemma~\ref{l.expansion} and its proof below, we will develop a series expansion for \eqref{e.p.expansion.} and obtain the $\e\to 0$ limit of it.
Stating Lemma~\ref{l.expansion} requires some notation.
For $\vecalpha,\vecalpha^0,\ldots,\vecalpha^{m}\in\dgm\intv{n}$, let
\begin{align}
	\label{e.Cop}
	\Cop^{\theta}_{\vecalpha}(t)
	&:=
	\int_{\Sigma(t)} \d \vecu \
	\prod_{i=1}^{|\vecalpha|-1} \Jop^{\theta}_{\alpha_{i}}(u_{i}) \, \heatsg_{\alpha_{i}\alpha_{i+1}}(u_{i+\frac{1}{2}}) \cdot
	\Jop^{\theta}_{\alpha_{|\vecalpha|}}(u_{|\vecalpha|})\ ,
\\
	\label{e.Dop}
	\Dop^{\theta}(\vecalpha^0,\ldots,\vecalpha^{m},t)
	&:=
	\frac{1}{(4\pi)^m}
	\int_{\Sigma(t)} \d \vecu \ 
	\heatsg_{\alpha^{0}_{0}}(u_{-1})^* \,
	\prod_{\ell=0}^{m} \Cop^{\theta}_{\vecalpha^\ell}(u_\ell) 
	\cdot \heatsg_{\alpha^{m}_\last}(u_{m+1})\ ,
\end{align}
where we wrote $\alpha^m_\last:=\alpha^m_{|\vecalpha^m|}$ to simplify notation.
For $a\in\R$, consider the series
\begin{align}
	\label{e.Eop}
	\Eop^{\theta,\aa}(t)
	:=
	\sg^{\theta}(t)
	+
	\sum_{m=1}^\infty
	\aa^m
	\sum_{\conc}
	\Dop^{\theta}(\vecalpha^0,\ldots,\vecalpha^{m},t)\ ,
\end{align}
where $\sum_{\conc}$ runs over $\vecalpha^{\ell}=(\alpha^{\ell}_{1},\ldots,\alpha^{\ell}_{\last})\in\dgm\intv{n}$ for $\ell=0,\ldots,m$ under the \emph{concatenating condition} that
\begin{align}
	\label{e.conc.constraint}
	\alpha^{\ell}_{\last} = \alpha^{\ell+1}_{1}\ ,
	\qquad
	\ell = 0, \ldots, m-1\ .
\end{align}
The sum \eqref{e.Eop} converges absolutely in operator norm thanks to \eqref{e.opbds} and Lemma~\ref{l.sum} (for $\ell=0$ and a large enough $\lambda$).
Hereafter, we write $\ip{\cdott,\cdott}$ for the inner product on $\Lsp^2(\R^{d},\d x)$.

\begin{lem}\label{l.expansion}
For every $a\in\R$, $s<t\in\R$, $n\in\N$, and $g,g'\in\Lsp^2(\R^{2\intv{n}},\d x)$,
\begin{align}
	\label{e.p.expansion}
	\lim_{\e\to 0}
	\int_{\Psp_{[s,t]}^n} \PMa_{[s,t]}^{n,\theta}(\d\vecpath)\ 
	e^{\aa\sum_{\alpha\in\pair\intv{n}} \inters_{\e}(\vecpath_{\alpha}) }\,
	g(\vecpath(s)) g'(\vecpath(t))
	=
	\Ip{ g, \Eop^{\theta,\aa}(t-s) g' }\ ,
\end{align}
where $\Eop^{\theta,\aa}$ is the series defined in \eqref{e.Eop}.
In particular, together with \eqref{e.inters.e}, we have
\begin{align}
	\text{rhs of \eqref{e.kahane}}
	=
	\Ip{ g, \Eop^{\theta,\aa}(t-s) g' } < \infty\ .
\end{align}
\end{lem}
\begin{proof}
The first step is a pre-limiting expansion.
Since both sides of \eqref{e.p.expansion} depend only on $t-s$, without loss of generality we take $s=0$.
For $\sigma=ij\in\pair\intv{n}$ and $x\in\R^{2\intv{n}}$, set 
\begin{align}
	\label{e.expansion.Psi}
	\Psi^{\e}_{\sigma}(x):=\frac{1}{\e^2\,|\log\e|^2}\ind_{|x_i-x_j|\leq \e}\ ,
	\quad
	\Psi^{\e}:=\sum_{\sigma\in\pair\intv{n}}\Psi^{\e}_{\sigma}\ .
\end{align}
Recall $\inters_{\e}(\Path,\Pathh)=\inters^{\Path,\Pathh}_{\e,[s,t]}[s,t]$ from \eqref{e.inters.e}.
In \eqref{e.p.expansion.}, Taylor expand the exponential, recall $\inters_{\e}(\Path,\Pathh)=\inters^{\Path,\Pathh}_{\e,[s,t]}[s,t]$ from \eqref{e.inters.e}, and swap the sum and integrals (justified after \eqref{e.Eop.e}--\eqref{e.Dop.e}).
Doing so gives
\begin{align}
	\label{e.expmom.e.1}
	\eqref{e.p.expansion.}
	&=
	\sum_{m=0}^\infty 
	\frac{\aa^m}{m!}\int_{[0,t]^m} \d t_1\cdots \d t_m \  
	\int_{\Psp_{[0,t]}^n} \PMa^{n,\theta}_{[0,t]}(\d\vecpath)\  
	g(\vecpath(0)) \, \prod_{\ell=1}^m\Psi_\e(\vecpath(t_\ell))  \cdot  g'(\vecpath(t))\ .
\end{align}
Given that the integrand is symmetric in $t_1,\ldots,t_m$, write $\frac{1}{m!}\int_{[0,t]^m}$ as $\int_{0<t_1<\cdots<t_m<t}$.
Then, rename the variables $u_i:=t_{i}-t_{i-1}$ and use \eqref{e.PMa.marginals} to express the result in terms of the delta-Bose semigroup.
We obtain
\begin{align}
	\label{e.expmom.e.2}
	\eqref{e.p.expansion.}
	&=
	\ip{g, \Eop^{\theta,\aa}_{\e}(t) g' }\ ,
\\
	\label{e.Eop.e.}
	\Eop^{\theta,\aa}_{\e}(t)
	&:=
	\sg^{\theta}(t)
	+
	\sum_{m=1}^\infty \aa^{m} \int_{\Sigma(t)}\d\vecu\
	\sg^{\theta}(u_0) \,\Psi_{\e} \,\sg^{\theta}(u_1) \,\Psi_{\e} \cdots \sg^{\theta}(u_m)\ ,
\end{align}
where $\Psi_{\e}=\Psi_\e(x)$ acts as a multiplicative operator in $\Lsp^2(\R^{2\intv{n}},\d x)$.

The next step is to further expand \eqref{e.expmom.e.2}--\eqref{e.Eop.e.} to arrive at a pre-limiting version of \eqref{e.Eop}.
Use \eqref{e.sgsum}--\eqref{e.sg} and the definition of $\Cop^{\theta}_{\vecalpha}(t)$ in \eqref{e.Cop} to get
$
	\sgsum^{\theta}_{\vecalpha}(t)
	=
	\int_{\Sigma(t)}\d \vecu\, 
	\heatsg_{\alpha_1}(u)^* \, \Cop^{\theta}_{\vecalpha}(u') \, \heatsg_{\alpha_{\last}}(u'') 
$ 
and insert this into \eqref{e.sg} gives us
\begin{align}\label{e.Fop.pre}
	\sg^{\theta}(t)
	=
	\heatsg(t)
	+ \sum_{\vecalpha\in\dgm\intv{n}} 
	\int_{\Sigma(t)}\d \vecu\, 
	\heatsg_{\alpha_1}(u)^* \, \Cop^{\theta}_{\vecalpha}(u') \, \heatsg_{\alpha_{\last}}(u'')\ .
\end{align}
Let us adopt the conventions that $\heatsg_{\emptyset}(u):=\heatsg(u)$ and $\Cop^{\theta}_{\emptyset}(u):=\delta_0(u)$, and put
\begin{align}
	\label{e.Fop}
	\Fop^{\e}_{\alpha\sigma\alpha'}(t)
	:=
	\int_{u+u'=t} \d u\ \heatsg_{\alpha}(u)\,\Psi^{\e}_{\sigma}\,\heatsg_{\alpha'}(u')^*\,.
\end{align}
Using \eqref{e.Fop.pre} and the second relation in \eqref{e.expansion.Psi} to expand the summand in \eqref{e.Eop.e.} results in our desired pre-limiting version of \eqref{e.Eop}:
\begin{align}
\label{e.Eop.e}
	\Eop^{\theta,a}_{\e}(t)
	=
	\sg^{\theta}(t)
	+
	\sum_{m=1}^\infty \aa^{m} 
	\sum\nolimits' \Dop^{\theta}_{\e}(\vecalpha^0,\ldots,\vecalpha^{m}, \sigma^1,\ldots,\sigma^m,t)\ ,
\end{align}
where $\sum'$ runs over $\vecalpha^0,\ldots,\vecalpha^m\in\dgm\intv{n}\cup\emptyset$ and over $\sigma^1,\ldots,\sigma^m\in\pair\intv{n}$, with the convention that $\alpha_1=\alpha_{\last}=\emptyset$ when $\vecalpha=\emptyset$, and the summand is defined as
\begin{align}
\label{e.Dop.e}
\begin{split}
	\Dop^{\theta}_{\e}&(\vecalpha^0,\ldots,\vecalpha^{m}, \sigma^1,\ldots,\sigma^m,t)
	:=
	\int_{\Sigma(t)}\d\vecu\
	\heatsg_{\alpha^{0}_{1}}(u_{-1})^*	
\\
	&\cdot
	\prod_{\ell=0}^{m-1}
	\Cop^{\theta}_{\vecalpha^{\ell}}(u_m) \,
	\Fop^{\e}_{\alpha^{\ell}_{\mathrm{last}} \sigma^{\ell+1} \alpha^{\ell+1}_{1}}(u_{\ell+\frac{1}{2}})
	\cdot
	\Cop^{\theta}_{\vecalpha^m}(u_m) \, \heatsg_{\alpha^{m}_{\mathrm{last}}}(u_{m})\ .
\end{split}
\end{align}

Let us address the convergence and boundedness of \eqref{e.Eop.e}, but defer the proof of the following technical results for $\Fop^{\e}_{\alpha\sigma\alpha'}$ to Appendix~\ref{s.a.ops}.
For $\alpha,\alpha'\in\pair\intv{n}\cup\{\emptyset\}$ and $\sigma\in\pair\intv{n}$,
\begin{subequations}
\label{e.Fop.bds}
\begin{align}
	\label{e.Fop.bd}
	&
	\Normop{ \Fop^{\e}_{\alpha\sigma\alpha'}(t) }
	\leq
	c\, t^{-1}\, e^{ct}\, |\log\e|^{-1}\ ,
\\
	\label{e.Fop.bdint}
	&
	\NOrmop{ \int_{0}^\infty \d t \ e^{-t} \Fop^{\e}_{\alpha\sigma\alpha'}(t) }
	\leq
	c\,
	\begin{cases}
		1 	&\text{when } \alpha=\sigma=\alpha'\ ,
	\\
		|\log\e|^{-1}		& \text{otherwise}\ ,
	\end{cases}
\\
	\label{e.Fop.convg}
	&
	\text{for every } t'\in(0,\infty),
	\quad
	\int_{0}^{t'} \d t \, \Fop^{\e}_{\sigma\sigma\sigma}(t) \longrightarrow \frac{1}{4\pi}\,\id\ \text{ strongly as } \e \to 0\ .
\end{align}
\end{subequations}
Combining \eqref{e.Fop.bd}--\eqref{e.Fop.bdint} and Lemma~\ref{l.sum} with $\ell=0$ and a large enough $\lambda$ shows that \eqref{e.Dop.e} converges absolutely in operator norm, uniformly over $\e\in(0,1/2]$.

Now we send $\e\to 0$ in \eqref{e.Eop.e}.
Given the uniform (in $\e$) convergence from the previous paragraph, we take the limit term by term.
The terms $\Dop^{\theta}_{\e}$ are indexed by $\vecalpha^0,\ldots,\vecalpha^m\in\dgm\intv{n}\cup\emptyset$ and $\sigma^1,\ldots,\sigma^m\in\pair\intv{n}$.
We consider separately the indices that satisfy the condition
\begin{align}
	\label{e.conc.constraint.}
	\alpha^{\ell}_{\last} = \sigma^{\ell+1} = \alpha^{\ell+1}_1 \, \quad \text{for all} \ \  \ell=0,\ldots,m-1
\end{align}
and those that violate the condition.
For the latter, using \eqref{e.Fop.bd}--\eqref{e.Fop.bdint} and Lemma~\ref{l.sum} for $\ell=0$ and a large enough $\lambda$ shows that $\Dop^{\theta}_{\e}$ converges to $0$ in operator norm.
To finish the proof, we hence take any $\Dop^{\theta}_{\e}$ with indices satisfying \eqref{e.conc.constraint.}, and prove that it converges to the $\Dop^{\theta}$ defined in \eqref{e.Dop}.
Recall $\Cop^{\theta}_{\vecalpha}(t)$ from \eqref{e.Cop}.
For $\vecalpha\neq\emptyset$, bound its operator norm by applying \eqref{e.opbds} and Lemma~\ref{l.sum} with $m=|\vecalpha|$ and $\ell=1$.
Doing so gives, for $\vecalpha\neq\emptyset$ and $c=c(\theta,n,\vecalpha)$,
\begin{align}
	\label{e.l.expansion.Cbd}
	\normop{\Cop^{\theta}_{\vecalpha}(t)}
	\leq 
	c\, t^{-1}\, |\log(t\wedge\tfrac{1}{2})|^{-2} \, e^{ct}\ .
\end{align}
Now apply Lemma~\ref{l.delta} with $\{\Top_{i,\e}\}_{i\in I} = \{4\pi \Fop^{\e}_{\sigma^{1}\sigma^{1}\sigma^{1}},\ldots,4\pi\Fop^{\e}_{\sigma^{m}\sigma^{m}\sigma^{m}}\}$ and $\{\Top_{i}\}_{i\in\intv{m}\setminus I}$ being the remaining operators in \eqref{e.Dop.e}.
The assumptions of that lemma are satisfied thanks to \eqref{e.opbd.incoming}, \eqref{e.Fop.bds}, \eqref{e.l.expansion.Cbd}, and the fact that $\heatsg_{\alpha}(u)$ is strongly continuous away from $u=0,\infty$ (because $\heatsg_{\alpha}(u+s)=\heatsg_{\alpha}(u)\,\heatsg(s)$).
The result of applying Lemma~\ref{l.delta} gives the desired convergence of $\Dop^{\theta}_{\e}$ to $\Dop^{\theta}$.
\end{proof}

Let us now prove Proposition~\ref{p.mome}.
\begin{proof}[Proof of Proposition~\ref{p.mome}]
Given Lemma~\ref{l.expansion}, the desired result follows once we show that
\begin{align}
	\label{e.resum}
	\Eop^{\theta,\aa}(t)=\sg^{\theta+\aa}(t)\ ,
	\quad
	\aa\in\R \ , t\in(0,\infty)\ .
\end{align}
To prove this, we perform a resummation of \eqref{e.Eop}.
Recall that the $\vecalpha$s there are subject to the concatenating condition \eqref{e.conc.constraint}. Using concatenation, we define
\begin{align}
	\vecsigma 
	=
	(\sigma_1,\ldots,\sigma_{|\vecsigma|})
	:= 
	(\alpha^0_1,\ldots,\alpha^{0}_\last, \alpha^{1}_2, \ldots,\alpha^{1}_\last, \ldots, \alpha^{m}_{\last}) \in \dgm\intv{n}\ .
\end{align}
Let $n_j$ be the number of deletions that occur at the $j$th component of $\vecsigma$ during the concatenation procedure.
For example, if the $\vecalpha$s are $(12)$, $(12,23,12)$, $(12)$, $(12,34)$, concatenating them gives $\vecsigma=(12,23,12,34)$.
In the process of concatenation, we deleted one $12$ when merging $(12)$, $(12,23,12)$, we deleted two $12$s when merging $(12,23,12)$, $(12)$, $(12,34)$, while $23$ and $34$ remained intact.
These correspond to $n_1=1$, $n_3=2$, and $n_2=n_4=0$, respectively.
Note that $m=n_1+\cdots+n_{|\vecsigma|}$.
Insert \eqref{e.Cop}--\eqref{e.Dop} into \eqref{e.Eop} and use the above notation to rewrite the resulting sum as
\begin{align}
	\notag
	\Eop^{\theta,\aa}(t)
	=
	\sg^{\theta}(t)
	+
	\sum_{\vecsigma\in\dgm\intv{n}} 
	&\sum_{n_i\text{s}}
	\int_{\Sigma(t)}
	\d\vecu\
	\heatsg_{\sigma_1}(u_{\frac{1}{2}})^* 
\\
	\label{e.resum.1}
	&\cdot 
	\prod_{j=1}^{|\vecsigma|-1}
	\prod_{i=1}^{n_j}
	\frac{\aa}{4\pi}\Jop^\theta_{\sigma_j}(u_{ij}) 
	\cdot
	\Jop^\theta_{\sigma_j}(u_{n_{j}+1\, j})\,
	\heatsg_{\sigma_{j}\sigma_{j+1}}(u_{j+\frac{1}{2}})
\\
	\notag
	&\cdot
	\prod_{i=1}^{n_{|\vecsigma|}}
	\frac{\aa}{4\pi}\Jop^\theta_{\sigma_j}(u_{ij}) \cdot
	\Jop^\theta_{\sigma_j}(u_{n_{|\vecsigma|}+1\, |\vecsigma|})\,
	\heatsg_{\sigma_{|\vecsigma|}}(u_{|\vecsigma|+\frac{1}{2}})\ ,
\end{align}
where the second sum runs over $n_1,\ldots,n_{|\vecsigma|}\in\N\cup\{0\}$ under the constraint $n_1+\cdots+n_{|\vecsigma|}\geq 1$ (because $m\geq 1$ in \eqref{e.Eop}).
Observe that the summand in \eqref{e.sgsum}, after the $\vecalpha$ there being renamed to $\vecsigma$, is exactly the summand in \eqref{e.resum.1} with $n_1=\cdots=n_{|\vecsigma|}=0$.
Given this observation, combine \eqref{e.sgsum}--\eqref{e.sg}, rename the $\vecalpha$ there to $\vecsigma$, use the result to write the $\sg^{\theta}(t)$ on the right-hand side of \eqref{e.resum.1} as $\heatsg(t)$ pulse a sum over $\vecsigma\in\dgm\intv{n}$, and combine the latter sum with the sum in \eqref{e.resum.1}.
Doing so removes the constraints on the $n_i$s, so that they all run over $\N\cup\{0\}$ unconstrained.
Knowing that the series converges absolutely in operator norm (explained after \eqref{e.Eop}), we pass the sums over the $n_i$s into the integral, arriving at
\begin{align}
	\label{e.resum.2}
	\Eop^{\theta,\aa}(t)
	&=
	\heatsg(t)
	+
	\sum_{\vecsigma\in\dgm\intv{n}} 
	\int_{\Sigma(t)}
	\d\vecu\
	\heatsg_{\sigma_1}(u_{\frac{1}{2}})^* 
	\prod_{j=1}^{|\vecsigma|-1}
	\til{\Jop}^{\theta,\aa}_{\sigma_j}(u_{i})\,
	\heatsg_{\sigma_{j}\sigma_{j+1}}(u_{j+\frac{1}{2}})
\\
	\label{e.resum.3}
	&
	\hspace{.3\linewidth}
	\cdot
	\til{\Jop}^{\theta,\aa}_{\sigma_{|\vecsigma|}}(u_{|\vecsigma|})\,
	\heatsg_{\sigma_{|\vecsigma|}}(u_{|\vecsigma|+\frac{1}{2}})\ ,
	\text{ where }
\\
	\label{e.Jop.til}
	\til{\Jop}^{\theta,\aa}_{\sigma}(t)
	&:=
	\sum_{j=1}^\infty \Big( \frac{\aa}{4\pi} \Big)^{j-1} 
	\int_{\Sigma(t)} \d\vecu\ \prod_{i=1}^j\Jop^{\theta}_{\sigma}(u_i)\ .
\end{align}
Next, from \eqref{e.jfn} and using the known identity $\int_{\Sigma(t)}\d\vecu\ \prod_{i=1}^j u_i^{v_i-1}=\frac{\prod_{i=1}^j \Gamma(v_i)}{\Gamma(v_1+\cdots+v_j)}t^{v_1+\cdots+v_j-1}$, we have
$
	\int_{\Sigma(t)} \d\vecu\ \prod_{i=1}^j\jfn^{\theta}(u_i) 
	= 
	\int_{0}^\infty \frac{\d v}{\Gamma(v)} t^{v-1} e^{v\theta} \frac{v^{j-1}}{(j-1)!}
$.
Using this in \eqref{e.Jop.til} (and recalling $\Jop^\theta_\sigma$ from \eqref{e.Jop}) shows that $\til{\Jop}^{\theta,\aa}_{\sigma}(t)=\Jop^{\theta+\aa}_{\sigma}(t)$.
Inserting this equality into \eqref{e.resum.2}--\eqref{e.resum.3} and comparing the result with \eqref{e.sg} for $\theta\mapsto\theta+\aa$ show that the result is equal to $\sg^{\theta+\aa}(t)$.
\end{proof}

\section{Operator embedding}
\label{s.embedding}
The goal of this section is to prove Proposition~\ref{p.embedding}, or more precisely, its equivalent form Proposition~\ref{p.embedding.}, which will be a key input in Section~\ref{s.coupling}. Recall that $\unit\D$ denotes the set of dyadic multiples of some fixed $\unit\in(0,\infty)$. Given $s'\leq s<t\leq t'\in\unit\D$, we think of $[s,t]$ and $[s',t']$ respectively as the smaller and larger background intervals and consider
\begin{align}
	\label{e.embedding.notation0}
	&&
	\PMw_0&:=\PMw^{\theta}_{[s,t]}\ , 
	&
	\inters_0(\Path,\Pathh)&:=\inters^{\Path,\Pathh}_{[s,t]}[s,t]\ ,
	&
	\intersop_0&:=\intersop^{\theta,[s,t]}_{[s,t]}\ ,
	&
	\Psp_0&:=\Psp_{[s,t]}\ ,
	&&
\\
	\label{e.embedding.notation1}
	&&
	\PMw_1&:=\PMw^{\theta}_{[s',t']}\ ,
	&
	\inters_1(\Path,\Pathh)&:=\inters^{\Path,\Pathh}_{[s',t']}[s,t]\ ,
	&
	\intersop_1&:=\intersop^{\theta,[s',t']}_{[s,t]}\ ,
	&
	\Psp_1&:=\Psp_{[s',t']}\ .
	&&
\end{align}
In \eqref{e.embedding.notation0}--\eqref{e.embedding.notation1}, we measure the intersection local times over the same period $[s,t]$ but in different background intervals $[s,t]$ and $[s',t']$.
Roughly speaking, we seek to embed $\intersop_0$ from $\Lsp^2(\Psp_0,\PMw_0)$ to $\Lsp^2(\Psp_1,\PMw_1)$ and show that the result gives $\intersop_1$.
To state this precisely, define the embedding
\begin{align}
	\label{e.embed}
	\embed: \BMsp(\Psp_0^{n}) \to \BMsp(\Psp_1^{n})\ ,
	\qquad
	\big( \embed f \big)(\vecpath)
	:= 
	\big( \embed f \big)(\vecpath|_{[s,t]})\ ,
\end{align}
where $\vecpath=(\Path_1,\ldots,\Path_{n})\in\Psp_1^{n}$, and $\vecpath|_{[s,t]}:=(\Path_1|_{[s,t]},\ldots,\Path_{n}|_{[s,t]})$.
Recall that $\intersop_0$ is positive and Hilbert--Schmidt on $\Lsp^2(\Psp_{0},\PMw_{0})$ and that we enumerate with multiplicity its nonzero eigenvalues as $\eigenval_1 \geq \eigenval_2 \geq \cdots >0$ and denote the corresponding orthonormal eigenfunctions  by $\eigenfn_{1},\eigenfn_{2},\ldots\in \Lsp^2(\Psp_0,\PMw_0)$.
Write $\ipp{\cdott,\cdott}_0$ and $\ipp{\cdott,\cdott}_1$  for the respective inner products on $\Lsp^2(\Psp_0,\PMw_0)$ and $\Lsp^2(\Psp_1,\PMw_1)$. 
The spectral decomposition reads $\intersop_0 = \sum_{i} \eigenfn_{i} \, \eigenval_{i} \, \ipp{\eigenfn_{i},\cdott}_0$.
Embedding $\eigenfn_{i}\in\BMsp(\Psp_0)$ gives $\embed\eigenfn_{i}\in\BMsp(\Psp_1)$.
The goal of this section is to prove the following embedding result.
\begin{prop}\label{p.embedding}
Notation as above.
Almost surely, there exists an $\PMw_0$ version of the eigenfunction, denoted by $\eigenfn_{i}\in\BMsp(\Psp_0)$, such that $\embed\eigenfn_{i}\in\Lsp^2(\Psp_1,\PMw_1)$ and 
\begin{align}
	\sum_{i}\eigenval_{i}\,\ipp{\embed\eigenfn_{i}, f}^2_1
	=
	\PMw_1^{\otimes 2}\big[\inters_1\,f^{\otimes 2} \big]
	\quad
	\text{for all } f\in\Lsp^2(\Psp_1,\PMw_1),
\end{align}
or equivalently
$
	\sum_{i} \embed\eigenfn_{i} \, \eigenval_{i} \, \ipp{\embed\eigenfn_{i},\,\cdot\,}_1
	=
	\intersop_1 
$
in $\Lsp^2(\Psp_1,\PMw_1)$.
Further, the function $\eigenfn_{i}$ is $\filt_{[s,t]}$-measurably random.
\end{prop}

We will rephrase Proposition~\ref{p.embedding} in terms of marginals as Proposition~\ref{p.embedding.} below.
The latter is more convenient to prove and will also connect to a previous remark, which we explain after Proposition~\ref{p.embedding.}.
Let $\PMw_{0'}$ be the time-$[s,t]$ marginal of $\PMw_1$:
\begin{align}
	\label{e.pushforward}
	\PMw_{0'}(B)
	:=
	\PMw_{1}\circ\margin{[s,t]}^{-1} (B)
	:=
	\PMw_1\{ \til{\Path}\in\Psp_1 \,|\, \til{\Path}|_{[s,t]} \in B  \}\ ,
	\quad \text{for Borel }\,
	B\subset \Psp_{1} \ . 
\end{align}
Define $\Psp_{[s',s]}:=\Psp'$ and $\Psp_{[t,t']}:=\Psp''$, and view $ \Psp_1 $ as a subset of $\Psp'\times\Psp_0\times\Psp''$.
Accordingly, view $\PMw_{1}(\d\til{\Path})=\PMw_{1}(\d\Path',\d\Path,\d\Path'')$ as a measure on the product space, where $\Path\in\Psp_0$, $\Path'\in\Psp'$, and $\Path''\in\Psp''$, and consider the disintegration $\PMw{}_{0'|1}(\Path,\d \Path',\d\Path'')$ of $\PMw_1$ wrt $\PMw_{0'}$.
See \cite[Chapter~3]{kallenberg2021foundations} for a discussion of measure disintegrations, for example.
Given any $\psi\in \Lsp^1(\Psp^{n}_1,\PMw{}^{\otimes n}_1)$, extend it to a function on $(\Psp'\times\Psp_0\times\Psp'')^{n}$ by setting it to be $0$ outside of $\Psp_1^{n}$ and define
\begin{align}
	\label{e.version.disintegration}
	\big(\PMw_{0'|1}^{\otimes n}\, \psi\big)(\vec{\Path})
	:= 
	\prod_{i=1}^{n}
	\int_{\Psp'\times\Psp''}  \PM_{0'|1}(\Path_i,\d\Path'_i,\d\Path''_i) \cdot 
	\psi(\vec{\Path}',\vec{\Path},\vec{\Path}'')\ .
\end{align}
In analogy to the tower property for conditional expectations, the disintegration satisfies 
\begin{align}
	\label{e.tower.rule}
	\PMw_{1}^{\otimes n} \psi
	= 
	\PMw_{0'}^{\otimes n} \big[ \,\PMw_{0'|1}^{\otimes n}\, \psi \,\big] \ ,
\end{align}
(see \cite[Theorem~3.4(i)]{kallenberg2021foundations} for example).
It is not difficult to check that 
\begin{align}
	\label{e.version.marginal}
	\PMw_{1}^{\otimes 2} \big[ f \, \inters_1 \big]
	=	
	\PMw_{0'}^{\otimes 2} \big[ \big(\PMw_{0'|1}^{\otimes 2} f\big)\, \inters_0 \big]\ ,
	\quad
	\text{for all }
	f\in\Lsp^2(\Psp_1^2,\PMw{}_1^{\otimes 2})\ .
\end{align}
Proposition~\ref{p.embedding} is equivalent to the following.
\begin{customprop}{\ref*{p.embedding}'}\label{p.embedding.}
Notation as above.
Almost surely, there exists an $\PMw_0$ version of the eigenfunction, denoted by $\eigenfn_{i}\in\BMsp(\Psp_0)$, such that
\begin{enumerate}
\item \label{p.embedding..1}
$\eigenfn_{i}$ is $\filt_{[s,t]}$-measurably random and a.s.\ in $\Lsp^2(\Psp_0,\PMw_{0'})$, and
\item \label{p.embedding..2}
a.s.,
$
	\sum_{i} \eigenval_{i}\, 
	\PMw_{0'}[\embed\eigenfn_{i}\, f]\,{}^2
	=
	\PMw_{0'}^{\otimes 2}
	[\inters_0\, f^{\otimes 2}]
$ for all $f\in\Lsp^2(\Psp_0,\PMw_{0'})$.
\end{enumerate}
\end{customprop}

As was explained in Remark~\ref{r.singular}, we conjecture that $\PMw_0$ and $\PMw_{0'}$ are a.s.\ mutually singular (in the probability space in Proposition~\ref{p.prob.space}).
Proving Proposition~\ref{p.embedding.} therefore requires comparing functions between $\Lsp^2$ spaces where the underlying measures are potentially mutually singular, which constitutes the main difficulty in the proof.

We begin the proof of Proposition~\ref{p.embedding.} by preparing a useful bound.

\begin{lem}\label{l.version}
For every deterministic $g,g'\in\Lsp^2(\R^4,\d x)$ and $\aa<\infty$,
\begin{align}
	\label{e.version.bd}
	\int_{\Psp_0^2} 
	\E(\PM_{0} + \PM_{0'})^{\otimes 2}(\d\Path, \d\Pathh)
	\, |g(\Path(s),\Pathh(s))\,g'(\Path(t),\Pathh(t))| 
	\,
	e^{\aa\inters_0(\Path,\Pathh)} 
	<
	\infty\ .
\end{align}
\end{lem}
\begin{proof}
On the left-hand side of \eqref{e.version.bd}, expand the product of measures into four terms, and use \eqref{e.version.marginal} to convert $\PM_{0'}$ back to $\PM_{1}$ to get
\begin{align}
\label{e.l.version.2}
\begin{split}
    \text{lhs of \eqref{e.version.bd}}
    =
    \Big( &
        \int_{\Psp_0^2} \E\PM_{0}^{\otimes 2} 
        + \int_{\Psp_{1}^2} \E\PM_{1}^{\otimes 2} 
        + \int_{\Psp_{0}\times\Psp_{1}} \E\PM_{0}\otimes\PM_{1} 
        + \int_{\Psp_{1}\times\Psp_{0}} \E\PM_{1}\otimes\PM_{0} 
    \Big)
    \\
    &
    \cdot
    |g(\cdott|_{s})\, g'(\cdott|_{t})| 
    \,
    e^{\aa\inters_{0}}\ .
\end{split}
\end{align}
We have slightly abused notation in that $\inters_{0}$ should have been suitably embedded.
For example, in the second integral in \eqref{e.l.version.2}, $\inters_{0}$ should be interpreted as $\embed\inters_{0}$.
The first integral in \eqref{e.l.version.2} is finite because of \eqref{e.inters.e.finite}.
The second integral in \eqref{e.l.version.2} is finite because $\embed\inters_0\leq \inters_{[s',t']}[s',t']$ and because of \eqref{e.inters.e.finite}.
For the third integral, we apply Proposition~\ref{p.prob.space}\eqref{p.prob.space.2} with $\ell=3$, $I_2=[s,t]$, $r_1=s-s'$, $r_2=t'-t$, and the degenerate intervals $I_1=[s',s']$ and $I_3=[t',t']$.
Even though Proposition~\ref{p.prob.space} was stated for nondegenerate intervals, its proof generalizes to the current degenerate setting by interpreting $\PM^{\theta}_{[s',s']}=\d x$ and  $\PM^{\theta}_{[t',t']}=\d x$ and defining  $\d x\, \opdott_{\,s-s'} \PM_0\, \opdott_{\,t'-t}\, \d x'$ in a similar way to \eqref{e.opendot.path}.
We have
\begin{align}
    \E[\PM_{1}&|\filt_{[s,t]}]
    =
    \d x\, \opdott_{\,s-s'} \PM_0\, \opdott_{\,t'-t}\, \d x'\ \in \Msp(\Psp_{[s',t']})\ .
\end{align}
View $\Psp_{[s',t']}$ as a subset of $\Psp_{[s',s]}\times\Psp_0\times\Psp_{[t,t']}$ and integrate over the first and third components of the product space.
Doing so gives
\begin{align}
    \int_{\Psp_{[s',s]}\times\Psp_{[t,t']}} 
    \E[\PM_{1}&|\filt_{[s,t]}](\d\Path',\d\Path,\d\Path'')\, \ind_{\Path(s)}(\Path'(s))\, \ind_{\Path(t)}(\Path''(t))
    =
    \PM_0(\d\Path)\ .
\end{align}
Multiplying both sides by $\PM_0\,|g(\cdott|_{s})\,g'(\cdott|_{t})|\,e^{\aa\inters_{0}}$ and integrating over $\Psp_0^2$ show that the third integral in \eqref{e.l.version.2} is equal to the first.
The same holds for the last integral in \eqref{e.l.version.2}.
\end{proof}

We are now ready to prove Proposition~\ref{p.embedding.}\eqref{p.embedding..1}.
Let us state a simple tool in Assumption~\ref{a.embed.eigen} and Lemma~\ref{l.embed.eigen}, which we prove in Appendix~\ref{s.a.embed.tool} and apply in the  proof of Proposition~\ref{p.embedding.}\eqref{p.embedding..1}.
\begin{assu}\label{a.embed.eigen}
Take a Polish space $\msp$, measures $\measure,\measure'\in\Msp(\msp)$, and a symmetric $\ker\in\BMsp(\msp^2)$ such that
\begin{align}
	\label{e.l.embed.eigen.0}
	(\measure + \measure')^{\otimes 2} [\ker^2] 
	:=
	\int_{\msp^2} (\measure + \measure')^{\otimes 2}(\d\pt,\d\pt') \, \ker^2(\pt,\pt') 
	< \infty\ .
\end{align}
Let $\kerop$ and $\kerop'$ respectively be integral operators in $\Lsp^2(\msp,\measure)$ and $\Lsp^2(\msp,\measure')$ with the same kernel $\ker$.
Assume further that $\kerop,\kerop'$ are positive.
Note that they are Hilbert--Schmidt by \eqref{e.l.embed.eigen.0}.
\end{assu}
\begin{lem}\label{l.embed.eigen}
Under Assumption~\ref{a.embed.eigen}, for every enumerated positive eigenvalue $\eigenval_{i}$, the associated eigenfunction $\eigenfn_{i}\in \Lsp^2(\msp,\measure)$ has finite $\Lsp^2(\msp,\measure')$ norm, after possibly redefining it on a $\measure$-null set.
\end{lem}

\begin{rmk}\label{r.version}
By $\psi\in\Lsp^2(\msp,\measure)$, we mean that $\psi\in\BMsp(\msp)$ is an everywhere defined Borel function such that $\measure\,\psi^2<\infty$, and we \emph{do not} view $\psi$ as an equivalent class of $\measure$-a.e.\ equal functions.
\end{rmk}

\begin{proof}[Proof of Proposition~\ref{p.embedding.}\eqref{p.embedding..1}]
Thanks to Lemma~\ref{l.version}, the condition \eqref{e.l.embed.eigen.0} a.s.\ holds for $(\measure,\measure',\ker)=(\PMw_0,\PMw_{0'},\inters_{0})$.
Lemma~\ref{l.embed.eigen} applies and completes the proof.
Note that $\eigenval_{i}=\eigenval_{i}(\omega),\eigenfn_{i}=\eigenfn_{i}(\omega)$ are $\filt_{[s,t]}$-measurably random, because they can be chosen to depend only on $\PMw_0$ and $\inters_0$, the latter of which is deterministic.
\end{proof}

To prove Proposition~\ref{p.embedding.}\eqref{p.embedding..2}, we proceed by approximating $\PMw_{0'}$.
Hereafter, we work in the probability space in Proposition~\ref{p.prob.space}, with the same $\unit$ at the start of this section.
As explained previously, $\PMw_{0'}$ and $\PMw_{0}$ are potentially mutually singular.
To overcome this challenge, recalling the operation $\opdott_{r}$ from \eqref{e.opendot.path}, we take any $\unit\D\ni r_1>r_2>\cdots\to 0$ to define 
\begin{align}
	\label{e.PMw_1r}
	\PMw_{1,\ell}
	:=
	\PM^{\theta}_{[s',s-r_{\ell}]} 
	\,\opdott_{\,r_\ell} \,
	\PM^{\theta}_{[s,t]} 
	\,\opdott_{\,r_\ell} \,
	\PM^{\theta}_{[t-r_\ell,t']}\big[
		e^{-|\Path(s')|-|\Path(t')|} \cdott
	\big]\ ,
\end{align}
and let $\PMw_{0',\ell}$ be the time-$[s,t]$ marginal of $\PMw_{1,\ell}$ in the same sense as in \eqref{e.pushforward}.
For every fixed $\ell<\infty$, it is readily checked that $\PMw_{0',\ell}\ll \PMw_{0}$ with
\begin{align}
	\label{e.radon.nikodym}
	\d\PMw_{0',\ell}/\d\PMw_{0}(\Path)
	=
	\flow^{\theta}_{s',s-r}\, \expfn \otimes g_{r,\Path(s)} 
	\cdot
	e^{|\Path(s)|+|\Path(t)|}
	\cdot
	\flow^{\theta}_{t+r,t'}\, g'_{r,\Path(t)}\otimes\expfn\ ,
\end{align}
where $r:=r_\ell$, $g_{s,y}(x):=\hk(s,x-y)$, $g'_{t,y}(x):=\hk(t,y-x)$, and $\expfn(x):=e^{-|x|}$.

We now state, in a general setting, a set of conditions under which statements like Proposition~\ref{p.embedding.}\eqref{p.embedding..2} follow.
\begin{assu}\label{a.embed.tool}
Let $\msp,\measure,\measure',\ker$ be as in Assumption~\ref{a.embed.eigen}, let $\eigenfn_{i}$ be as in Lemma~\ref{l.embed.eigen}, and let $\eigenval_{i}>0$ be the eigenvalue.
Take $\measure_1,\measure_2,\ldots\in\Msp(\msp)$ such that $\measure_\ell\ll\measure$ and $\d\measure_{\ell}/\d\measure$ is bounded on bounded sets, and take a $D\subset\Cbsp(\msp)\cap\Lsp^2(\msp,\measure')$ that is dense in $\Lsp^2(\msp,\measure')$.
Consider the following conditions:
\begin{align}
	\label{e.embed.tool.1}
	&&&&&
	\measure_{\ell} [ \psi\eigenfn_{i} ]
	\xlongrightarrow{\ell}
	\measure' [\psi\eigenfn_{i}]\ ,
	&&
	\psi\in D\ , \ i=1,2,\ldots\ ,
\\
	\label{e.embed.tool.2}
	&&&&&
	\measure_{\ell}^{\otimes 2} [ \ker\, \psi^{\otimes 2} ]
	\xlongrightarrow{\ell}
	\measure'{}^{\otimes 2} [ \ker\, \psi^{\otimes 2}] \ ,
	&&
	\psi\in D\ ,
\\
	\label{e.embed.tool.3}
	&&&&&
	\measure_{\ell} [\eigenfn_{i}\eigenfn_{j} ]
	\xlongrightarrow{\ell}
	\measure' [\eigenfn_{i}\eigenfn_{j} ]\ ,
	&&
	i,j=1,2,\ldots\ ,
\\
	\label{e.embed.tool.4}
	&&&&&
	\measure_{\ell}^{\otimes 2} [ \ker^2 ]
	\xlongrightarrow{\ell}
	\measure'{}^{\otimes 2} [\ker^2]\ ,
\\
	\label{e.embed.tool.5}
	&&&&&
	\sup_{\ell}
	\sum_{i\geq 1, j \geq N}
	\eigenval_{i}\,\eigenval_{j}\,
	\measure_{\ell} [\eigenfn_{i}\eigenfn_{j}]\,{}^2
	\xlongrightarrow{N}
	0 \ .
\end{align}
\end{assu}

\begin{lem}\label{l.embed.tool}
Setup as in Assumption~\ref{a.embed.tool}.
If \eqref{e.embed.tool.1}--\eqref{e.embed.tool.5} hold, then
\begin{align}
	\sum_{i=1}^\infty \eigenval_{i}\,
	 \measure' [ f\eigenfn_{i} ] \,{}^2
	=
	\measure'{}^{\otimes 2}\big[ \ker\, f^{\otimes 2} \big]\ ,
	\quad
	f\in\Lsp^2(\msp,\measure')\ .
\end{align}
\end{lem}

To maintain the flow of the presentation, we place the proof of Lemma~\ref{l.embed.tool} in Appendix~\ref{s.a.embed.tool}.
Next we prove Proposition~\ref{p.embedding.}\eqref{p.embedding..2} using Lemma~\ref{l.embed.tool}.
\begin{proof}[Proof of Proposition~\ref{p.embedding.}\eqref{p.embedding..2}]
We begin with a few reductions.
Given Lemma~\ref{l.embed.tool}, the proof amounts to verifying Assumption~\ref{a.embed.tool}.
More precisely, we specialize $(\measure,\measure_{\ell},\measure')$ to $(\PMw_0,\PM_{0',\ell},\PMw_{0'})$ and
seek to verify that the conditions hold a.s.
Note that, by \eqref{e.radon.nikodym}, a.s.\ $\d\PM_{0',\ell}/\d\PM_{0}$ is bounded on bounded sets.
Note also that, by Corollary~\ref{c.dense}, there exists a \emph{deterministic}, countable $D\subset\Cbsp(\Psp_0)\cap\Lsp^2(\Psp_0,\PMw_{0'})$ that is a.s.\ dense in $\Lsp^2(\Psp_0,\PMw_{0'})$. 
This being the case, we fix any deterministic $\psi\in\Cbsp(\Psp_0)$, and when verifying \eqref{e.embed.tool.1}--\eqref{e.embed.tool.2}, we can and will only show that they hold a.s.\ for this one $\psi$.

To prove \eqref{e.embed.tool.1} and \eqref{e.embed.tool.3} for our fixed $\psi\in\Cbsp(\Psp_0)$, consider
\begin{align}
	\label{e.p.embedding..A}
	A^{i}
	&:= 
	\PMw_{0'}\big[ \psi\eigenfn_{i} \big]
	= 
	\PMw_{1}\big[ \embed(\psi\eigenfn_{i}) \big]\ ,
	&
	A^{i}_{\ell} 
	&:= 
	\PMw_{0',\ell}\big[ \psi\eigenfn_{i} \big]
	= 
	\PMw_{1,\ell}\big[ \embed(\psi\eigenfn_{i})\big]\ ,
\\
	\label{e.p.embedding..B}
	B^{ij}
	&:= 
	\PMw_{0'}\big[ \eigenfn_{i}\eigenfn_{j} \big]
	= 
	\PMw_{1}\big[ \embed(\eigenfn_{i}\eigenfn_{j})\big]\ ,
	&
	B^{ij}_{\ell} 
	&:= 
	\PMw_{0',\ell}\big[ \eigenfn_{i}\eigenfn_{j} \big]
	= 
	\PMw_{1,\ell}\big[ \embed(\eigenfn_{i}\eigenfn_{j}) \big]\ .
\end{align}
Verifying \eqref{e.embed.tool.1} and \eqref{e.embed.tool.3} amounts to showing that $A^{i}_{\ell}\to A^{i}$ and $B^{ij}_{\ell}\to B^{ij}$ a.s.\ as $\ell\to\infty$. 
Let 
\begin{align}
    \filt_{\ell} := \filt_{[s',s-r_\ell]\cup[s,t]\cup[t+r_\ell,t']}\ ,
\end{align}
where the right-hand side notation was defined before Proposition~\ref{p.prob.space}.
Recall from Part~\eqref{p.embedding..1} that the eigenfunctions $\eigenfn_{i}=\eigenfn_{i}(\omega)$ are $\filt_{[s,t]}$-measurable.
This being the case, applying Proposition~\ref{p.prob.space}\eqref{p.prob.space.2} to the last expressions in \eqref{e.p.embedding..A}--\eqref{e.p.embedding..B} gives that
\begin{align}
	\label{e.p.embedding..martingales}
	A^{i}_{\ell} = \E[ A^{i} \,|\, \filt_{\ell} ] \ ,
	\quad
	B^{ij}_{\ell} = \E[ B^{ij} \,|\, \filt_{\ell} ] \ .
\end{align}
This together with Proposition~\ref{p.prob.space}\eqref{p.prob.space.3} gives the desired convergence.

The arguments for verifying \eqref{e.embed.tool.2} and \eqref{e.embed.tool.4} are similar, and we do only the former.
To verify \eqref{e.embed.tool.2}, consider
\begin{align}
	\label{e.p.embedding..C}
	C
	&:=
	\PMw{}_{0'}^{\otimes 2}\big[ \inters_0\,\psi^{\otimes 2} \big]
	=
	\PMw{}_{1}^{\otimes 2}\big[ \embed(\inters_0\,\psi^{\otimes 2})	\big]
	\ ,
\\
	\label{e.p.embedding..Cl}
	C_\ell 
	&:=
	\PMw_{0',\ell}^{\otimes 2}\big[ \inters_0\,\psi^{\otimes 2} \big]
	=
	\PMw_{1,\ell}^{\otimes 2}\big[ \embed(\inters_0\,\psi^{\otimes 2}) \big]
	=
	\sum_{i\geq 1} \eigenval_{i} \, (A^{i}_{\ell})^2\ ,
\end{align}
where the last expression in \eqref{e.p.embedding..Cl} is obtained by applying the spectral decomposition of $\intersop_0$ on $\Lsp^2(\Psp_0,\PMw_0)$ to the first expression on the right-hand side (note that $\PMw_{0',\ell}\ll\PMw_{0}$).
Verifying \eqref{e.embed.tool.2} amounts to showing that $C_{\ell}\to C$ a.s.\ as $\ell\to\infty$.
Recall from the proof of Part~\eqref{p.embedding..1} that the eigenvalues $\eigenval_{i}=\eigenval_{i}(\omega)$ are $\filt_{[s,t]}$-measurable.
Using this and \eqref{e.p.embedding..martingales} in the last expressions in \eqref{e.p.embedding..Cl} shows that $C_\ell$ is a submartingale.
By \eqref{e.p.embedding..martingales} and Jensen's inequality, $C_\ell \leq \E[C|\filt_{\ell}]$, and by Proposition~\ref{p.prob.space}\eqref{p.prob.space.3}, $\E[C|\filt_{\ell}]\to C$.
These properties together with the facts that $C_\ell \geq 0$ and $\E C<\infty$ imply that $C_\ell$ converges a.s.\ and in $\E|\cdott|$ to a limit $C'$ with $C'\leq C$.
Hence, it suffices to show that $\E C_\ell\to \E C$ as $\ell\to\infty$.

To show that $\E C_\ell\to \E C$ (thereby completing the verification of \eqref{e.embed.tool.2}), it suffices to show that for $\PM_{1,\ell}$ defined the same way as \eqref{e.PMw_1r} but without the exponential weights,
\begin{align}
	\label{e.p.embedding..C.}	
	\E\PM_{1,\ell}^{\otimes 2}\leq \E\PM_{1}^{\otimes 2}\ ,
	\qquad
	\E\PM_{1,\ell}^{\otimes 2}\xrightarrow{\ell} \E\PM_{1}^{\otimes 2} \ \text{ vaguely} .
\end{align}
This implies the desired result because the first statement upgrades the vague convergence (tested against $\Cloc(\Psp_1^2)$ functions) to the convergence against any $f\in\Lsp^1(\Psp_1^2,\E\PM^{\otimes 2}_{1})$ and taking $f=\inters_0\,(\psi\,\wfn)^{\otimes 2}$ with $\wfn(\Path):=e^{-|\Path(s')|-|\Path(t')|}$ gives $\E C_\ell\to \E C$.
To prove \eqref{e.p.embedding..C.}, recall the annealed measures from \eqref{e.PMa}, let $s_\ell:=s-r_\ell$ and $t_\ell:=t+r_\ell$, and write
\begin{align}
	\label{e.p.embedding..C.1}	
	\E\PM_{1,\ell}^{\otimes 2}
	=
	\PMa^{2,\theta}_{[s',s_{\ell}]} 
	\, \opdott_{r_\ell}^{2} \,
	\PMa^{2,\theta}_{[s,t]}
	\, \opdott_{r_\ell}^{2} \, 
	\PMa^{2,\theta}_{[t_\ell,t']}\ .
\end{align}
Given this expression, the second statement in \eqref{e.p.embedding..C.} follows from \eqref{e.PMa.ck}.
Next, for small $r\in\unit\D_{>0}$, write the two $\opdott_{r_\ell}^{2}$s in \eqref{e.p.embedding..C.1} respectively as
\begin{align}
	\opdott_{r}^{2} \,
	\wien_{[s_{\ell}+r,s-r]}^{\otimes 2}
	\, \opdott_{r}^{2} 
	\quad \text{and} \quad
	\opdott_{r}^{2} \,
	\wien_{[t+r,t_{\ell}-r]}^{\otimes 2}
	\, \opdott_{r}^{2} \,	
	\PMa^{2,\theta}_{[t_\ell,t']}\ ,
\end{align}
apply the bound \eqref{e.PMa.>=heat} for $n=2$, and take $r\to 0$ using \eqref{e.PMa.ck}.
Doing so proves the first statement in \eqref{e.p.embedding..C.}.

To verify \eqref{e.embed.tool.5}, put $D_{N}:=\sup_{\ell} \sum_{i\geq 1, j \geq N} \eigenval_{i}\eigenval_{j} (B^{ij}_\ell)^2$.
Our goal is to show that $D_N\to 0$ a.s.
Since $D_N$ decreases in $N$ and is nonnegative, it suffices to show that $\E D_N\to 0$.
Bound
\begin{align}
	\label{e.p.embed..D0}
	\E \, D_{N} 
	\leq
	\sum_{i\geq 1, j \geq N}
	\E\Big[ \eigenval_{i}\eigenval_{j} \sup_{\ell} \, (B^{ij}_\ell)^2 \Big]\ .
\end{align}
Recall from \eqref{e.p.embedding..martingales} that $B^{ij}_\ell$ is a martingale in $\ell$ and recall that $\eigenval_{i}\eigenval_{j}$ is $\PMw_0$-measurable.
Hence, $\sqrt{\eigenval_{i}\eigenval_{j}} B^{ij}_\ell$ is a martingale in $\ell$.
Applying Doob's $\Lsp^2$ maximal inequality in \eqref{e.p.embed..D0} and using that $\E[\eigenval_{i}\eigenval_{j}(B_{\ell}^{ij})^2]\leq\E[\eigenval_{i}\eigenval_{j}(B^{ij})^2]$ (by \eqref{e.p.embedding..martingales} and Jensen's inequality over $\E[\cdott|\filt_{\ell}]$) gives that
\begin{align}
	\label{e.p.embed..D1}
	\E \, D_{N}
	\leq
	4
	\sum_{i\geq 1, j \geq N}
	\sup_{\ell} \E\big[ \eigenval_{i}\eigenval_{j} (B^{ij}_\ell)^2 \big]
	\leq
	4
	\sum_{i\geq 1, j \geq N}
	\E\big[ \eigenval_{i}\eigenval_{j} (B^{ij})^2 \big]	
	\ .
\end{align}
Next, let us check that the full sum
$
	\sum_{i,j\geq 1}
	\E[ \eigenval_{i}\eigenval_{j} (B^{ij})^2]	
$
is finite.
To this end, given that $B^{ij}_{\ell}\to B^{ij}$ a.s., using Fatou's lemma yields the bound 
\begin{align}
	\label{e.p.embed..D2}
	\sum_{i,j\geq 1}
	\E[ \eigenval_{i}\eigenval_{j} (B^{ij})^2]	
	\leq
	\liminf_{\ell}
	\sum_{i,j\geq 1}
	\E[ \eigenval_{i}\eigenval_{j} (B^{ij}_{\ell})^2]\ .
\end{align}
Use \eqref{e.p.embedding..C} to write the right-hand side as $\liminf_{\ell}\E C_\ell$.
This is equal to $\E C$, as we showed in the previous paragraph.
Hence, the left-hand side of \eqref{e.p.embed..D2} is finite, which implies that the right-hand side of \eqref{e.p.embed..D1} converges to $0$ as $N\to\infty$.
This completes the proof.
\end{proof}

\section{Coupling, Proof of Theorem~\ref{t.main}}
\label{s.coupling}
As the last step toward proving Theorem~\ref{t.main}, our main goal here is to couple the GMCs over different intervals together.
To state the coupling result, fix any $\unit\in(0,\infty)$, let
\begin{align}
	\label{e.dyadic.intvl}
	\I_m := \big\{ [s,t] \, \big| s<t\in \unit 2^{-m}\Z \big\}\ ,
	\quad
	\I := \I_1 \cup \I_2 \cup \cdots\ .
\end{align}
Recall the probability space $(\Omega,\filt,\P)$ from Proposition~\ref{p.prob.space}, which in particular houses the polymer measures $\PM^{\theta}_{[s,t]}$ for $s<t\in \unit\D$.
We also fix $\aa > 0$.
Throughout this section, we call an isotropic Gaussian vector with strength $\aa$ simply a \textbf{Gaussian vector}.
Consider an extended probability space $(\Omega',\filt',\P')$ equipped with Gaussian vectors 
\begin{align}
	\noise_{I}:\lsp^2\to\Lsp^2(\Omega',\filt',\P')\ ,
	\quad
	I\in\I\ .
\end{align}
Let $\filtt_I$ be the sigma algebra generated by $\noise_{I'}$ for all $I'\subset I\in\I$, and for disjoint (not just internally disjoint) $I_1,\ldots,I_\ell$, let $\filtt_{I_1\cup\cdots\cup I_\ell}$ be the sigma algebra generated by $\filtt_{I_1}\ldots,\filtt_{I_\ell}$. Recall that the sigma algebras $\filt_{I_1\cup\cdots\cup I_\ell}$ are defined before Proposition~\ref{p.prob.space} and $\opdott_{r}$ is as in \eqref{e.opendot.path}.
\begin{prop}\label{p.coupling}
Notation as above.
There exists an extended space $(\Omega',\filt',\P')$ such that (\ref{p.coupling.1})-(\ref{p.coupling.2}) below hold. 
\begin{enumerate}
\item \label{p.coupling.1}
For every internally disjoint $I_1,\ldots,I_\ell\in \I$, the sigma algebras $\filt,\filtt_{I_1},\ldots,\filtt_{I_{\ell}}$ are independent.
\item \label{p.coupling.2}
For $t_1<\cdots<t_{2\ell}\in\unit\D$, $I_i:=[t_{2i-1},t_{2i}]$, $r_i:=t_{2i+1}-t_{2i}$, and $\gmcw_{I}:=\gmcw[\PM^\theta_I,\Ystt{I}^{\theta},\noise_{I}]$, 
\begin{align}
	\label{e.coupling}
	\gmcw_{I_1} \,\opdott_{\,r_1}\, \cdots \,\opdott_{\,r_{\ell-1}}\, \gmcw_{I_{\ell}} 
	=
	\E' \big[ \, \gmcw_{[t_1,t_{2\ell}]} \, \big| \, \filt_{I_1\cup\cdots\cup I_\ell}, \filtt_{I_1\cup\cdots\cup I_\ell} \big] \ .
\end{align}
\end{enumerate}
\end{prop}

Key to proving Proposition~\ref{p.coupling} is to construct a coupling of the Gaussian vectors.
To illustrate the idea for this, we construct the coupling in a simpler setting in Section~\ref{s.coupling.pre}.
We then generalize the idea in Section~\ref{s.coupling.} to construct the full coupling and  prove Proposition~\ref{p.coupling}.
Equipped with Proposition~\ref{p.coupling} and results from previous sections, we prove Theorem~\ref{t.main} in Section~\ref{s.coupling.pfmain}.

\subsection{Primer for coupling}
\label{s.coupling.pre}
As said, to illustrate the idea of  the general treatment in Section~\ref{s.coupling.}, we begin with a simpler setting.
Instead of coupling $\noise_{J}$ for all $J\in\I$, taking any internally disjoint $I_1,\ldots,I_k\in\I$ whose union is an interval $I$ (which is in $\I$), we seek to couple $\noise_{I_1},\ldots,\noise_{I_k},\noise_{I}$.
Since $I_1,\ldots,I_k$ are internally disjoint, we just take $\noise_{I_1},\ldots,\noise_{I_k}$ to be independent.
To couple them to $\noise_{I}$, however, the naive way of taking $\noise_{I}$ as $\noise_{I_1}\oplus\cdots\oplus\noise_{I_{k}}$ does not work, as explained in Section~\ref{s.intro.singular}.
Below, we will devise an isometric embedding $\isom_{I_{1} \cdots I_{k}}:\lsp^2\to\oplus_{i=1}^k\lsp^2$ and use it to construct the suitable coupling in \eqref{e.noise.coupling}.

Let us rephrase Proposition~\ref{p.embedding} in the following equivalent form.
Write
\begin{align}
	\label{e.coupling.notation}
	\PMw_{I}=\PMw{}_I^{\theta}\ ,
	&&
	\intersop^{I',\theta}_{I}=\intersop^{I'}_{I}\ ,
	&&
	\Lsp^2_I=\Lsp^2(\Psp_{I},\PMw_{I})\ ,
	&&
	\Ystt{I}=\Ystt{I}^{\theta}: \lsp^2 \to \Lsp^2_I\ .
\end{align}
For $I\subset I'\in\I$, let $\embed:\BMsp(\Psp_I)\to \BMsp(\Psp_{I'})$ denote the embedding (see \eqref{e.embed}), where we have omitted the dependence of $\embed$ on $I,I'$ to simplify notation.
%
\begin{customprop}{\ref*{p.embedding}''}\label{p.embed..}
Given any $I\subset I'\in\I$, almost surely, there exist $\PMw_I$ versions of $\Ystt{I}e_{1}$, $\Ystt{I}e_{2}$, \ldots such that $\embed\Ystt{I}:\lsp^2\to\Lsp^2_{I'}$ is bounded and 
\begin{align}
	\embed\Ystt{I} \, \big( \embed\Ystt{I} \big)^*
	=
	\intersop^{I'}_{I}
	\quad \text{in } \Lsp^2_{I'}\ .
\end{align}
\end{customprop}

We now construct the desired isometric embedding $\isom_{I_1\ldots I_k}$ and coupling.
With $I_1,\ldots,I_k,I$ as above, apply Proposition~\ref{p.embed..} with $(I,I')\mapsto (I_{i},I)$ for $i=1,\ldots,k$ and sum the result over $i=1,\ldots,k$ using \eqref{e.intersop.additive} to get
\begin{align}
	\label{e.translation.Y.id}
	\embed\Ystt{I_{1}} \, \big( \embed\Ystt{I_{1}} \big)^*
	\, + \, \cdots \, + \,
	\embed\Ystt{I_{k}} \, \big( \embed\Ystt{I_{k}} \big)^*
	=
	\intersop^{I}_{I}
	=
	\Ystt{I} \, \Ystt{I}^*\  .
\end{align}
Let us factorize the left-hand side.
For linear $\op_i:\hilsp_i\to \hilsp$ between inner product spaces, put
\begin{align}
	\op_{1}\oplus'\cdots\oplus'\op_{k} \,: \ \hilsp_{1} \oplus \cdots \oplus \hilsp_{k} \longrightarrow \hilsp\ ,
	\qquad
	(v_1,\ldots,v_k)
	\longmapsto
	\op_1\,v_1 + \cdots + \op_k\,v_k\ .
\end{align}
Namely, $\oplus'$ means direct summing operators only in the domain.
We put
\begin{align}
	\label{e.Yop.sum}
	\Yop_{I_1 \ldots I_k} 
	:= 
	\embed \Ystt{I_{1}} \oplus' \cdots \oplus' \embed \Ystt{I_{k}}\ ,
	\quad
	\lsp^{2} \oplus \cdots \oplus \lsp^{2} \longrightarrow \Lsp^2_{I}\ .
\end{align}
It is readily checked that the left-hand side of \eqref{e.translation.Y.id} factorizes into
$\Yop_{I_1\ldots I_k}\, \Yop_{I_1 \ldots I_k}{}^*$.
Note also that $\nullsp\Ystt{I}=\{0\}$ (see \eqref{e.Yst}).
Given these properties, applying Lemma~\ref{l.polar}\eqref{l.polar.1} with the identification $\lsp^{2} \oplus \cdots \oplus \lsp^{2}\cong\lsp^2$ yields the isometric embedding
\begin{align}
	\label{e.isom.Y}
	\isom_{I_1 \ldots I_k}:
	\lsp^2
	\to
	\lsp^{2} \oplus \cdots \oplus \lsp^{2}\ ,
	\quad
	\text{such that }
	\Ystt{I} 
	=
	\Yop_{I_1 \ldots I_k }\,\isom_{I_1 \ldots I_k } \ ,
\end{align}
with range satisfying (see \eqref{e.l.polar})
\begin{align}
	\label{e.isom.embed}
	\range \isom_{I_1 \ldots I_k}
        =
        \range \Yop_{I_1\ldots I_k}^{*}
	\subset
	\nullsp\embed\Ystt{I_1}\,{}^\perp 
	\oplus\cdots\oplus
	\nullsp\embed\Ystt{I_k}\,{}^\perp\ .
\end{align}
Now, take $\noise_{I_1},\ldots,\noise_{I_k}$ to be independent Gaussian vectors, and couple them to $\noise_{I}$ by letting
\begin{align}
	\label{e.noise.coupling}
	\noise_{I}
	:=
	\big( \noise_{I_1}\oplus\cdots \oplus \noise_{I_k} \big)\, \isom_{I_1\ldots I_k}\ .
\end{align}

Two remarks are in place.
First, even though $\noise_{I}$ is constructed from $\isom_{I_1\ldots I_k}$, conditioned on the latter, the law of the former is always that of a Gaussian vector.
Therefore, $\noise_{I}$ is independent of $\isom_{I_1\ldots I_k}$ and hence $\filt$.
Next, to understand why \eqref{e.noise.coupling} gives the suitable coupling, note that there exist (at least) two choices for building the conditional GMC over $I$.
One is to use $\Ystt{I}$ as the operator and $\noise_{I}$ as the Gaussian vector.
Another is to use $\embed \Ystt{I_{1}} \oplus' \cdots \oplus' \embed \Ystt{I_{k}}$ and $\noise_{I_1}\oplus\cdots \oplus \noise_{I_k}$.
The coupling \eqref{e.noise.coupling} ensures that the two choices give the same result.

Behind the coupling \eqref{e.noise.coupling} lies the picture of a ``universal'' Hilbert space.
This picture will be used in Section~\ref{s.coupling.}, and we explain it now using \eqref{e.noise.coupling} as an example.
First, parameterize the independent Gaussian vectors $(\noise_{I_1},\ldots,\noise_{I_k})$ by
\begin{align}
	\hilsp_1:=\lsp^2(\{I_1,\ldots,I_k\}\times\N)\ ,
\end{align} 
where each $\lsp^2\cong\lsp^2(\{I_i\}\times\N)\hookrightarrow\hilsp_1 $ parametrizes $\noise_{I_i}$.
Similarly parameterize $\noise_{I}$ by $\hilsp_0:=\lsp^2(\{I\}\times\N)$.
A coupling between $(\noise_{I_1},\ldots,\noise_{I_k})$ and $\noise_{I}$ can be encoded by a way of isometrically embedding $\hilsp_0$ and $\hilsp_1$ into a common Hilbert space, which we call the \emph{universal} Hilbert space.
For example, isometrically embedding $\hilsp_1$ and $\hilsp_0$ into their orthogonal direct sum
\begin{align}
	\label{e.coupling.uhilsp.}
	\hilsp'
	:=
	\hilsp_1 \oplus \hilsp_0
\end{align}
corresponds to the trivial coupling of making $(\noise_{I_1},\ldots,\noise_{I_k})$ and $\noise_{I}$ independent.
To obtain the desired coupling \eqref{e.noise.coupling}, view \eqref{e.coupling.uhilsp.} as an algebraic direct sum (not orthogonal) and consider
\begin{align}
	\calN := \big\{ (v_1,v_0) \, \big| \, \isom_{I_1\ldots I_k} v_0 + v_1 =0 \big\} \subset \hilsp'\ ,
\end{align}
where $v_1\in\hilsp_1$, $v_0\in\hilsp_0$, and $\isom_{I_1\ldots I_k}:\hilsp_0\to\hilsp_1$ by identifying $\hilsp_0\cong\lsp^2$ and $\hilsp_1\cong \oplus_{i=1}^k\lsp^2$.
We equip the quotient space $\hilsp:=\hilsp'/\calN$ with the inner product
\begin{align}
	\Ip{ \overline{(v_1,v_0)}, \overline{(w_1,w_0)} }_{\hilsp}
	:=
	\Ip{ \isom_{I_1\ldots I_k} v_0 + v_1 ,  \isom_{I_1\ldots I_k} w_0 + w_1 }_{\hilsp_1}\ ,
\end{align}
where the bars denote the equivalence classes modulo $\calN$.
The resulting $\hilsp$ is the desired universal Hilbert space.
It is not difficult to check that the componentwise embeddings $\hilsp_i \hookrightarrow \hilsp_1\oplus\hilsp_0$ induce, through the quotient, isometric embeddings $\embedd_i:\hilsp_i\hookrightarrow \hilsp$, and these embeddings respect $\isom_{I_1\ldots I_k}$, namely $\embedd_0=\embedd_1\isom_{I_1\ldots I_k}$.
To summarize, the universal Hilbert space $\hilsp$ together with the embeddings $\embedd_0,\embedd_1,\isom_{I_1\ldots I_k}$ encodes the coupling \eqref{e.noise.coupling}.

\subsection{Coupling, proof of Proposition~\ref{p.coupling}}
\label{s.coupling.}
Here, we construct the coupling of $\noise_{I}$ for all $I\in\I$.
As illustrated in the last paragraph in Section~\ref{s.coupling.pre}, the coupling will be done by isometrically embedding $\lsp^2$ spaces into a universal Hilbert space, conditioned on a realization of the $\isom_{I_1\ldots I_k}$s.
The latter are $\filt$-measurable, so this can be understood as conditioned on $\filt$ too.

Let us prepare some notation and state the required embedding result.
Consider any finite set $\intvs=\{I_1,\ldots,I_k\}$ of internally disjoint $I_1,\ldots,I_k\in\I$, and let $\Intvs$ denote the collection of all such $\intvs$s. 
The relevant $\lsp^2$ spaces are
$\lsp^2(\intvs\times\N)$ for $\intvs\in\Intvs$.
When $\intvs\subset\intvs'$, we write 
\begin{align}\label{e.refine.pre}
	\embedd_{\intvs}^{\intvs'}:
	\lsp^2(\intvs\times\N)
	\hookrightarrow
	\lsp^2(\intvs'\times\N)	\ ,
	\quad
	\embedd_{\intvs}^{\intvs'\, *}:
	\lsp^2(\intvs'\times\N)
	\to
	\lsp^2(\intvs\times\N)	
\end{align}
for the coordinatewise embedding and ``projection''.
To simply notation, we will often use set and list notation interchangeably by writing $\isom_{I_1\ldots I_k}=\isom_{\{I_1,\ldots, I_k\}}$, $\embedd^{\intvs'}_{I}=\embedd^{\intvs'}_{\{I\}}$, etc.
We write $\intvs\preceq\intvs'$ if every interval $I\in \intvs$ is a union of internally disjoint intervals in $\intvs'$, and we use $\intvs'(I)$ to denote the collection of intervals in $\intvs'$ that are subsets of $I$.
Note that $\intvs\subset\intvs'$ implies $\intvs\preceq\intvs'$.
For $\intvs\preceq\intvs'$ and notation as above, define the isometric embedding
\begin{align}
	\label{e.isomu}
	\isomu^{\intvs'}_{\intvs}: \lsp^2(\intvs\times\N) \to \lsp^2(\intvs'\times\N)\ ,
	\quad
	\isomu^{\intvs'}_{\intvs} 
	:=
	\sum_{I\in\intvs} \embedd_{\intvs'(I)}^{\intvs'}\, \isom_{\intvs'(I)} \, \embedd_{I}^{\intvs\, *}\ .
\end{align}

\begin{prop}\label{p.hilsp}
Conditioned on a realization of the $\isom_{I_1\ldots I_k}$s (or on $\filt$), there exists a separable Hilbert space $\hilsp$ equipped with isometric embeddings 
\begin{align}
	\embedd_{\intvs}:\lsp^2(\intvs\times\N)\hookrightarrow\hilsp\ ,
	\quad
	\intvs\in\Intvs\ ,
\end{align}
such that for every $\intvs\preceq\intvs'\in\Intvs$
\begin{align}
	\label{e.hilsp.consistent}
	\embedd_{\intvs} = \embedd_{\intvs'}\, \isomu^{\intvs'}_{\intvs}\ .
\end{align}
\end{prop}
\noindent{}%
The proof of Proposition~\ref{p.hilsp} is similar to the last paragraph of Section~\ref{s.coupling.pre}, and we place it in Appendix~\ref{s.a.hilsp}.

We now construct the coupling of $\noise_{I}$, $I\in\I$, based on Proposition~\ref{p.hilsp}.
To reiterate, the coupling is done conditioned on the $\isom_{I_{1} \ldots I_{k}}$s (or $\filt$).
Recall that a Gaussian vector has domain $\lsp^2$.
Identify $\hilsp\cong\lsp^2$ and take a Gaussian vector
\begin{align}
	\noiseu:\hilsp \longrightarrow \Lsp^2(\noisesp,\filtt,\PP)\ .
\end{align}
For each $I\in\I$, identify $\lsp^2\cong\lsp^2(\{I\}\times\N)$ and define
\begin{align}
	\label{e.noise.I}
	\noise_I: \lsp^2(\{I\}\times\N) \longrightarrow \Lsp^2(\noisesp,\filtt,\PP)\ ,
	\qquad
	\noise_I := \noiseu\, \embedd_{I}\ .
\end{align}

Let us verify that these $\noise_{I}$s satisfy the analog of \eqref{e.noise.coupling}.
First, since $\isom_I$ and $\embedd_{\intvs}^{\intvs}$ are identity operators by definition,
\begin{align}
	\label{e.isomu.id1}
	&&\isomu^{\intvs'}_{\intvs} 
	&=
	\embedd^{\intvs'}_{\intvs}\ , &&\text{when } \intvs\subset\intvs'\in\Intvs\ ,
\\
	\label{e.isomu.id2}
	&&\isomu^{I_1\ldots I_k}_{I'} 
	&=
	\isom_{I_1\ldots I_k} \ , &&\text{when } \{I_1,\ldots, I_k\}\in\Intvs\ , \  I'=I_1\cup\cdots\cup I_k\in \I \ .
\end{align}
Next, take any $\{I_1,\ldots, I_k\}\in\Intvs$, combine \eqref{e.hilsp.consistent} and \eqref{e.isomu.id1} with $\intvs=\{I_i\}$ and $\intvs'=\{I_1,\ldots, I_k\}$, and sum the result over $i$ using $\oplus'$ .
In the result, identify $\embedd^{\intvs'}_{I_1}\oplus'\cdots \oplus'\embedd^{\intvs'}_{I_k}$ for $\intvs'=\{I_1,\ldots,I_k\}$ as the identity operator.
Doing so gives
\begin{align}
	\label{e.noise.indep}
	\embedd_{I_1\ldots I_k}
	=
	\embedd_{I_1}\oplus'\cdots\oplus'\embedd_{I_k}\ ,
	\quad
	\{I_1,\ldots, I_k\} \in \intvs\ .
\end{align}
Now, under the assumption of \eqref{e.isomu.id2}, right multiply both sides of \eqref{e.noise.indep} by $\isomu^{I_1\ldots I_k}_{I'}=\isom_{I_1\ldots I_k}$ and left multiply both sides by $\noiseu$ gives us that
\begin{align}\label{e.noise.coupling.pre}
\noiseu\, \embedd_{I_1\ldots I_k}\,\isomu^{I_1\ldots I_k}_{I'} = (\noiseu \,\embedd_{I_1})\oplus'\cdots\oplus' (\noiseu\embedd_{I_k}) \,\isom_{I_1\ldots I_k}\,.
\end{align}
Applying \eqref{e.hilsp.consistent} with $\intvs=\{I'\}$ and $\intvs'=\{I_1,\ldots, I_k\}$ on the left-hand side of \eqref{e.noise.coupling.pre} and then the definition of $\noise_{I}$ in \eqref{e.noise.I} results in the following analog of \eqref{e.noise.coupling}
\begin{align}
	\label{e.noise.coupling.}
	\noise_{I'} = \big(\noise_{I_1}\oplus'\cdots\oplus'\noise_{I_k} \big)\,\isom_{I_1\ldots I_k}\ ,
	\quad
	\{I_1,\ldots, I_k\}\in\Intvs\ , \  I'=I_1\cup\cdots\cup I_k\in \I \ .
\end{align}
Note that the direct sum is primed because the $\noise_{I}$s take values in the same space $\Lsp^2(\noisesp,\filtt,\PP)$.

We are now ready to prove Proposition~\ref{p.coupling}.
\begin{proof}[Proof of Proposition~\ref{p.coupling}]
First, Proposition~\ref{p.hilsp} and the above coupling are done conditioned on a realization of the $\isom_{I_1\ldots I_k}$s or $\filt$, so taking into account the randomness of the latter gives 
$(\Omega\times\noisesp,\filt\vee\filtt,\P\otimes\PP)$ as the desired extended probability space. 

Part~\eqref{p.coupling.1} follows by \eqref{e.noise.indep}.
Even though the $\embedd$s depend on $\filt$, they are always isometric embeddings.
Hence, conditioned on $\filt$, the law of $\noise_{I_1}\oplus'\cdots\oplus'\noise_{I_k}$ is that of a Gaussian vector (of the fixed strength $\aa$), which verifies Part~\eqref{p.coupling.1}.

To show Part~\eqref{p.coupling.2}, take $\ell=2$ to simplify notation, put $I_\textc:=[t_2,t_3]$ and $I:=[t_1,t_4]$, and use \eqref{e.Yop.sum}--\eqref{e.isom.Y} and \eqref{e.noise.coupling.} to write
\begin{align}
	\label{e.p.coupling.1}
	\gmcw_{I}
	&=
	\gmcw\big[ 
		\PM^{\theta}_{I}\, 
        ,(\embed\Ystt{I_1}^{\theta}\oplus'\embed\Ystt{I_\textc}^{\theta}\oplus'\embed\Ystt{I_2}^{\theta}) 
		\isom_{I_1\,I_\textc\,I_2} \, ,
		(\noise_{I_1}\oplus'\noise_{I_\textc}\oplus'\noise_{I_2}) \isom_{I_1\,I_\textc\,I_2}
	\big] .
\end{align}
By Part~\eqref{p.coupling.1}, $\noise_{I_1},\noise_{I_\textc},\noise_{I_2},\filt$ are independent.
Accordingly, write $\EE_1$, $\EE_\textc$, $\EE_2$, and $\E$ respectively for the expectations of the marginal laws of $\noise_{I_1},\noise_{I_\textc},\noise_{I_2}$, and $(\PM^{\theta}_{I'}|I'\in\I)$, and put $\EE:=\EE_1\otimes\EE_\textc\otimes\EE_2$.
This way, the right-hand side of \eqref{e.coupling} is $\E[\EE_\textc\gmcw_{I}|\filt_{I_1\cup I_2}]$.
To evaluate this, first note that the $\isom$s in \eqref{e.p.coupling.1} can be removed, which is not hard to check from the characterization \eqref{e.shamov.1}--\eqref{e.shamov.2} and \eqref{e.isom.embed}.
Then, use \eqref{e.kahane.martingale} to write  
\begin{align}
	\label{e.p.coupling.2}
	\gmcw_{I}
	=
	\lim_{k\to\infty} \PM^{\theta}_{I} 
	\, \exp
	\sum_{J=I_1,I_\textc, I_2} \sum_{i_J=1}^{k} 
	\Big(
		\noise_{J} e_{i}\, \Ystt{J}^{\theta} e_{i} 
		- \frac{\aa}{2} (\Ystt{J}^{\theta} e_{J i})^2 
	\Big)\ .
\end{align}
When tested against a function in $\Cloc(\Psp_I)$, the expression within the limit is bounded in $\E\EE(\cdott)^2$ thanks to the fact that it is a martingale and thanks to \eqref{e.inters.finite}.
Given this property, take $\E[\,\EE_\textc[\cdott]|\filt_{I_1\cup I_2}]$ on both sides of \eqref{e.p.coupling.2} and swap the expectations with the limit.
In the result, observe that taking $\EE_\textc$ amounts to replacing $\sum_{J=I_1,I_\textc, I_2}$ with $\sum_{J=I_1, I_2}$, and use Proposition~\ref{p.prob.space}\eqref{p.prob.space.2} to take $\E[\cdott|\filt_{I_1\cup I_2}]$.
Doing so verifies \eqref{e.coupling}.
\end{proof}

The proof of Theorem~\ref{t.main} requires one more result, which we state as Lemma~\ref{l.naimark}\eqref{l.naimark.2} below and prove in Appendix~\ref{s.a.hilsp}.
For every $I\in\I$, put $\hilsp_I:=\embedd_{I}\,\lsp^2(\{I\}\times\N)\subset \hilsp$ and let $\proj_I:\hilsp\to\hilsp$ denote the orthogonal projection of $\hilsp$ onto $\hilsp_I$.
Recall the notations $\intersop^{I'}_{I}$, $\Lsp^2_I$, $\Ystt{I}$ from \eqref{e.coupling.notation}. 
Using the identification $\lsp^2\cong\lsp^2(\{I\}\times\N)$ for the domain of $\embedd_{I}$, we define
\begin{align}
	\Yopu_{I}: \hilsp\longrightarrow\Lsp^2_I\ ,
	\quad
	\Yopu_{I} := \Ystt{I}\, \embedd_{I}^{\,*}\ .
\end{align}
\begin{lem}\label{l.naimark}
\begin{enumerate}
\item[]
\item 
\label{l.naimark.1}
For every $I\subset I'\in\I$, we have $\intersop^{I'}_{I}=\Yopu_{I'}\,\proj_{I}\,\Yopu^{*}_{I'}$\ .
\item
\label{l.naimark.2}
For any decreasing sequence $\{I_k\}_{k}\subset\I$ that converges to a point,
\begin{align}
    \P\text{-a.s.,}
    \quad
    \proj_{I_k} \longrightarrow 0
    \quad
    \text{strongly.} 
\end{align}
\end{enumerate}
\end{lem}

\begin{rmk}\label{r.naimark}
We make some remarks related to dilation theory of operators, which we will not use but has its own independent interest.
\begin{enumerate}
\item \label{r.naimark.1}
Let us mention a variant of the projection $\proj_I$.
For $I\in\I$ and large enough $m$, partition $I$ into internally disjoint intervals of length $\unit 2^{-m}$, and let $\intvs_m(I)$ denote the collection of the latter intervals. 
With $\proj_{\intvs_m(I)}:\hilsp\to\hilsp$ denoting the orthogonal projection onto $\embedd_{\intvs_m(I)}\,\lsp^2(\intvs_m(I)\times\N)$, we define
\begin{align}
    \hat{\proj}_{I}
    :=
    \lim_{m\to\infty} \proj_{\intvs_m(I)}\ ,
\end{align}
where the limit exists strongly, as can be verified. The statements that follow can also be readily checked. The family of projections  $\{ \hat{\proj}_I \}_{I\in \I}$ is commuting and, in fact, satisfies the spectral property: $\hat{\proj}_{I\cap I'}=\hat{\proj}_{I}\hat{\proj}_{I'}$.  It is  countably additive in the sense that if $I_1, I_2,\ldots \in \I $ are internally disjoint and have union $\bigcup_{i=1}^{\infty}I_j=I\in \I$, then $ \proj_I=\sum_{i=1}^{\infty} \proj_{I_j}$, where the series converges strongly.  Using these properties, the map $I\to  \hat{\proj}_I $ can be extended to a Projection Valued Measure (PVM), defining $\hat{\proj}_J$  on all Borel sets $J\subset \R$, as in the conclusion of Naimark's dilation theorem; note that it is not surprising that the conclusion of Naimark's theorem holds here considering that we started with a non-unital Positive Operator Valued Measure (POVM) in \eqref{e.intersop.additive}.

\item \label{r.naimark.2} Consider the family $\J$ that consists of finite unions of intervals in $\I$. The above definitions of $\hilsp_I$, $\hat{\proj}_I$, $\Yopu_{I}$ for $I\in\I$ generalize straightforwardly to that for $J\in\J$.  The proof of Lemma~\ref{l.naimark} in Appendix~\ref{s.a.hilsp} gives, mutatis mutandis, a more general result in which $\proj_I$ is replaced by $\hat{\proj}_J$ with $J\in\J$ and $J\subset I'$, and $\{I_k\}_{k}$ is replaced by a decreasing sequence $\{J_k\}_{k}\subset\J$ that $J_k^\circ$ converges to the empty set.
\end{enumerate}
\end{rmk}

\subsection{Proof of Theorem~\ref{t.main}}
\label{s.coupling.pfmain}
We begin with a few reductions.
Indeed, the desired result follows from applying Proposition~\ref{p.axioms} to the probability space in Proposition~\ref{p.coupling} with the family $ (\PMm_I)_{I\in \I} $ given by $\PMm_I=\gmcw_I:=\gmcw[\PM^\theta_I,\Ystt{I}^{\theta},\noise_{I}]$. Given $t_0<\cdots <t_\ell\in \unit\D$, define  $I_{j}=[t_{j-1},t_j]$.
By Remark~\ref{r.indep} and Proposition~\ref{p.coupling}\eqref{p.coupling.1}, the Gaussian vectors $\noise_{I_1},\ldots, \noise_{I_\ell}$ and random measures $\PM^\theta_{I_1},\ldots, \PM^\theta_{I_\ell}$ are jointly independent.
Thus, the random measures $\gmcw_{I_1},\ldots,\gmcw_{I_\ell}$ are independent, giving us the condition in Proposition~\ref{p.axioms}\eqref{p.axioms.ind}. 
Next, the condition in Proposition~\ref{p.coupling}\eqref{p.axioms.mome} is satisfied by Proposition~\ref{p.mome}.
To verify the condition in Proposition~\ref{p.coupling}\eqref{p.axioms.ind}, recall from \eqref{e.opendot} that we define
\begin{align}
    \floww_{t_0,t_1}\opdot\cdots\opdot\floww_{t_{\ell-1},t_{\ell}}
    :=
    \lim_{r\to 0} 
    \floww_{t_0,t_1-r}\opdot_r\cdots\opdot_r\floww_{t_{\ell-1},t_{\ell}}\ ,
\end{align}
where the $\opdot_r$ is the left-indented product defined in \eqref{e.opendotl}.
Note that the random measure within the limit is the time-$\{t_0,t_1,\ldots,t_\ell\}$ marginal of
$\gmcw_{[t_0,t_1-r]}\, \opdott_{r}\, \gmcw_{[t_1,t_2-r]} \cdots \opdott_{r}\, \gmcw_{[t_{\ell-1},t_{\ell}]}$, so the desired condition follows once we prove that
\begin{align}
    \gmcw_{[t_0,t_1-r]}\, \opdott_{r}\, \gmcw_{[t_1,t_2-r]} \cdots \opdott_{r}\, \gmcw_{[t_{\ell-1},t_{\ell}]}
    \longrightarrow
    \gmcw_{[t_0,t_\ell]}\ ,
    \quad
    \text{vaguely in probability as } r \to 0\ .
\end{align}
By \eqref{e.coupling}, the left-hand side is equal to
\begin{align}
    \label{e.pfmain.1}
    \E'\big[ \,\gmcw_{[t_0,t_{\ell}]} \, \big| \, \filt_{[t_0,t_1-r]\cup\cdots\cup[t_{\ell-1},t_\ell]}, \filtt_{[t_0,t_1-r]\cup\cdots\cup[t_{\ell-1},t_\ell]} \big]\ .
\end{align}
Hence, it suffices to show that \eqref{e.pfmain.1} converges to $\gmcw_{[t_0,t_\ell]}$ vaguely in probability as $r\to 0$.

We now prove the last stated convergence, considering the case $\ell=2$ to avoid unwieldy notation. 
Given Proposition~\ref{p.prob.space}\eqref{p.prob.space.3}, we achieve this by showing that
\begin{align}
	\label{e.pfmain.goal.}
	  \filt_{[t_0,t_2]} \vee \bigvee_{r'\in\unit\D_{>0}} \filtt_{[t_0,t_1-r']\cup[t_1,t_2]}
	\supset
	\filtt_{[t_0,t_2]}\ ,
\end{align}
which amounts to proving that, for every $s<t\in[t_0,t_2]$, $\noise_{[s,t]}$ is measurable wrt to the left-hand side of \eqref{e.pfmain.goal.}.
This statement follows immediately when $s,t<t_1$ or $s,t\geq t_1$, so we consider only $s< t_1\leq t$.
Partition $[s,t]$ into $J_1(r):=[s,t_1-r]$, $J_\textc(r):=[t_1-r,t_1]$, and $J_2:=[t_1,t]$, and consider
$
	\noise(r) := \noise_{[s,t]} - \noiseu\, \proj_{[t_1-r,t_1]} \, \embedd_{[s,t]}
$.
It is straightforward to check that $\noise(r)$ is measurable wrt to the left-hand side of \eqref{e.pfmain.goal.}.
Further, $\EE((\noise_{[s,t]}-\noise(r))e_i)^2=\norm{\proj_{[t_1-r,t_1]} \, \embedd_{[s,t]}\, e_i}^2$, which by Lemma~\ref{l.naimark}\eqref{l.naimark.2} tends to $0$ $\P$-a.s.
Hence, $\noise_{[s,t]}$ is measurable wrt to the left-hand side of \eqref{e.pfmain.goal.}.

\section{Applications, proof of Corollaries~\ref{c.positivity}--\ref{c.superc}}
\label{s.apps}
Now we turn to the proof of Corollaries~\ref{c.positivity}--\ref{c.superc} based on Theorem~\ref{t.main}.
We prepare a few tools in Section~\ref{s.apps.tool} and prove Corollaries~\ref{c.positivity}--\ref{c.superc} in Section~\ref{s.apps.}.
The proof requires a few results about the moments, which we state in Section~\ref{s.apps.} and prove in Section~\ref{s.apps.moment}.
Throughout this section, we will mostly drop the dependence on $[s,t]$ and write 
\begin{align}
	\PM^{\theta}:=\PM^{\theta}_{[s,t]}\ ,
	&&
	\Psp:=\Psp_{[s,t]}\ ,
	&&
	\Yst^{\theta}:=\Ystt{[s,t]}^{\theta}\ ,
	&&
	\inters(\cdott,\cdott):= \inters^{\cdott,\cdott}_{[s,t]}[s,t]\ ,
	&&
	\intersop^{\theta}:=\intersop^{\theta,[s,t]}_{[s,t]}\ .
\end{align}

\subsection{Comparison principle, GMC coupling}
\label{s.apps.tool}

The first ingredient of the proof is a comparison principle that follows from a Gaussian zero-one law.
\begin{lem}\label{l.01}
Let $\msp$, $\measure$, $\Yop$, $\noise$, and $\gmcw[\pmeasure,\Yop,\noise]$ be as in Section~\ref{s.tools.gmc}.
For every $0\leq f\in\BMsp(\msp^n)$, 
\begin{align}
	\measure^{\otimes n}f >0 \text{ implies that } \PP\text{-a.s.\ } \gmcw[\measure,\Yop,\noise]^{\otimes n}f>0 \ .
\end{align}
\end{lem}
\begin{proof}
Write $\gmcw[\pmeasure,\Yop,\noise]=\gmcw[\noise]$ to simplify notation.
By \eqref{e.kahane.} we have that $\EE\,\gmcw[\noise]^{\otimes n}f \geq \pmeasure^{\otimes n}f$, so it suffices to show that
\begin{align}\label{e.01.}
	\PP\big[\gmcw[\noise]^{\otimes n} \, f>0\big] \in \{ 0, 1 \}\ .
\end{align}
By \eqref{e.shamov.2}, we have 
\begin{align}\label{e.01..}
    \gmcw[\noise+\shift]^{\otimes n} \, f
    =
    \gmcw[\noise]^{\otimes n} \, (e^{\Yop\shift})^{\otimes n}f\ ,
    \quad
    \shift\in\lsp^2\ .
\end{align}
Since $e^{\Yop\shift}>0$ $\measure$-a.e.\ and every $\measure$-null set is a.s.\ a $\gmcw[\noise]$-null set (in consequence of \eqref{e.shamov.1}), we have that $e^{\Yop\shift}>0$ $\gmcw[\noise]$-a.e.
Hence, the right-hand side of~\eqref{e.01..} is strictly positive if and only if $\gmcw[\noise]^{\otimes n} \, f$ is, which implies $\PP[\gmcw[\noise+\shift]^{\otimes n} \, f>0]=\PP[\gmcw[\noise]^{\otimes n} \, f>0]$.
Given this, \eqref{e.01.} follows from a Gaussian zero-one law: see \cite[Theorem~2.5.2]{bogachev1998gaussian}.
\end{proof}
\noindent{}%
Lemma~\ref{l.01} applied to the conditional GMC immediately gives the following result.
\begin{cor}\label{c.01}
For any $\PM^\theta$-measurable random $0\leq f=f(\omega)\in\BMsp(\Psp)$, 
\begin{align}
	\P\otimes\PP\text{-a.s.,}
	\quad
	\ind_{\PM^\theta{}^{\otimes n}F>0}
	\leq
	\ind_{\gmcw[\PM^\theta,\Yst^{\theta},\noise]^{\otimes n}F>0}\ .	
\end{align}
\end{cor}

The second ingredient in the proof of Corollaries~\ref{c.positivity}--\ref{c.superc}  is the GMC coupling described in Proposition~\ref{p.gmc.coupling}, which we now prove.
Recall that $\noise_1(\aa),\noise_2(\aa),\ldots$ denote independent two-sided Brownian motions in $\aa\in\R$ and that $\noise_i(\theta,\theta'):=\noise_i(\theta')-\noise_i(\theta)$.
We let $\noise(\theta,\theta'):\lsp^2\to\Lsp^2(\noisesp,\filtt,\PP)$, $(\noise(\theta,\theta'))v:=\sum_{i} (\noise_i(\theta,\theta'))v_i$.
\begin{proof}[Proof of Proposition~\ref{p.gmc.coupling}]
We begin by establishing a consistent property.
Take any $\phi\in\R$ and consider $\PMm(\theta):=\gmcw[\PM^{\phi},\Yst^{\phi},\noise(\phi,\theta)]$ for $\theta\geq\phi$.
We claim that, for every $\theta_0>\phi$, a.s.
\begin{align}
	\label{e.p.gmc.coupling.goal}
	\Yst^{\phi}\, \Yst^{\phi\,*}
	=
	\intersop^{\theta_0}
	\quad
	\text{in }\Lsp^2(\Psp,\PMw{}^{\theta_0})\ .
\end{align}
To prove this, in light of Kahane's martingale approximation \eqref{e.kahane.martingale}, consider
\begin{align}
	\label{e.p.gmc.coupling.1}
	\PMm_{\ell}
	:=
	\PM^{\phi} \,
	\exp\Big( \sum_{i=1}^{\ell} \noise_i(\phi,\theta_0)\, \Yst^{\phi}e_i - \frac{\theta_0-\phi}{2}\, \big(\Yst^{\phi}e_i \big)^2 \Big)\ ,
\end{align}
and apply Lemmas~\ref{l.embed.eigen} and \ref{l.embed.tool} with $\msp,\measure,\measure_{\ell},\measure',\ker$ respectively being $\Psp,\PM^{\theta_0},\PMm_{\ell},\PM^{\theta_1},\inters$.
Doing so requires verifying Assumptions~\ref{a.embed.eigen} and \ref{a.embed.tool}.
Given \eqref{e.inters.finite} and given that $\PMm_{\ell}$ is martingale, the verification follows similar arguments in the proofs of Lemma~\ref{l.version} and Proposition~\ref{p.embedding.}, which we do not repeat.
Next, the following tower property of the GMC is readily verified 
\begin{align}
    \label{e.gmcw.tower}
    \gmcw[\measure, \Yop, \noise(\phi,\theta)]
    =
    \gmcw\big[\,\gmcw[\measure,\Yop,\noise(\phi,\theta_0)]\,, \Yop, \noise(\theta_0,\theta)\big]\ ,
    \quad
    \phi<\theta_0\leq \theta\ .
\end{align}
Using this for $\measure=\PM^{\phi}$ and $\Yop=\Yst^{\phi}$ and combining the result with \eqref{e.p.gmc.coupling.goal} and Lemma~\ref{l.polar}\eqref{l.polar.2} give the consistent property
\begin{align}
	\label{e.p.gcm.coupling.consistent}
	\big( \PMm(\theta) \big)_{\theta \geq \theta_0}
	\text{ is equal in law to }
	\big( \gmcw[ \PMm(\theta_0), \Yst^{\theta_0}, \noise(\theta_0,\theta) ] \big)_{\theta \geq \theta_0}\ .
\end{align}

We now construct the desired coupling.
So far, we have obtained a coupling of $\PM^{\theta}$ for $\theta\geq\phi$ through $\PMm(\cdott)$.
Given the consistent property \eqref{e.p.gcm.coupling.consistent}, Kolmogorov extension theorem extends this coupling to $\theta\in\R$. 
To show \eqref{e.gmc.coupling}, under the so constructed coupling, for every $\phi<\theta_0\in\R$, let $\filttt_{\phi,\theta_0}:=\sigma(\PM^{\theta'}_{[s,t]}\,|\,\theta'\in[\phi,\theta_0])$ observe that by \eqref{e.p.gcm.coupling.consistent}
the law of $(\PM^{\theta})_{\theta\geq\theta_0}$ conditioned on $\filttt_{\phi,\theta_0} $ is equal to the law of $
	( \gmcw[ \PM^{\theta_0}, \Yst^{\theta_0}, \noise(\theta_0,\theta) ] )_{\theta \geq \theta_0}
$.
Combining this with $\vee_{\phi>-\infty}\filttt_{\phi,\theta_0}=\filttt_{\theta_0}$
 verifies \eqref{e.gmc.coupling}.
\end{proof}

\subsection{Proof of Corollaries~\ref{c.positivity}--\ref{c.superc}.}
\label{s.apps.}
First, a few declarations.
As said, the proof requires a few more moment results, and, to streamline the presentation, we defer the proof of those results to Section~\ref{s.apps.moment} and make explicit references in the proof.
Throughout the proof, we work under the GMC coupling in Proposition~\ref{p.gmc.coupling}.

Now we prove Corollary~\ref{c.positivity}. Note that the second statement follows immediately from the first statement because $\flow^{\theta}_{s,t}$ is the time-$\{s,t\}$ marginal of $\PM^{\theta}$. 
Let us extend our notation conventions at the beginning of Section~\ref{s.apps} by putting $\wien:=\wien_{[s,t]}$.  
We assume that the $f$ in Corollary~\ref{c.positivity} is bounded and has bounded support.
Doing so loses no generality since the truncated function $\hatf(\Path):=R\wedge f(\Path)\,\ind_{\norm{\Path}_\infty\leq R}$ is not $\wien$-a.e.\ zero for a large enough $R>0$ and satisfies $\PM^{\theta}f\geq\PM^{\theta}\hatf$.
Under the GMC coupling, Lemma~\ref{l.01} with $n=1$ and $F=f$ yields that
\begin{align}
    \label{e.positivity.coupling}
    \text{for every } \theta'\leq\theta,\quad
    \ind_{\PM^{\theta'}\hatf>0} \leq \ind_{\PM^{\theta}\hatf>0}\ \text{a.s.}
\end{align}
We will prove in Section~\ref{s.apps.moment} that, for any $n\in \N$ and bounded $\psi\in\BMsp(\Psp)$ having a bounded support
\begin{align}
    \label{e.infiniteT.lim}
    \E\big(\PM^{\theta'} \psi - \wien \psi\big)^{2n}
    \longrightarrow 0,
    \quad
    \text{as }\theta'\to-\infty\ .
\end{align}
Applying \eqref{e.infiniteT.lim} with $\psi=f$ gives that $\PM^{\theta'}f$ converges in probability to the constant $\wien f>0$ as $\theta'\to-\infty$, and consequently $\ind_{\PM^{\theta'}f>0}\to 1$ in probability.
Combining this with \eqref{e.positivity.coupling} yields the desired result.

Next, we state and prove a slightly stronger version of Corollary~\ref{c.superc}.
Note that Corollary~\ref{c.superc} follows if $\flow^\theta_{s,t}\,g\otimes g'\to 0$ a.s.\ for every $g,g'\in \Cloc(\R^2)$, so the following is a slightly stronger version of it.
\begin{customcor}{\ref*{c.superc}'}
\label{c.superc.} 
For every deterministic $g,g'\in\BMsp(\R^2)$ such that
$
	\Ip{ |g|^{\otimes 2}, \sg^{\intv{2},\theta}(t-s)\, |g'|^{\otimes 2} } < \infty
$
for all $\theta\in\R$, we have that
$
	\flow^{\theta}_{s,t}\, g\otimes g' \to 0
$
a.s.\ as $\theta\to\infty$.
\end{customcor}

In the proof of Corollary~\ref{c.superc.} below, we assume without loss of generality that $g,g'$ are nonnegative and not Lebesgue-a.e.\ zero.
If $g,g'$ are signed, we can decompose them into their positive and negative parts, and if $g$ or $g'$ is Lebesgue-a.e.\ zero, we would already have $\PM^\theta g\otimes g'=0$ a.s.


We begin the proof of Corollary~\ref{c.superc.} by establishing a positivity property.
Write $\ipp{\,\cdot\, , \,\cdot\,}_{\theta}$ for the inner product on $\Lsp^2(\Psp,\PMw{}^{\theta})$ and put $f:=g\otimes g'\,\circ\margin{\{s,t\}}$.  
We show that
\begin{align}
	\label{e.superc.positive}
	\text{for every }\theta_0\in\R,
	\quad
	\ipp{f,\intersop^{\theta_0} f}_{\theta_0}>0 \text{ a.s. }
\end{align}
For $R>0$, consider the truncation $\hatg:=R\wedge g\,\ind_{|\cdott|\leq R}$ of $g$, and define $\hatg'$ analogously in terms of $g'$.  
We put $\hatf:=\hatg\otimes \hatg'\,\circ\margin{\{s,t\}}$ and assume that $R$ is large enough so that $\hatg,\hatg'$ are not Lebesgue-a.e.\ zero. 
To establish~\eqref{e.superc.positive}, it suffices to show that $\ipp{\hatf,\intersop^{\theta_0} \hatf}_{\theta_0}>0$ a.s.  
For $\theta\in \R$, define the variable
\begin{align}
    h_{\theta}:= \frac{\ipp{\hatf,\intersop^{\theta} \hatf}_{\theta}}{\E\ipp{\hatf,\intersop^{\theta} \hatf}_{\theta}}\ .
 \end{align}
The expectation in the denominator can be seen to be strictly positive and finite using the formula \eqref{e.fTf} in Section~\ref{s.apps.moment}.
Apply Lemma~\ref{l.01} under the GMC coupling with $n=2$ and $F=\hatf{}^{\otimes 2}\,\inters$ and note that $\PMw{}^{\theta\,\otimes 2}F=\ipp{\hatf,\intersop^{\theta} \hatf}_{\theta}$.
Doing so gives that for any $\theta< \theta_0$
\begin{align}\label{e.superc.coupling.}
    \ind_{h_{\theta}>0} \leq \ind_{h_{\theta_0}>0}
    \quad\text{a.s.}
\end{align}
We will prove in Section~\ref{s.apps.moment} that
\begin{align}
	\label{e.variance}
	\Var h_{\theta}
	=
	\frac{\Var\ipp{\hatf,\intersop^{\theta} \hatf}_{\theta} }{ (\E\ipp{\hatf,\intersop^{\theta} \hatf}_{\theta})^2 }
	\longrightarrow 
	0 \ ,
	\quad
	\text{as }\theta\to-\infty\ .
\end{align}
Since $\E h_{\theta}=1$, this implies that $h_{\theta}\to 1$ in probability as $\theta\to-\infty$.
Combining this with \eqref{e.superc.coupling.} gives the desired positivity of $\ipp{\hatf,\intersop^{\theta_0} \hatf}_{\theta_0}$ and thus of~\eqref{e.superc.positive}.

With~\eqref{e.superc.positive} in hand, we turn to completing the proof of Corollary~\ref{c.superc.}. 
To be specific, let us set $\theta_0=0$. 
For $\gmcw{}^{\theta'}:=\gmcw[\PM^{0},\Yst^{0},\noise(0,\theta')]$, the process $(\gmcw{}^{\theta'}f)_{\theta'\geq 0}$ is a nonnegative martingale and thus converges a.s.\ as $\theta'\to\infty$. 
Hence, we only need to show that $\gmcw{}^{\theta'} f \to 0$ in probability as $\theta'\to\infty$. 
Since $\ipp{f,\intersop^{0} f}_{0}=\sum_{i=1}^{\infty}\ipp{f,\Yst^{0} e_i}_{0}^2 $, the positivity~\eqref{e.superc.positive} implies that there a.s.\ exists a $j=j(\PM^{0})\in\N$ such that $\ipp{f,\Yst^{0} e_j}_{0}\neq 0$.
This further implies that there exists an $\PM^{0}$-positive set on which $f$ and $\Yst^{0} e_j$ are both nonzero. 
Use $\EE$ to denote the expectation of $\noise(\cdott)$.
Taking $\EE[ \cdott| \, \noise_j(\cdott),\PM^{0} ]$ in \eqref{e.kahane.martingale} and swapping the conditional expectation with the limit give
\begin{align}\label{e.superc.coupling.re}
	\EE\big[ \,\gmcw{}^{\theta'} f \big| \, \noise_j(\cdott),\PM^{0} \big]
	=
	\PM^{0} \big[ e^{\noise_j(0,\theta)\Yst^{0} e_j-\frac{\theta'}{2}(\Yst^{0} e_j)^2} f \big]\ ,
\end{align}
where the swapping is justified due to the assumption on $g,g'$ in Corollary~\ref{c.superc.}.
Because there exists an $\PM^{0}$-positive set on which $f$ and $\Yst^{0} e_j$ are both nonzero, the right-hand side of~\eqref{e.superc.coupling.re} converges to $0$ in probability as $\theta'\to\infty$.
Since $\gmcw{}^{\theta'} f$ is nonnegative, this forces $\gmcw{}^{\theta'} f$ to also converge to $0$ in probability.

\begin{rmk}\label{r.superc}
Theorem~1.4 in \cite{caravenna2025singularity} gives that, for every deterministic bounded $A\subset\R^2$ and $\theta_0\in\R$ fixed, $\flow_{0,t'}^{\theta_0}1\otimes \ind_{A}\to 0$ in law as $t'\to\infty$.
The long time limit can be translated into the $\theta\to\infty$ limit by the scaling property of the SHF.
Namely, $\flow_{0,t'}^{\theta_0}1\otimes \ind_{A}$ is equal in law to $\flow^{\theta_0+\log t'}_{0,1}\,1\otimes t' \ind_{t'{}^{-1/2}A}$.
Hence, \cite[Theorem~1.4]{caravenna2025singularity} and Corollary~\ref{c.superc.} are similar but different in how the test functions are scaled.
\end{rmk}

\subsection{Proof of \eqref{e.infiniteT.lim} and \eqref{e.variance}}
\label{s.apps.moment}
We start by investigating the behaviors of $\jfn^\theta$ and $\Jop^{\intv{n},\theta}_{\alpha}$ as $\theta\to-\infty$.
Since $n$ will vary in this subsection, we restore the explicit dependence on it by writing $\Jop^{\intv{n},\theta}_{\alpha}$, $\sg^{\intv{n},\theta}$, etc.
First, there exists a universal $c<\infty$ such that, for all $\theta\leq -1$ and $t>0$,
\begin{align}
	\label{e.jfn.bd}
	\jfn^\theta(t)
	\leq
	c \, e^{ct} \, t^{-1} \, (|\log (t\wedge\tfrac{1}{2})| + |\theta| )^{-2}\ .
\end{align}
This follows from the proof of \cite[Lemma~8.4, Equation~(8-5)]{gu2021moments}.
The proof of that lemma fixes $\theta$ (called $\beta_\star$ there), but the same calculations yield \eqref{e.jfn.bd} for $\theta\leq -1$.
Applying  \eqref{e.jfn.bd} to \eqref{e.Jop} yields that 
\begin{subequations}
\label{e.infinityT}
\begin{align}
	\label{e.infinityT.Jopbd}
	\Normop{\Jop^{\intv{n},\theta}_{\alpha}(t)}
	&\leq
	c \, e^{ct} \, t^{-1} \, (|\log (t\wedge\tfrac{1}{2})| + |\theta| )^{-2}\ ,
\\
	\label{e.infinityT.Jopbd.int}
	\int_{0}^{\infty} \d t \, e^{-2ct}\, \Normop{\Jop^{\intv{n},\theta}_{\alpha}(t)}
	&\leq
	c \, |\theta|^{-1}\ ,
\end{align} 
\end{subequations}
for all $r\in(0,\infty)$ and $\theta\leq -1$.  Moreover, applying \cite[Lemma~8.4, Equation~(8-4)]{gu2021moments} with $z=-1$ (or any fixed negative value) yields that 
$
	 \int_{0}^{r} \d t\ |\theta|\,\jfn^{\theta}(t) \to 1
$
as $\theta\to-\infty$, for any $r\in(0,\infty)$, and consequently
\begin{align}
	\label{e.infinityT.Jop.}
	\int_{0}^{r} \d t \ |\theta|\,\Jop^{\intv{n},\theta}_{\alpha}(t)
	&\longrightarrow
	\id
	\text{ strongly as } \theta\to-\infty\ .
\end{align} 

Let us now prove \eqref{e.infiniteT.lim}.
As explained after \eqref{e.annealed.centered}, $\E(\PM^{\theta}-\wien)^{\otimes 2n}$ is a positive measure.
We use this fact to bound the left-hand side of \eqref{e.infiniteT.lim} by
$
	\E(\PM^{\theta}-\wien)^{\otimes 2n} |\psi|^{\otimes 2n}
$.
Next, since $\psi$ is bounded with a bounded support, there exists an $m>0$ large enough so that $|\psi(\Path)| \leq m\, \ind_{|\Path(s)|<m}\, \ind_{|\Path(t)|<m}$.
Combining this bound and \eqref{e.annealed.centered} gives
\begin{align}
	\text{lhs of \eqref{e.infiniteT.lim}}
	\leq
	m\,\IP{ \ind_{|\cdott|<m}^{\otimes 2n}, \sum_{\vecalpha\in\dgm_*\intv{2n}} \sgsum^{\intv{2n},\theta}_{\vecalpha}(t-s) \, \ind_{|\cdott|<m}^{\otimes 2n} }\ ,
\end{align}
in which $\dgm_*\intv{n}:=\{\vecalpha\in\dgm\intv{n}\,|\, \alpha_1\cup\cdots\cup\alpha_{\last}=\intv{n}\}$.
Combining \eqref{e.opbd.incoming}--\eqref{e.opbd.swapping.int}, \eqref{e.infinityT.Jopbd}--\eqref{e.infinityT.Jopbd.int}, and Lemma~\ref{l.sum} shows that, as $\theta\to-\infty$, the sum of operators on the right-hand side converges to $0$ in operator norm.
This proves \eqref{e.infiniteT.lim}.

To prove \eqref{e.variance}, our first step is to derive explicit expressions for $\E\ipp{\hatf,\intersop^{\theta} \hatf}_{\theta}$ and $\E\ipp{\hatf,\intersop^{\theta} \hatf}_{\theta}^2$ using a similar argument as in the proof of Lemma~\ref{l.expansion}.
Let $\intersop^{\theta}_{\e}$ be the analog of $\intersop^{\theta}=\intersop^{\theta,[s,t]}_{[s,t]}$ with the $\inters_{[s,t]}[s,t]$ (defined in \eqref{e.intersop}) replaced by $\inters_{[s,t],\e}[s,t]$ (defined in \eqref{e.inters.e}). 
Recall that $\Psi^{\e}_{\alpha}(x):=|\log\e|^{-2}\,\ind_{|x_{i}-x_j|\leq\e}$ for $\alpha=ij\in\pair\intv{n}$ and $\hatf=\hatg\otimes \hatg'$.  
Put  $\barg:=\hatg\, e^{-|\cdott|}$ and $\barg':=\hatg'\, e^{-|\cdott|}$.
Similarly to \eqref{e.expmom.e.2}, we have
\begin{align}
	\label{e.fTf.e}
	\E \ipp{\hatf,\intersop^{\theta}_\e \hatf}_{\theta}
	&=
	\int_{\Sigma(t-s)} \d u \ 
	\Ip{ \barg^{\otimes 2}, 
		\sg^{\intv{2},\theta}(u)\, \Psi^{\e}_{12}\, \sg^{\intv{2},\theta}(u') 
	\, \barg'{}^{\otimes 2} }
	\ ,
\\
	\label{e.fTf2.e}
	\E \ipp{\hatf,\intersop^{\theta}_\e \hatf}_{\theta}^2
	&=
	2\int_{\Sigma(t-s)} \d \vecu \ 
	\Ip{ \barg^{\otimes 4}, 
		\sg^{\intv{4},\theta}(u)\, \Psi^{\e}_{12}\, \sg^{\intv{4},\theta}(u') \, \Psi^{\e}_{34}\, \sg^{\intv{4},\theta}(u'') 
	\, \barg'{}^{\otimes 4} }
	\ ,
\end{align}
where the inner products appearing on the right-hand sides of \eqref{e.fTf.e} and \eqref{e.fTf2.e} respectively correspond to $\Lsp^2(\R^{4},\d x)$ and $\Lsp^2(\R^{8},\d x)$.
By Lemma~\ref{l.expansion} and \eqref{e.inters.L1conv}, the left-hand sides of \eqref{e.fTf.e}--\eqref{e.fTf2.e} converge respectively to $\E\ipp{\hatf,\intersop^{\theta} \hatf}_{\theta}$ and $\E\ipp{\hatf,\intersop^{\theta} \hatf}_{\theta}^2$ as $\e\to 0$. 
Moreover, the $\e\to 0$ limits of the right-hand sides of \eqref{e.fTf.e}--\eqref{e.fTf2.e} can be obtained through the same argument that follows \eqref{e.expmom.e.2}, which we do not repeat, and we describe the limiting expressions below.
Recall $\Cop^{\theta}_{\vecalpha}=\Cop^{\intv{n},\theta}_{\vecalpha}$ from \eqref{e.Cop}, and let
\begin{align}
	\label{e.Uop}
	\Uop^{\theta}(t')
	&:=
	\int_{\Sigma(t')} \d \vecu \ 
	\heatsg^{\intv{2}}_{12}(u_1)^*\, \Jop^{\intv{2},\theta}_{12}(u_2) \, \tfrac{1}{4\pi} \, \Jop^{\intv{2},\theta}_{12}(u_3) \, \heatsg^{\intv{2}}_{12}(u_4)
	\ ,
\\
	\label{e.Vop}
	\Vop^{\theta}(t')
	&:=
	\sum_{12\to 34}
	2\int_{\Sigma(t')} \d \vecu \ 
	\heatsg^{\intv{4}}_{\alpha^{0}_1}(u_{-1})^* \, \Cop^{\intv{4},\theta}_{\vecalpha^0}(u_0)\, 
	\tfrac{1}{4\pi}\, \Cop^{\intv{4},\theta}_{\vecalpha^1}(u_1)\,  \, 
	\tfrac{1}{4\pi}\, \Cop^{\intv{4},\theta}_{\vecalpha^2}(u_2)\, \heatsg^{\intv{4}}_{\alpha^{2}_\last}(u_3)
	\ ,
\end{align}
where $\sum_{12\to 34}$ runs over $\vecalpha^0,\vecalpha^1,\vecalpha^2\in\dgm\intv{4}$ under the condition that
$
	\alpha^{0}_{\last} = \alpha^{1}_{1} = 12
$
and
$
	34 = \alpha^{1}_{\last} = \alpha^{2}_{1}
$.
We have
\begin{align}
	\label{e.fTf}
	\E \ipp{\hatf,\intersop^{\theta} \hatf}_{\theta}
	=
	\ip{ \barg^{\otimes 2}, \Uop^{\theta}(t-s) \, \barg'{}^{\otimes 2} }\ ,
	&&
	\E \ipp{\hatf,\intersop^{\theta} \hatf}_{\theta}^2
	=
	\ip{ \barg^{\otimes 4}, \Vop^{\theta}(t-s) \, \barg'{}^{\otimes 4} }\ .
\end{align}

Next, we separate $\Vop^{\theta}$ into its leading term and remainder, and bound the latter.
Consider the total index length $|\vecalpha^0|+|\vecalpha^1|+|\vecalpha^2|$ under the sum $\sum_{12\to 34}$.
The shortest total length is achieved by $\vecalpha^0=(12)$, $\vecalpha^1=(12,34)$, $\vecalpha^{2}=(34)$.
Let
\begin{align}
	\label{e.Vop.leading.}
	\Vop^{\theta}_0(t')
	&:=
	2\int_{\Sigma(t')} \d \vecu \ 
	\heatsg^{\intv{4}}_{12}(u_{1})^* \, \Cop^{\intv{4},\theta}_{(12)}(u_2)\, 
	\tfrac{1}{4\pi}\, \Cop^{\intv{4},\theta}_{(12,34)}(u_3)\,  \, 
	\tfrac{1}{4\pi}\, \Cop^{\intv{4},\theta}_{(34)}(u_4)\, \heatsg^{\intv{4}}_{34}(u_5)
\end{align}
denote the corresponding term, let $\Vop^{\theta}_\rest(t')$ denote the rest of the sum in \eqref{e.Vop}, so that $\Vop^{\theta}=\Vop^{\theta}_0+\Vop^{\theta}_\rest$.
Note that each summand of $\Vop^{\theta}_\rest$ has total index length $\geq 5$, and note from \eqref{e.Cop} that $\Cop^{\intv{n},\theta}_{\vecalpha}$ has $|\vecalpha|$ copies of $\Jop$s in its definition. 
Given these observations and using \eqref{e.opbd.incoming}--\eqref{e.opbd.swapping.int}, \eqref{e.infinityT.Jopbd}, and Lemma~\ref{l.sum}, we have that 
\begin{align}
	\label{e.Vop.rest}
	\Normop{ \Vop^{\theta}_\rest(t') }
	\leq
	c(t') \, |\theta|^{-5}\ ,
	\qquad
	\theta \leq -1\ .
\end{align}

Next, let us derive the $\theta\to-\infty$ asymptotics of $\Uop^{\theta}$ and $\Vop^{\theta}_0$.
Use \eqref{e.Cop} to write as 
\begin{align}
\label{e.Vop.leading}
\begin{split}
	\Vop^{\theta}_0(t')
	=
	2\int_{\Sigma(t')} \d \vecu \ 
	&\heatsg^{\intv{4}}_{12}(u_{1})^* \, \Jop^{\intv{4},\theta}_{12}(u_2)\, 
	\tfrac{1}{4\pi}\, 
	\Jop^{\intv{4},\theta}_{12}(u_3) \, 
\\
	&\cdot \heatsg^{\intv{4}}_{12\ 34}(u_4)\, 
	\Jop^{\intv{4},\theta}_{34}(u_5) \, 	
	\tfrac{1}{4\pi}\, \Jop^{\intv{4},\theta}_{34}(u_6)\, 
	\heatsg^{\intv{4}}_{34}(u_7)\ .
\end{split}
\end{align}
From \eqref{e.Uop} and \eqref{e.Vop.leading}, we see that $\Uop^{\theta}$ and $\Vop^{\theta}_0$ respectively have $2$ and $4$ copies of $\Jop$s.
In view of this and \eqref{e.infinityT.Jopbd.int}, we consider the $\theta\to-\infty$ limits of $|\theta|^2\Uop^{\theta}(t')$ and $|\theta|^{4}\Vop^{\theta}_0(t')$.
For the first limit, take \eqref{e.Uop}, multiply both sides by $|\theta|^2$, and apply Lemma~\ref{l.delta} with $\e=1/|\theta|$, $m=4$, $I=\{2,3\}$, and $\Top_{2,\e}=\Top_{3,\e}=|\theta|\Jop^{\intv{2},\theta}_{12}$.
For the second limit, take \eqref{e.Vop.leading}, multiply both sides by $|\theta|^4$, and apply Lemma~\ref{l.delta} with $\e=1/|\theta|$, $m=7$, $I=\{2,3,5,6\}$, $\Top_{2,\e}=\Top_{3,\e}=|\theta|\Jop^{\intv{4},\theta}_{12}$, and $\Top_{4,\e}=\Top_{5,\e}=|\theta|\Jop^{\intv{4},\theta}_{34}$.
Doing gives that, as $\theta\to-\infty$,
\begin{align}
	\label{e.Uop.lim}
	&|\theta|^{2}\,\Uop^{\theta}(t') 
	\longrightarrow 
	\Uop(t')
	:=
	\frac{1}{4\pi} \int_{\Sigma(t')} \d \vecu \ 
	\heatsg^{\intv{2}}_{12}(u)^*\, \heatsg^{\intv{2}}_{12}(u')
	\quad
	\text{strongly},
\\
	\label{e.Vop.lim}
	&|\theta|^{4}\,\Vop^{\theta}(t') 
	\longrightarrow 
	\Vop(t')
	:=
	\frac{2}{(4\pi)^2} \int_{\Sigma(t')} \d \vecu \ 
	\heatsg^{\intv{4}}_{12}(u)^*\, \heatsg^{\intv{4}}_{12\ 34}(u')\, \heatsg^{\intv{4}}_{34}(u'')
	\quad
	\text{strongly}.
\end{align}

Let us now complete the proof.
Combining \eqref{e.fTf}, \eqref{e.Vop.rest}, and \eqref{e.Uop.lim}--\eqref{e.Vop.lim} gives
\begin{align}
	\label{e.variance.1}
	\lim_{\theta\to-\infty}
	\frac{ \Var\ipp{\hatf,\intersop^{\theta} \hatf}_{\theta} }{ (\E\ipp{\hatf,\intersop^{\theta} \hatf}_{\theta})^2 }
	=
	\frac{
		\ip{ \barg^{\otimes 4}, \Vop(t-s)\, \barg'{}^{\otimes 4} }
		-
		\ip{ \barg^{\otimes 2}, \Uop(t-s)\, \barg'{}^{\otimes 2} }^2 	
	}{ 
		\ip{ \barg^{\otimes 2}, \Uop(t-s)\, \barg'{}^{\otimes 2} }^2 
	}\ .
\end{align}
The denominator of the right-hand side is nonzero and finite because $\barg,\barg'$ are not Lebesgue a.e.\ zero.
By \eqref{e.incoming}--\eqref{e.swapping} and the semigroup property of the heat kernel, it is straightforward to check that the integrand in \eqref{e.Vop.lim} has the kernel
\begin{align}
\begin{split}
	\int_{\R^2}\d\yc &\prod_{i=1}^2 \hk(u,x_i-\yc)\,\hk(u'+u'',\yc-x'_i) 
\\
	&\cdot
	\prod_{i=3}^4 \int_{\R^2}\d\yc\ \hk(u+u',x_i-\yc)\,\hk(u'',\yc-x'_i)\ .
\end{split}
\end{align}
Multiply this kernel by $\barg(x_1)\cdots\barg(x_4)$ and $\barg'(x'_1)\cdots\barg'(x'_4)$, integrating over the $x_i$s and $x'_i$s, and rename $u,u',u''$ to $t_1-s,t_2-t_1,t-t_2$.
Doing so gives
\begin{align}
	\label{e.variance.2}
	\ip{ \barg^{\otimes 4}, \Vop(t-s)\, \barg'{}^{\otimes 4} }
	=
	\frac{2}{(4\pi)^2} \int_{s<t_1<t_2<t} 
	\hspace{-20pt} \d t_1 \d t_2
	\prod_{j=1}^2 \Ip{\barg^{\otimes 2}, \heatsg^{\intv{2}}_{12}(t_j-s)^*\,\heatsg^{\intv{2}}_{12}(t-t_j)\, \barg'{}^{\otimes 2} }\ .
\end{align}
Because the integrand is symmetric in $t_1,t_2$, the right-hand side of \eqref{e.variance.2} is equal to $\ip{ \barg^{\otimes 2}, \Uop(t-s)\, \barg'{}^{\otimes 2} }^2$.
This shows that \eqref{e.variance.1} is equal to $0$ and completes the proof of \eqref{e.variance}.

\appendix

\section{Properties of the polymer measure}
\label{s.a.PM}

\begin{proof}[Proof of Proposition~\ref{p.axioms}]
First, using Condition~\eqref{p.axioms.coupling} for $\ell=2$ and considering the time-$\{t_0,t_2\}$ marginals give that $\floww_{t_0,t_2}=\floww_{t_0,t_1}\cldot\floww_{t_1,t_2}$.
Combining this with Conditions~\eqref{p.axioms.ind}--\eqref{p.axioms.mome} and \cite[Theorem~1.9]{tsai2024stochastic} shows that $\floww$ is equal in fdd to the SHF($\theta'$), or more precisely, the SHF($\theta'$) restricted to times $s<t\in\unit\D$.
This together with Condition~\eqref{p.axioms.coupling} and \eqref{e.PM.fdd} shows that, for every finite $A\subset\unit\D\cap[0,\unit]$, the marginal $\gmcw_{[0,\unit]}\circ\margin{A}^{-1}$ is equal in law to $\PM^{\theta'}_{[0,\unit]}\circ\margin{A}^{-1}$.
Since $\gmcw_{[0,\unit]}$ and $\PM^{\theta'}_{[0,\unit]}$ are measures on the set $\Psp_{[0,\unit]}$ of \emph{continuous} paths, it follows that $\gmcw_{[0,\unit]}$ is equal in law to $\PM^{\theta'}_{[0,\unit]}$.
\end{proof}

Let us prepare an approximation tool.
Fix $[s,t]\subset\R$, write $\Psp:=\Psp_{[s,t]}$ to simplify notation.
For finite $A\subset [s,t]$, recall that $\margin{A}:\Psp\to\R^{2A}$ is the evaluation map.
For any $S\subset\Psp$, let $\margin{A|S}:S\to\R^{2A}$ be the restriction of $\margin{A}$ to $S$,
and let $\Cfin(S)$ denote the collection of bounded continuous functions on $S\subset \Psp$ of the form $\varphi=g\circ\margin{A|S}$ for some finite $A\subset [s,t]$ and $g\in\Cloc(\R^{2A})$, where the script ``fd'' stands for ``finite-dimensional''.

\begin{lem}\label{l.dense} 
Let $N$ be a random measure on $\Psp$ such that  $\E N\circ\margin{s}^{-1}$ is locally finite (finite on bounded sets).
Then there exists a deterministic countable subset $D$ of $\Cfin(\Psp)$ such that a.s.\ $D^{\otimes n}$ is dense in $L^p(\Psp^n,N^{\otimes n})$ for all $p\in [1,\infty)$ and $n\in\N$.
\end{lem}
\noindent{}%
Under the assumption of Lemma~\ref{l.dense}, it is not difficult to verify that every $\varphi=g\circ\margin{A}\in\Cfin(\Psp)$ with $s\in A$ is a.s.\ in $\Lsp^p(\Psp,N)$ for all $p\in[1,\infty)$, which we take as given in the following proof.

\begin{proof}
We begin by establishing a compactness property.
Let $B_r:=\{ |\Path(s)|<r\}\subset\Psp$.
The assumption that $\E N \circ\margin{s}^{-1}$ is locally finite ensures that $\E N$ is locally finite.
By the inner regularity of locally finite measures on Polish spaces, there exists a \emph{deterministic} sequence of compact sets $K_1\subset K_2 \subset \ldots \Psp$ with $K_j\subset B_j$ such that $\E N\,(B_j\setminus K_j)\leq 2^{-j}$.
In particular, there exists an event $\Omega_0$ with $\E1_{\Omega_0}=1$ such that
\begin{align}
    \label{e.l.dense.cpt}
    N(B_{j}\setminus K_j) \longrightarrow 0 \text{ as } j\to\infty\ ,
    \quad
    \text{ on } \Omega_0\ .
\end{align}

We now construct the set $D$.
Note that $\Cfin(K_j)$ is an algebra of functions in $\Csp(K_j)$ that separates points.
By the Stone--Weierstrass theorem, $\Cfin(K_N)$ is dense in $\Csp(K_N)$ wrt the uniform norm.
Note also that $\Cfin(K_j)$ is separable, so there exists a countable subset $D_j'$ of $\Cfin(K_j)$ that is dense in $\Csp(K_j)$. 
To each $\varphi'\in D_j'$, we will construct an extension $\varphi\in\Cfin(\Psp)$ that preserves the uniform norm (namely $\norm{\varphi}_{\infty}=\norm{\varphi'}_{\infty}$) with $\supp\varphi\subset B_j$. Once this is done, we form the set $D_j\subset \Cfin(\Psp)$ by collecting the extensions of members of $D_j'$ and let $D:=\cup_{j\in\N}D_j$.
To construct the extension, write $\varphi'=g'\circ\margin{A|K_j}$ for some finite $A\subset[s,t]$ and $g'\in\Csp(\R^{2A})$, and truncate $g'$ to obtain $g:= (g'\wedge \norm{\varphi'}_\infty)\vee(-\norm{\varphi'}_{\infty})$.
This way, the function $g\circ\margin{A}\in\Csp(\Psp)$ gives a uniform-norm preserving extension of $\varphi'$.
To ensure that the support of the extension is contained in $B_j$, let $r:=\sup\{ |\Path(s)| \, | \, \Path\in K_j \}$, and note that $r<j$ because $\overline{K_j}=K_j\subset B_j=B_j^\circ$. 
We then take any $h\in\Csp(\R^2,[0,1])$ with $h(x)=1$ for all $|x|\leq r$ and $h(x')=0$ for all $|x'|>j$, and put $\varphi:=h\circ|_s\,\cdot \,g \circ \margin{A}$.

We now prove that the so constructed $D$ satisfies the desired property.
Note that $\Cloc(\Psp)^{\otimes n}$ is a.s.\ dense in $L^p(\Psp^{n},N^{\otimes n})$, as can be readily seen from \cite[Lemma~1.37]{kallenberg2021foundations} for example.
Hence, it suffices to show that, for every $f\in\Cloc(\Psp)$, we have that $\inf_{\varphi\in D}\,N[|f-\varphi|^p]=0$ on $\Omega_0$.
To prove this, fixing any $f\in\Cloc(\Psp)$ and any $i\in\N$ large enough such that $\supp f\subset B_i$, for $\varphi\in D_j$ we bound
\begin{align}
    \label{e.l.desnse.bd}
    N\big[ |f-\varphi|^p \big]
    \leq
    \sup_{K_j} \big|f-\varphi\big|^p\, N( K_j )
    +
    \norm{f}^p_\infty \, N( B_i \setminus K_j )
    +
    \norm{\varphi}^p_\infty \, N( B_j \setminus K_j )\ ,
\end{align}
where we used that $\supp f \subset B_i$ and $\supp\varphi\subset B_j$.
Recall that $\varphi\in D_j$ is an extension of a member of $D'_j$, so $\norm{\varphi}_\infty=\norm{\varphi|_{K_j}}_\infty$.
Using this property to bound $\norm{\varphi}_\infty^p$ by
$2^p\norm{f}_{\infty}^p+2^p\sup_{K_j}|f-\varphi|^p$, insert this bound into \eqref{e.l.desnse.bd}, and take the infimum over $\varphi\in D_j$ on both sides.
Doing so gives
\begin{align}
    \label{e.l.dense.bd.}
    \inf_{\varphi\in D_j}
    N\big[ |f-\varphi|^p \big]
    \leq 
    &
    \inf_{\varphi\in D_j}
    \sup_{K_j} 
    \big|f-\varphi\big|^p \, 
    \big( N(K_j) + 2^p\,N(B_j\setminus K_i) \big)
\\
    \label{e.l.dense.bd..}
    &+
    \norm{f}^p_\infty \, \big( N( B_i \setminus K_j ) + 2^p\,N( B_j \setminus K_j ) \big)\ .
\end{align}
The second infimum in \eqref{e.l.dense.bd.} is zero by the construction of $D_j$.
The last term \eqref{e.l.dense.bd..} tends to $0$ as $j\to\infty$ on $\Omega_0$ by \eqref{e.l.dense.cpt}.
Hence, $\inf_{\varphi\in D}\,N[|f-\varphi|^p]=0$ on $\Omega_0$.
\end{proof}

We derive a convenient corollary from Lemma~\ref{l.dense}.
Take $[s,t]\subset[s',t']$, let $\PM_0:=\PM{}^{\theta}_{[s,t]}$, let $\PM_1:=\PM{}^{\theta}_{[s',t']}$, let $\PM_{0'}$ be time-$[s,t]$ marginal of $\PM_1$, let $\PMw_1:=\PMw{}^{\theta}_{[s',t']}$, let $\PMw_{0'}$ be time-$[s,t]$ marginal of $\PMw_1$, and put $\Psp_0:=\Psp_{[s,t]}$.

\begin{cor}\label{c.dense}
There exists a deterministic, countable $D\subset\Cfin(\Psp_0)$ such that a.s.\ $D^{\otimes n}$ is  dense in $\Lsp^p(\Psp^{n}_0,\PM^{\otimes n}_0)$, $\Lsp^p(\Psp^{n}_{0},\PM{}^{\otimes n}_{0'})$, $\Lsp^p(\Psp^{n}_0,\PMw{}^{\otimes n}_0)$, and $\Lsp^p(\Psp^{n}_{0},\PMw{}^{\otimes n}_{0'})$ for all $p\in[1,\infty)$ and $n\in\N$.
\end{cor}

\begin{proof}
Since $\E\PM_0,\E\PM_{0'}$ are both $\wien_{[s,t]}$ and since $ \wien_{[s,t]} \circ\,\margin{s}^{-1}$ is the Lebesgue measure on $\R^2$, Lemma~\ref{l.dense} applies to $\PM_{0},\PM_{0'}$, giving the desired result for them.
Recall that $\PMw_0,\PMw_{0'}$ differ from $\PM_0,\PM_{0'}$ only by exponential weights, so Lemma~\ref{l.dense} also applies to $\PMw_{0},\PMw_{0'}$.
\end{proof}

\begin{proof}[Proof of Proposition~\ref{p.prob.space}]
Part~\eqref{p.prob.space.1} follows from \cite[Proposition~2.12(ii)]{clark2024continuum}.
Part~\eqref{p.prob.space.2} for $\ell=2$ follows from \cite[Proposition~2.13]{clark2024continuum} and the same argument generalizes to $\ell>2$.

We now prove Part~\eqref{p.prob.space.3}.
To simplify notation, take $\ell =2$ and consider the case of fixing $t_3$ so that $t_2=t_3-r_1$.
The inclusion $\subset$ follows by definition.
To prove the inclusion $\supset$, we fix any $s<t\in\unit\D\cap[t_0,t_4]$ and show that $\flow^{\theta}_{s,t}$ is measurable wrt the left-hand side of \eqref{e.p.prob.space.3}.
When $[s,t]\subset[t_0,t_3)$ or $[t_3,t_4]$, the statement follows by definition, so we consider $s<t_3\leq t$.
By Axiom~\eqref{d.shf.ck}, $\flow^{\theta}_{s,t}$ is the $r\to 0$ limit of $(\flow_{s,t_1-r}\opdot_{r}\flow_{t_1,t})(\d x,\R^2,\d x')$, so it is measurable wrt the left-hand side of \eqref{e.p.prob.space.3}.
\end{proof}

\begin{lem}\label{l.positive}
The statement holds a.s: For all Borel $A\subset[s,t]$ and $\varphi\in\Lsp^2(\Psp_{[s,t]},\PMw{}^{\theta}_{[s,t]})$,
\begin{align}
    \label{e.l.positive}
	\int_{\Psp_{[s,t]}^2} \PMw^{\theta\,\otimes 2}_{[s,t]}(\d\Path,\d\Pathh) \,
	\inters^{\Path,\Pathh}_{[s,t]}(A) \, \varphi(\Path) \, \varphi(\Pathh)
	\geq 0\ .
\end{align}
\end{lem}
\begin{proof}
First, we perform a few measure-theoretical reductions.
For any fixed $A\subset[s,t]$, the left-hand side of \eqref{e.l.positive} is a continuous (quadratic) function in $\varphi\in\Lsp^2(\Psp_{[s,t]},\PMw{}^{\theta}_{[s,t]})$ thanks to the bound \eqref{e.inters.finite}.
Hence, by Corollary~\ref{c.dense} for $n=1$ and $p=2$, it suffices to prove the statement for any deterministic bounded $\varphi\in\Csp(\Psp_{[s,t]})$ that is supported in $\{|\Path(s)|\leq m\}$ for a large enough $m$.
For the fixed $\varphi$, because $\inters_{[s,t]}$ has no atoms, it suffices to prove the statement for any deterministic $A=[s,t_1]$.

Next, to prove the reduced statement, we appeal to the approximation \eqref{e.inters.e} but with a different choice of $\Phi$.
By \cite[Theorem~5.12]{clark2023planar}, the convergence in \eqref{e.inters.L1conv} holds for every radially symmetric, bounded, compactly supported $\Phi\in\BMsp(\R^2)$ such that $\int_{\R^2} \d x\,\Phi(x)=\pi$ (see Remark~\ref{r.units}).
Write $B(x,r)$ for the closed ball in $\R^2$ centered at $x$ with radius $r$.
Instead of choosing $\Phi$ as the indicator function, here we choose $\Phi(x):=|B(x,1)\cap B(0,1)| /c_0$, where $|\cdott|$ denotes the area, and $c_0:=\int_{\R^2} \d x\, |B(x,1)\cap B(0,1)|/\pi$.
For this $\Phi$, use the fact that $|B(x_1,1)\cap B(x_2,1)|$ depends only on $|x_2-x_1|$ to get
\begin{align}
	\label{e.l.positive.1}
	\inters^{\Path,\Pathh}_{\e,[s,t]}(\d u)
	=
	\frac{\d u}{c_0\e^2|\log\e|^2}
	\,
	\int_{\R^2} \d x \, \ind_{B(\Path(u)/\e,1)}(x) \, \ind_{B(\Pathh(u)/\e,1)}(x)\ .
\end{align}
The product form on the right-hand of \eqref{e.l.positive.1} yields that
\begin{align}
	\int_{\Psp_{[s,t]}^2} \PMw^{\theta\,\otimes 2}_{[s,t]}(\d\Path,\d\Pathh)\,
	\inters^{\Path,\Pathh}_{\e,[s,t]}[s,t_1]\,
	\varphi(\Path)\,\varphi(\Pathh)
	\geq 0\ 
\end{align}
and hence completes the proof.
\end{proof}

\begin{rmk}\label{r.units}
To match notation in the current paper and \cite{clark2023planar}, call the respective small parameters $\e$ and $\e'$, let $\Phi(x)$ be as in the current paper, let $\phi_{\e'}(a)$ be as in \cite[Theorem~2.15]{clark2023planar}, and use the relations $\e=\sqrt{2}\,\e'$, $|x|=\sqrt{2}\, a$, and 
$
	\frac{1}{\e^{2}|\log\e|^{2}}\Phi(\frac{x}{\e})
	=
	\phi_{\e'}(a)
$.
\end{rmk}

\section{Operators}
\label{s.a.ops}

\begin{proof}[Proof of Lemma~\ref{l.sum}]
The idea follows that of \cite[Lemma~8.10]{gu2021moments}.
Observe that, because of $\Sigma(t)=\{\vecu \,|\, u_1+\cdots+u_{m}=t\}$, at least one $u_{i_0}$ is $\geq t/m$. 
Using this observation and multiplying and dividing the integrand by $e^{\lambda t}$ gives
\begin{align}
	\label{e.l.sum.1}
	\text{lhs of \eqref{e.l.sum}}
	\leq
	e^{\lambda t}
	\sum_{i_0\in\intv{m}} 
	\sup_{u\in[\frac{t}{m},t]} e^{-\lambda u }\Normop{\Top_{i_0}(u)}
	\, 
	\prod_{i\in\intv{m}\setminus\{i_0\}}
	\NOrmop{ \int_0^{t} \d u \, e^{-\lambda u}\Top_{i}(u) }\ .
\end{align}
Hereafter, we assume $\lambda \geq c'+2$, and write $c$ for universal constants.
Using \eqref{e.l.sum.bd} to bound terms in the summand gives
\begin{align}
	\label{e.l.sum.2}
	\sup_{u\in[\frac{t}{m},t]} e^{-\lambda u }\Normop{\Top_{i}(u)}
	&\leq
	c\, c_i \,
	\begin{cases}
		(t/m)^{-b_i} \leq m\,t^{-b_{i}}\ ,	&\text{ when } i\in A_1\ ,
	\\
		m\, t^{-1} \, |\log(t\wedge\frac{1}{2})|^{-2}\ ,	&\text{ when } i\in A_2\ ,
	\\
		m\, t^{-1} \ ,	&\text{ when } i\in A_3\ ,
	\end{cases}
\\
	\label{e.l.sum.3}
	\NOrmop{ \int_0^{t} \d u \, e^{-\lambda u}\Top_{i}(u) }
	&\leq
	c\, c_i \,
	\begin{cases}
		t^{1-b_i}/(1-b_i)\ ,	&\text{ when } i\in A_1\ ,
	\\
		1 \ ,	&\text{ when } i\in A_3\ ,
	\end{cases}	
\end{align}
and 
$
	\Normop{\int_0^{t} \d u \, e^{-\lambda u}\Top_{i}(u)}
	\leq
	c\,c_i
	\int_0^t\d u \,e^{-(\lambda-c')u}\,u^{-1}\,|\log(u\wedge\frac{1}{2})|^{-2}
$ ,
when $i\in A_2$.
In the last integral, bounding $e^{-(\lambda-c')t}\leq e^{-2t}$ gives the bound $c\, |\log(t\wedge\frac{1}{2})|^{-1}$; alternatively, releasing the range of integration to $[0,\infty)$ and bounding the result gives the bound $c\,/\log(\lambda-c'-1)$.
Hence,
\begin{align}
	\label{e.l.sum.4}
	\NOrmop{ \int_0^{t} \d u \, e^{-\lambda u}\Top_{i}(u) }
	\leq
	c\, c_i \, 
	\begin{cases}
		|\log(t\wedge\tfrac{1}{2})|^{-1} 
	\\
		1/\log(\lambda-c'-1)
	\end{cases}	
	\ ,
	\quad
	\text{when } i\in A_2\ .
\end{align}
Apply \eqref{e.l.sum.2}--\eqref{e.l.sum.4} to bound the terms in \eqref{e.l.sum.1}.
When $i_0\in A_2$, we apply the first bound in \eqref{e.l.sum.4} $\ell$  times and the second bound therein $(|A_2|-1-\ell)_+$  times.
When $i_0\notin A_2$, we use the first bound in \eqref{e.l.sum.4}  $\min\{\ell+1,|A_2|\}$  times and the second bound $(|A_2|-1-\ell)_+$ times.
Doing so gives the desired result.
\end{proof}

\begin{proof}[Proof of \eqref{e.Fop.bds}]
We begin with an identity.
Recall $\Sop_\alpha$ from \eqref{e.Sop} and consider its $\e$ analog
\begin{align}
	\label{e.ysp.}
	&\R^{4}\times\R^{2\intv{n}\setminus\alpha} 
	:= 
	\big\{ y=(\yr, \yc, (y_i)_{i\in\intv{n}\setminus\alpha}) \, \big| \, \yr,\yc,y_i\in\R^2 \big\}
	\ ,
\\
	\label{e.Sop.e}
	&S^\e_{\alpha}: \R^{4}\times\R^{2\intv{n}\setminus\alpha} \to \R^{2\intv{n}},
	\qquad
	\big( S^\e_{\alpha} y \big)_{k}
	:=
	\begin{cases}
		\yc+\e\yr/2 & \text{when } k=i \,,
	\\
		\yc-\e\yr/2 & \text{when } k=j \,,
	\\
		y_{k} & \text{when } k\in\intv{n}\setminus\alpha\ ,
	\end{cases}
\end{align}
where we index the first two coordinates in \eqref{e.ysp.} respectively by $\textr$ and $\textc$ for ``relative'' and  ``center of mass''.
For $\alpha\in\pair\intv{n}\cup\{\emptyset\}$ and $\sigma\in\pair\intv{n}$, let
\begin{align}
	&\heatsger{\alpha}{\sigma}(t,y,y') := \heatsg(t,\Sop_\alpha y-\Sop^\e_{\sigma}y') \text{ when } \alpha\neq\emptyset,
	&&
	\heatsg^{\hphantom{\emptyset}\e}_{\emptyset\sigma}(t,x,y') := \heatsg(t,x-\Sop^\e_{\sigma}y')\ ,
\end{align}
where $x\in\R^{2\intv{n}}$, $y\in\R^{2}\times\R^{2\intv{n}\setminus\alpha}$, and $y'\in\R^{4}\times\R^{2\intv{n}\setminus\sigma}$, and let
$
	\heatsgel{\sigma}{\alpha}(t) := \heatsger{\alpha}{\sigma}(t)^*
$.
Let $\phi=\phi(\yr):=\ind_{|\yr|\leq 1}$ and view it as a multiplicative operator on $\Lsp^2(\R^{4}\times\R^{2\intv{n}\setminus\alpha},\d y)$.
It is straightforward to check from \eqref{e.expansion.Psi} and \eqref{e.Fop} that
\begin{align}
	\label{e.Fop..}
	\Fop^{\e}_{\alpha\sigma\alpha'}(t)
	=
	\frac{1}{|\log\e|^2 }
	\int_{u+u'=t} \d u\
	\heatsger{\alpha}{\sigma}(u)\, \phi \, \heatsgel{\sigma}{\alpha'}(u')\ .
\end{align}

Let us prove \eqref{e.Fop.bd}--\eqref{e.Fop.bdint}.
First, for $c=c(n)$, $\alpha\in\pair\intv{n}\cup\{\emptyset\}$, and $\sigma\in\pair\intv{n}$,
\begin{align}
	\label{e.Fop..bd}
	\Normop{ \heatsger{\alpha}{\sigma}(t)\, \phi }
	&\leq
	c\, t^{-1}\, e^{ct}\ ,
\\
	\label{e.Fop..bdint}
	\NOrmop{ \int_{0}^\infty \d t \ e^{-t}\,\heatsger{\alpha}{\sigma}(t)\, \phi }
	&\leq
	c\,
	\begin{cases}
		|\log\e|	&\text{when } \alpha=\sigma\ ,
	\\
		1 &\text{when } \alpha\neq\sigma\ .
	\end{cases}
\end{align}
When $\alpha\neq\sigma$, these bounds can be obtained by arguments similar to the first paragraph in \cite[Section~3.1]{surendranath2024two}, which we omit.
When $\alpha=\sigma$, the operator $\heatsger{\sigma}{\sigma}(t)\, \phi$ has the kernel
\begin{align}
	\label{e.Fop.ker}
	\hk(2t,\e\yr') \,\phi(\yr') \, \hk(\tfrac{t}{2},\yc-\yc') \, \prod_{i\in\intv{n}\setminus\sigma} \hk(t,y_i-y'_i)\ .
\end{align}
Indeed, the heat kernels in the middle and last factors together contract norm on $\Lsp^2(\R^{2}\times\R^{2\intv{n}\setminus\alpha},\d y)$.
This fact implies that the left-hand sides of \eqref{e.Fop..bd}--\eqref{e.Fop..bdint} for $\alpha=\sigma$ are respectively bounded by the $\Lsp^2(\R^2,\d\yr)$ norm of the functions $\hk(2t,\e\yr)\phi(\yr)$ and $\int_{0}^\infty \d t \, e^{-t}\,\hk(2t,\e\yr)\,\phi(\yr)$.
The former is indeed bounded by $c t^{-1}$.
As for the latter, we have that $\int_{0}^\infty \d t \ e^{-t}\,\hk(2t,\cdot)=\frac{1}{\pi}K_0(|\cdot|)$, where $K_0$ denotes the $0$th order modified Bessel function of the second kind, which satisfies $K_0(r)\leq c \log(1/r)$ for $0<r<1/2$; see \cite[page~379, Equation~9.6.13]{abramowitz1965handbook} for example.
These properties yield \eqref{e.Fop..bd}--\eqref{e.Fop..bdint} for $\alpha=\sigma$.
The bound \eqref{e.Fop.bd} now follows by considering the cases $u>t/2$ and $u<t/2$ in \eqref{e.Fop..} separately and using \eqref{e.Fop..bd}--\eqref{e.Fop..bdint}. The bound \eqref{e.Fop.bdint} then follows by applying $\int_{0}^\infty \d t\ e^{-t}$ to \eqref{e.Fop..} and using \eqref{e.Fop..bdint} in the result.

To prove \eqref{e.Fop.convg}, let $\sff_\e(u,u'):= |\log\e|^{-2} \int_{|\yr|\leq 1} \d\yr\,\hk(2u,\e\yr)\,\hk(2u',\e\yr)$ and use the kernel expression \eqref{e.Fop.ker} to obtain
\begin{align}
	\label{e.Fop.ker.}
	|\log\e|^{-2} \, \heatsger{\sigma}{\sigma}(u)\,  \phi \, \heatsgel{\sigma}{\sigma}(u')
	=
	\sff_\e(u,u') \, \hk(\tfrac{u+u'}{2}) \otimes \hk(u+u')^{\otimes\intv{n}\setminus\alpha}\ .
\end{align}
It is straightforward to check that, for every $0<\delta<t'<\infty$, as $\e\to 0$, $\int_{[0,t']^2\setminus [0,\delta]^2} \d u \d u' \ \sff_\e(u,u') \to 0 $ and $\int_{[0,t']^2} \d u \d u'\ \sff_\e(u,u') \to 1/4\pi$.
In \eqref{e.Fop.ker.}, using these properties together with the fact that the heat semigroup is strongly continuous and norm bounded gives \eqref{e.Fop.convg}.
\end{proof}

Next, we prepare a lemma for taking operator limits.
To set up the notation and assumptions, fix $m\in\N$, and for   $i\in\intv{m}$ and $\e>0$, consider operators $\Top_{i,\e}(t):\Lsp^2(\R^{d_i},\d x) \to \Lsp^2(\R^{d_{i-1}},\d x)$, which we will assume to have nonnegative kernels.
We will send $\e\to 0$.
Take $I\subset\intv{m-1}$ and assume the following.
\begin{enumerate}
\item When $i\in \intv{m}\setminus I$, $\Top_{i,\e}(t)=\Top_{i}(t)$ does not depend on $\e$ and satisfies, for $c'=c'(m)$,
\begin{align}
	\label{e.l.delta.bd1}
	\Normop{\Top_{i}(t)} \leq c'\, t^{-1}\, |\log(t\wedge\tfrac{1}{2})|^{-2}\, e^{c't}\ .
\end{align}
\item 
\label{l.delta.conti}
The last operator $\Top_{m}(t)$ is strongly continuous away from $t=0,\infty$.
More precisely, for every $g\in\Lsp^{2}(\R^{d_i}, \d x)$, $\Top_{m}(t)g$ is continuous on $t\in(0,\infty)$. 
\item
\label{l.delta.delta}
When $i\in I$, the following hold for $c'=c'(m)$, for $c(\e)$ such that $c(\e)\xrightarrow{\e} 0$, and for every fixed $\delta\in(0,\infty)$,
\begin{align}
	\label{e.l.delta.bd2}
	&\Normop{\Top_{i,\e}(t)} \leq c(\e) \,e^{c' t}\,t^{-1}\ ,
\\
	&
	\label{e.l.delta.bd3}
	\NOrmop{ \int_{0}^\infty \d t\ e^{-c't}\,\Top_{i,\e}(t) } \leq c'\ ,
	\quad
	\int_{0}^{\delta} \d t \ \Top_{i,\e}(t) \xlongrightarrow{\e} \id \text{ strongly.} 
\end{align}
\end{enumerate}
\begin{lem}\label{l.delta}
Under the preceding setup, for every $t\in(0,\infty)$, 
\begin{align}
	\label{e.l.deta}
	\int_{\Sigma(t)} \d\vecu\ \prod_{i=1}^{m} \Top_{i,\e}(u_i)
	\xlongrightarrow{\e}
	\int_{\Sigma(t)} \d\vecu\ \prod_{i\notin I} \Top_{i}(u_i)	
	\text{ strongly,}
\end{align}
where the $\Sigma(t)$ on the left-hand side is $\{\sum_{i\in\intv{m}}u_i=t\}$, while the $\Sigma(t)$ on the right-hand side is $\{\sum_{i\in\intv{m}\setminus I}u_i=t\}$.
\end{lem}

The intuition behind the convergence~\eqref{e.l.deta} is that, under Assumption~\eqref{l.delta.delta}, the operators $\Top_{i,\e}(u_i)$ for $i\in I$ behave like $\delta_0(u_i)$ as $\e\to 0$.

\begin{proof}
The first step is to truncate the integrals.
Let $\Uop_{\e}(t)$ and $\Uop(t)$ respectively denote the left and right-hand sides of \eqref{e.l.deta}.
To truncation them, for small $r,\delta>0$, let
\begin{align}
    \label{e.Sigma'.}
    \Sigma'_{r,\delta}(t) 
    &:= 
    \Big\{ 
        \vecu\in[0,\infty)^{\intv{m}} 
    \Big| 
        \sum_{i=1}^{m} u_i=t, \
        u_{i}\geq r \ \forall i\in\intv{m}\setminus I, \
        u_{i} <\delta \ \forall i\in I
    \Big\}\ ,
\\
    \label{e.Sigma'}
    \Sigma'_r(t) 
    &:= 
    \Big\{ 
        \vecu\in[0,\infty)^{\intv{m}\setminus I} 
    \Big| 
        \sum_{i\in\intv{m}\setminus I} u_i=t,  \ u_{i}\geq r \ \forall i\in\intv{m}\setminus I
    \Big\}\ ,
\end{align}
let $\Uop'_{r,\delta,\e}(t)$ be obtained by replacing $\Sigma(t)$ by $\Sigma'_{r,\delta}(t)$ on the left-hand side of \eqref{e.l.deta}, and similarly let $\Uop'_{r}(t)$ be obtained by replacing $\Sigma(t)$ by $\Sigma'_{r}(t)$ on the right-hand side of \eqref{e.l.deta}.
In the analysis below, we fix $t\in(0,\infty)$ and will take $\e\to 0$ first, $\delta\to 0$ second, and $r\to 0$ last.

Let us bound the truncation errors $\normop{\Uop'_{r}(t)-\Uop(t)}$ and $\normop{\Uop'_{r,\delta,\e}(t)-\Uop_{\e}(t)}$.
First, observe that
\begin{align}
	\label{e.l.delta.1}
	\normop{\Uop'_{r}(t)-\Uop(t)}
	\leq
	\sum_{j\in \intv{m}\setminus I} \NOrmop{  
		\int_{\Sigma(t), u_j<r} \d\vecu\ \prod_{i\notin \intv{m}\setminus I} \Top_{i}(u_i)
	}\ .
\end{align}
Because of $\Sigma(t)$, the sum of $u_i$ over $i\notin I$ is equal to $t$, which forces $u_{i_0}\geq t/(m-|I|)\geq t/m$ for at least one $i_0$, and this $i_0$ cannot be $j$ when $r$ is small, because $u_j<r$. 
Use this observation to bound the right-hand side of \eqref{e.l.delta.1} by
\begin{align}
	\sum_{j\in\intv{m}\setminus I} 
	\sum_{i_0\notin I\cup\{j\}}
	\sup_{u\in[\frac{t}{m},t]} \Normop{ \Top_{i_0}(u) }
	\,
	\prod_{i\in \intv{m}\setminus (I\cup\{i_0,j\})}
	\NOrmop{ \int_0^{t} \d u\, \Top_{i}(u) } 
	\cdot
	\NOrmop{ \int_{0}^{r} \d u\, \Top_{j}(u) }\ .
\end{align}
Within the sum, the first two factors are finite thanks to \eqref{e.l.delta.bd1}, and the last factor tend to $0$ as $r\to 0$ thanks also to \eqref{e.l.delta.bd1}.
Hence, $\normop{\Uop'_{r}(t)-\Uop(t)}\to 0$ as $r\to 0$.
A similar but more tedious argument shows that $\normop{\Uop'_{r,\delta,\e}(t)-\Uop_{\e}(t)}\to 0$ as $\e\to 0$ first, $\delta\to 0$ second, and $r\to 0$ last.

To complete the proof, we show that, for every small fixed $r>0$, $\Uop'_{r,\delta,\e}(t)-\Uop'_{r,\e}(t)\to 0$ strongly as $\e\to 0$ first and $\delta\to 0$ second.
To this end, take any $g\in\Lsp^2(\R^{d_m},\d x)$, let $u_I:=\sum_{i\in I}u_i$, and express the difference acting on $g$ as
\begin{align}
\label{e.l.delta.3}
\begin{split}
	&\big( \Uop'_{r,\delta,\e}(t)-\Uop'_{r,\e}(t) \big)\,g
	=
\\
	&\int_{\Sigma'_r(t)} \d\vecu 
	\int_{[0,\delta]^I} \d \vecu_{I} \
	\prod_{i=1}^{m-1} \Top_{i,\e}(u_i) \cdot \Top_{m}(u_m-u_I) g
	-\int_{\Sigma'_r(t)} \d\vecu 
	\prod_{i\in\intv{m}\setminus I} \Top_{i}(u_i) g\ ,
\end{split}
\end{align}
Note that we have used our assumption that $m\notin I$ in the first integral in \eqref{e.l.delta.3} and also that $u_m\geq r\geq m \delta \geq u_I$ holds for small enough $\delta$.
Before bounding this expression, let us replace $\Top_{m}(u_m-u_I) g$ with $\Top_{m}(u_m) g$.
The error induced by this replacement is negligible: By Assumption~\eqref{l.delta.conti} and the bounds in \eqref{e.l.delta.bd1}--\eqref{e.l.delta.bd3}, the $\Lsp^2$ norm of the error converges to $0$ as $\delta\to 0$, uniformly in $\e>0$.
After the replacement, we have
\begin{align}
	\label{e.l.delta.2}
	\int_{\Sigma'_r(t)} \d\vecu 
	\int_{[0,\delta]^I} \d \vecu_I \
	\prod_{i=1}^{m} \Top_{i,\e}(u_i) g
	-\int_{\Sigma'_r(t)} \d\vecu 
	\prod_{i\in\intv{m}\setminus I} \Top_{i}(u_i) g\ .
\end{align}
With $r>0$ being fixed, $\normop{\Top_{i}(u_i)}$, $i\in\intv{m}\setminus I$, are uniformly bounded over $\vecu\in\Sigma'_r(t)$.
Combining this property with \eqref{e.l.delta.bd3} shows that \eqref{e.l.delta.2} converges to $0$ in $\Lsp^2$ as $\e\to 0$, for fixed $\delta,r>0$.
This completes the proof.
\end{proof}

\section{Tools for operator embedding}
\label{s.a.embed.tool}

\begin{proof}[Proof of Lemma~\ref{l.embed.eigen}]
Let $\eigenval_1\geq \eigenval_{2} \geq \cdots > 0$ be an enumeration with multiplicity of the nonzero eigenvalues of $\kerop$  and $\eigenfn'_{1},\eigenfn'_{2},\ldots \in \Lsp^2(\msp,\measure)$ be corresponding orthonormal eigenvectors, which satisfy
\begin{align}
	\label{e.l.embed.eigen.1}
	\eigenfn'_{i}(\pt)
	=
	\frac{1}{\eigenval_i} \int_{\msp} \measure(\d\pt') \, \ker(\pt,\pt')\, \eigenfn_i'(\pt')\ ,
	\quad
	\text{for $\measure$-a.e.\ $\pt$\ . } 
\end{align}
Defining $\eigenfn_{i}(\pt):=\frac{1}{\eigenval_i} \int_{\msp} \measure(\d\pt') \, \ker(\pt,\pt')\, \eigenfn_i'(\pt')$ for \emph{every} $\pt\in\msp$, we have $\eigenfn_{i}=\eigenfn'_{i}$ $\measure$ a.e.\ by \eqref{e.l.embed.eigen.1}.
Combining \eqref{e.l.embed.eigen.0} and \eqref{e.l.embed.eigen.1} shows that $\eigenfn_i\in\Lsp^2(\msp,\measure')$.
\end{proof}

Next, we prove Lemma~\ref{l.embed.tool} in two steps, stated in \eqref{l.embed.tool..1}-\eqref{l.embed.tool..2} of the following lemma.
Write $\ipp{\cdott,\cdott}_{\measuree}=\ipp{\cdott,\cdott}_{\Lsp^2(\msp,\measuree)}$ and $\norm{\cdott}_{\measuree}=\norm{\cdott}_{\Lsp^2(\msp,\measuree)}$ for $\measuree=\measure,\measure'$.
\begin{lem}\label{l.embed.tool.}
Setup as in Assumptions~\ref{a.embed.eigen} and \ref{a.embed.tool}.
\begin{enumerate}
\item \label{l.embed.tool..1}
If \eqref{e.embed.tool.1}--\eqref{e.embed.tool.2} are satisfied,
\begin{align}
	\sum_{i=1}^\infty \eigenval_{i}\, \measure' [ f\eigenfn_{i}]^2
	\leq
	 \measure'{}^{\otimes 2}\,\big[ \ker\, f^{\otimes 2}\big]
	=
	\ipp{ f, \kerop' f }_{\measure'}
	\ ,
	\quad
	f\in\Lsp^2(\msp,\measure')\ .
\end{align}
\item \label{l.embed.tool..2}
If, in addition, \eqref{e.embed.tool.3}--\eqref{e.embed.tool.5}, are satisfied, the inequality becomes an equality.
\end{enumerate}
\end{lem}
\begin{proof}
(\ref*{l.embed.tool..1})\
In view of Lemma~\ref{l.embed.eigen}, we stipulate that $\eigenfn_{i}\in\Lsp^2(\msp,\measure')$ for all $i\in \mathbb{N}$.
Take any $\psi\in D$.
It suffices to show that
\begin{align}
	\label{e.l.embed.tool.1}
	\measure'{}^{\otimes 2}\,\big[ \ker\, \psi^{\otimes 2}\big]
	-
	\sum_{i=1}^{N} \eigenval_{i}\, \measure' \,[\psi\eigenfn_{i} ]^2
	\geq
	0\ ,
	\quad N\in \mathbb{N}\ .
\end{align}
Indeed, both terms in \eqref{e.l.embed.tool.1} are continuous in $\psi$ wrt $\norm{\cdot}_{\measure'}$, so once the above inequality is proven for $\psi\in D$, it extends to all $f\in\Lsp^2(\msp,\measure')$, and that being true for all $N\in\mathbb{N}$ implies the desired result.
To show \eqref{e.l.embed.tool.1}, we use \eqref{e.embed.tool.2} to write the first term as 
$
	\lim_{\ell}
	\measure_{\ell}{}^{\otimes 2}\, \big[\ker\, \psi^{\otimes 2}\big]
$,
and use \eqref{e.embed.tool.1} to write the second term as 
$
	\lim_{\ell}
	\sum_{i=1}^{N} \eigenval_{i}\,\measure_{\ell} \,[\psi\eigenfn_{i}]^2
$.
Next, taking the difference of these limits yields
\begin{align}
	\label{e.l.embed.tool.2}
	\text{lhs of \eqref{e.l.embed.tool.1}}
	=
	\lim_{\ell}
	\measure_{\ell}{}^{\otimes 2}\, 
	\Big[ 
		\ker\, \psi^{\otimes 2} - \sum_{i=1}^{N}\eigenval_{i}\,(\psi\eigenfn_{i})^{\otimes 2}
	\Big]\ .
\end{align}
Given that $\d\measure_{\ell}/\d \measure$ is bounded on the support of $\psi$, we change the measure from $\measure_{\ell}$ to $\measure$ to express the right-hand side of \eqref{e.l.embed.tool.2} as the limit of 
\begin{align}
    \measure^{\otimes 2}
    \Big[
         \ker\, (\psi\,\d\measure_{\ell}/\d \measure)^{\otimes 2} 
         - 
         \sum_{i=1}^{N}\eigenval_{i}\,(\eigenfn_{i}\,\psi\,\d\measure_{\ell}/\d \measure)^{\otimes 2}
    \Big]
    \, .
\end{align}
By the spectral decomposition of $\kerop$, the last expression is equal to
$
	\sum_{i> N}\eigenval_{i}\,\ipp{ \eigenfn_{i} , \psi\,\d\measure_{\ell}/\d \measure }^2_{\measure}
$, 
which is clearly nonnegative, and hence \eqref{e.l.embed.tool.1} follows.

(\ref*{l.embed.tool..2})\
Let us prepare some notation and tools.
For a Hilbert--Schmidt integral operator $\op$ with kernel $\hsker$ on $\Lsp^2(\msp,\measuree)$, write
$
    \norm{\op}_{\HS}^2
    =
    \norm{\op}_{\HS,\measuree}^2
    :=
    \measuree\, [\hsker^2] 
$,
which is equal to
$
	\sum_{i,j}
	\ipp{ \phi_i, \op\,\phi_j }_{\measuree}^2
$
for any orthonormal basis $\phi_1,\phi_2,\ldots$. (See \cite[Theorem VI.22--23]{reed1972methods} for example.)
For positive operators $\op$ and $\op'$, the ordering $\op\leq\op'$ means that the difference $\op'-\op$ is positive.
Using the orthonormal-basis representation of the Hilbert--Schmidt norm, it is not difficult to check that:
\begin{align}
\label{e.hs.equivalence}
\begin{split}
	&\text{For any positive Hilbert--Schmidt } \op_1\leq\op_2\leq \cdots \leq\op', \text{ these statements are equivalent:}
\\
	&
	\qquad
	\norm{\op'-\op_\ell}_{\HS} \to 0\ ,
	\qquad
	\norm{\op_\ell}_{\HS} \to \norm{\op'}_{\HS}\ ,
	\qquad
	\op_\ell \to \op' \text{ weakly.}
\end{split}
\end{align}

We need a few more operators.
Let $\kerop_{\ell}$ be the integral operator on $\Lsp^2(\msp,\measure)$ with the kernel
\begin{align}
	\ker_{\ell}
	:=
	\ker\, \sqrt{\d\measure_{\ell}/\d\measure}^{\ \otimes 2}\ .
\end{align}
This operator is positive and Hilbert--Schmidt.
It being Hilbert--Schmidt follows by writing
\begin{align}
	\label{l.embed.tool..2.1}
	\measure_{}{}^{\otimes 2} \, \big[\ker^{2}_{\ell} \big]=\measure_{\ell}{}^{\otimes 2} \big[\ker^2\big]
\end{align}
and noting that the right-hand side is finite by \eqref{e.embed.tool.4}.
To see the positivity, observe that for any $f\in\Cloc(\msp)$ the spectral decomposition of $\kerop$ yields that
\begin{align}
	\label{e.ker.ell}
	\ipp{ f, \ker_{\ell} f}_{\measure}
	=
	\sum_{i=1}^{\infty} \eigenval_{i}\, \Ipp{ \,\eigenfn_{i}\, , \, f\sqrt{\d\measure_{\ell}/\d\measure}\, }_{\measure}^2\ ,
\end{align}
observe that this is nonnegative, and note that $\Cloc(\msp)$ is dense in $\Lsp^2(\msp,\measure)$ (which follows from \cite[Lemma~1.37]{kallenberg2021foundations} and Urysohn's lemma for example).
Next, in view of \eqref{e.ker.ell}, consider the finite-dimensional approximation $\kerop_{\ell,N}$ of $\kerop_{\ell}$
\begin{align}
	\label{l.embed.tool..2.2}
	\kerop_{\ell,N}: \Lsp^2(\msp,\measure)\to\Lsp^2(\msp,\measure)\ ,
	\quad
	\kerop_{\ell,N} 
	:=
	\sum_{i=1}^N 
	\eigenval_{i}\, \eigenfn_{i}\sqrt{\d\measure_{\ell}/\d\measure} 
	\, \big(\eigenfn_{i}\sqrt{\d\measure_{\ell}/\d\measure}\,,\,\cdott \big)_{\measure}\ ,
\end{align}
which is indeed positive and Hilbert--Schmidt.
From \eqref{e.ker.ell}--\eqref{l.embed.tool..2.2}, we see that $\kerop_{\ell,1}\leq\kerop_{\ell,2}\leq\cdots\leq \kerop_{\ell}$ and $\kerop_{\ell,N}\to\kerop_{\ell}$ weakly as $N\to\infty$.
By \eqref{e.hs.equivalence}, we have $\lim_N\norm{\kerop_{\ell,N}}_{\HS,\measure}=\norm{\kerop_{\ell}}_{\HS,\measure}$.

Let us now prove the statement of \eqref{l.embed.tool..2}.
First, the result of \eqref{l.embed.tool..1} implies that 
\begin{align}
	\label{l.embed.tool..2.kerop''}
	\kerop'': \Lsp^2(\msp,\measure') \to \Lsp^2(\msp,\measure')\ ,
	\quad
	\kerop'' f := \sum_{i=1}^\infty 
	\eigenval_{i}\, \eigenfn_{i} \, \ipp{\eigenfn_{i},f}_{\measure'}
\end{align}
defines a Hilbert--Schmidt operator with $\kerop''\leq \kerop'$, so our goal is to show that $\kerop''=\kerop'$.
This, by \eqref{e.hs.equivalence} with $\op_1=\op_2=\cdots=\kerop''$ and $\op'=\kerop'$, follows once we show that $\norm{\kerop''}_{\HS,\measure'}=\norm{\kerop'}_{\HS,\measure'}$.
To do the latter, use \eqref{e.embed.tool.4} and \eqref{l.embed.tool..2.1} to write $\norm{\kerop'}_{\HS,\measure'}=\lim_{\ell}\norm{\kerop_{\ell}}_{\HS,\measure'}$, which is equal to $\lim_{\ell}\lim_N\norm{\kerop_{\ell,N}}_{\HS,\measure}$, as we demonstrated in the previous paragraph.
So far, we have
\begin{align}
	\label{l.embed.tool..2.3}
	\norm{\kerop'}_{\HS,\measure'}^2
	=
	\lim_{\ell}\lim_N
	\norm{\kerop_{\ell,N}}_{\HS,\measure}^2
	=
	\lim_{\ell}\lim_N
	\sum_{i,j=1}^{\infty} \Ipp{ \eigenfn_i,\, \kerop_{\ell,N} \eigenfn_{j} }_{\measure}^2\ ,
\end{align}
where we used the representation of $\norm{\cdot}_{\HS,\measure}$ with the orthonormal functions $\eigenfn_{1}, \eigenfn_{2},\ldots$.
(The closure of $\operatorname{span}(\eigenfn_1,\eigenfn_2,\ldots)$ may not be the full space $\Lsp^2(\msp,\measure)$, but it does contain $\range(\kerop_{\ell,N})$, which suffices.)
Use \eqref{l.embed.tool..2.2} to write the right-hand side of \eqref{l.embed.tool..2.3} as
$
	\lim_{\ell}\lim_N \sum_{i,j\leq N} 
	\eigenval_{i}\eigenval_{j}\,  \measure_{\ell}^{\otimes 2}\, [\eigenfn_{i}\, \eigenfn_{j}]^2
$. 
By \eqref{e.embed.tool.5}, the two limits can be swapped.
Swapping the limit and using \eqref{e.embed.tool.3} in the result give
\begin{align}
	\label{l.embed.tool..2.4}
	\norm{\kerop'}_{\HS,\measure'}^2
	=
	\sum_{i,j=1}^{\infty} 
	\eigenval_{i}\eigenval_{j}\, \measure'{}^{\otimes 2}\,[ \eigenfn_{i}\, \eigenfn_{j}]^2\ .
\end{align}
Recognizing the right-hand side as $\norm{\kerop''}_{\HS,\measure'}^2$ (see \eqref{l.embed.tool..2.kerop''}) completes the proof.
\end{proof}

\section{Hilbert-space constructions}
\label{s.a.hilsp}

We prepare a consistency property of the $\isom_{I_1\ldots I_k}$s.

\begin{lem}\label{l.isom.consistent}
Take any internally disjoint $I_1, \ldots, I_k\in\I$ whose union gives an interval, and, for each $i=1,\ldots,k$, take internally disjoint $I_{i1},\ldots,I_{i n_j}\in\I$ whose union gives $I_i$.
We have
\begin{align}
	\isom_{\mathrm{all}\, I_{ij}\mathrm{s}} 
	= 
	\big( 
		\isom_{I_{11} \ldots I_{1 n_1}} \oplus \cdots \oplus \isom_{I_{k 1} \ldots I_{k n_k}}
	\big)
	\,
	\isom_{I_{1} \ldots I_{k}}\ .
\end{align}
\end{lem}
\begin{proof}
First, using \eqref{e.Yop.sum}--\eqref{e.isom.Y} for $(I,I_{1},\ldots,I_{k})\mapsto (I_i,I_{i1},\ldots I_{in_i})$ gives
\begin{align} \label{l.isom.consistent.1}
	\Ystt{I_i} 
	=
	\big( \embed \Ystt{I_{i1}} \oplus' \cdots \oplus' \embed \Ystt{I_{in_i}} \big)
	\,\isom_{I_{i1} \ldots I_{in_i}} \ .
\end{align}
The embedding $\embed$ here maps $\BMsp(\Psp_{I_{ij}})\to\BMsp(\Psp_{I_i})$.
On both sides, apply the embedding $\embed:\BMsp(\Psp_{I_i})\to \BMsp(\Psp_{I})$, note that composing the two embeddings give the embedding $\embed:\BMsp(\Psp_{I_{ij}})\to\BMsp(\Psp_{I})$, sum both sides over $i=1,\ldots,k$ under $\oplus'$, right multiply both sides by $\isom_{I_1\ldots I_{k}}$, and use \eqref{e.Yop.sum}--\eqref{e.isom.Y} to write the left-hand side of the result as $\Ystt{I}$.
Doing so gives
\begin{align}
	\label{e.l.isom.consistent.1}
	\Ystt{I} 
	=
	\bigoplus'\nolimits_{ij} \embed \Ystt{I_{ij}} 
	\cdot
	\isom'\ ,
	\quad
        \text{where }
	\isom'
	:=
	\big( \isom_{I_{11} \ldots I_{1 n_1}} \oplus \cdots \oplus \isom_{I_{k 1} \ldots I_{k n_k}} \big)
	\,
	\isom_{I_{1} \ldots I_{k}}\ .
\end{align}
Given \eqref{e.l.isom.consistent.1}, we seek to apply the uniqueness statement in Lemma~\ref{l.polar}\eqref{l.polar.1} with $\Yop_1=\Ystt{I}$ and $\Yop_2=\oplus'_{ij}\embed \Ystt{I_{ij}}$.
The conditions on $\Yop_1,\Yop_2$ are satisfied due to \eqref{e.translation.Y.id}.
The operator $\isom'$ is an isometric embedding since the $\isom$s are, so in particular $\nullsp\isom'=\{0\}$. 
Note that $\nullsp\Ystt{I}=\{0\}$.
Also, $\Ystt{I}=\bigoplus'\nolimits_{ij} \embed\Ystt{I_{ij}} \cdot \isom_{\text{all} I_{ij}\text{s}}$ and $\nullsp\isom_{\text{all} I_{ij}\text{s}}=\{0\}$ by definition.
The uniqueness statement in Lemma~\ref{l.polar}\eqref{l.polar.1} applies and completes the proof.
\end{proof}

Combining Lemma~\ref{l.isom.consistent} and \eqref{e.isomu} immediately gives that
\begin{align}
	\label{e.isomu.consistent}
	\isomu^{\intvs''}_{\intvs}=\isomu^{\intvs''}_{\intvs'}\isomu^{\intvs'}_{\intvs}\ ,
	\quad
	\intvs\preceq\intvs'\preceq\intvs''\ .
\end{align}
Equipped with this property, we now prove Proposition~\ref{p.hilsp}.

\begin{proof}[Proof of Proposition~\ref{p.hilsp}]
Consider 
\begin{align}
	\hilsp'
	:= 
	\bigoplus_{\intvs\in\Intvs} \lsp^2(\intvs\times\N)
	=
	\big\{ v=(v_{\intvs})_{\intvs\in\Intvs} \, | \, v_{\intvs} = 0 \text{ for all but finitely many } \intvs\text{s} \big\}\ .
\end{align}
For $v\in\hilsp'$, let $\supp v :=\{\intvs\in\Intvs\,|\, v_{\intvs}\neq 0\}$ and $\refine v := \{ \intvss\in\Intvs\,|\, \intvs \preceq \intvss \text{ for all } \intvs\in\supp v \}$, and define
\begin{align}
	\label{e.p.hilsp.}
	\calN
	:=
	\Big\{ v\in\hilsp'_{\infty} \, \Big| \, \sum_{\intvs\in\supp v} \isomu_{\intvs}^{\intvss} v_{\intvs} = 0\  \text{ for some } \intvss \in \refine v \Big\}\ .
\end{align}
By \eqref{e.isomu.consistent}, once the condition in \eqref{e.p.hilsp.} holds for one $\intvss\in\refine v$, it holds for all $\intvss'\in\refine v$.
This property implies that $\calN$ is a linear subspace of $\hilsp'$.
Let $\hilsp'':=\hilsp'/\calN$, write $\overline{v}\in\hilsp''$ for the equivalence class containing $v\in\hilsp'$ modulo $\calN$, and endow $\hilsp''$ with the inner product
\begin{align}
	\label{e.p.hilsp..}
	\Ip{ \overline{v}, \overline{w} }_{\hilsp''}
	:=
	\IP{ 
		\sum_{\intvs\in\supp v} \isomu^{\intvss}_{\intvs} v_{\intvs}
		,  
		\sum_{\intvs\in\supp w} \isomu^{\intvss}_{\intvs} w_{\intvs}
	}_{\lsp^2(\intvss\times\N)}\ ,
\end{align}
where $\intvss$ is any element in $\refine v\cap \refine w$.
By \eqref{e.isomu.consistent}, the right-hand side of \eqref{e.p.hilsp..} does not depend on $\intvss$ as long as it is in $\refine v\cap \refine w$.
It follows from the definition of $\calN$ that the right-hand side of \eqref{e.p.hilsp..} does not depend on the representatives $v,w$ of the equivalence classes. 
We now define the universal Hilbert space $\hilsp$ as the closure of $\hilsp''$.
Write $\proj:\hilsp'\to\hilsp''\hookrightarrow\hilsp$ for the quotient map from $\hilsp'$ onto $\hilsp''=\hilsp'/\calN$, and write $\embedd'_{\intvs}:\lsp^2(\intvs\times\N)\to\hilsp'$ for the coordinatewise embedding.
It is not hard to check that $\embedd_{\intvs}:=\proj\,\embedd'_{\intvs}$ gives the desired isometric embedding of $\lsp^2(\intvs\times\N)$ into $\hilsp$.
\end{proof}

Let us prove Lemma~\ref{l.naimark}.

\begin{proof}[Proof of Lemma~\ref{l.naimark}]
(\ref*{l.naimark.1})\
We begin with a reduction.
Write $\ipp{\cdott,\cdott}$ for the inner product on $\Lsp^2_{I'}$, and note that $v\mapsto\ipp{v,\intersop^{I'}_{I}v}$ and
$
	v \mapsto
	\ipp{v,\Yopu_{I'}\,\proj_{I}\,\Yopu^{*}_{I'}v}
$
define quadratic forms on $\Lsp^2_{I'}$.
The first quadratic form vanishes on $\nullsp\intersop^{I'}_{I'}$~, because $\intersop^{I'}_{I}\leq\intersop^{I'}_{I'}$~.
The second quadratic form also vanishes on $\nullsp\intersop^{I'}_{I'}$~, because $\nullsp\intersop^{I'}_{I'}=\nullsp\Yopu^{*}_{I'}$~.
Let $(\eigenval_i,\eigenfn_i)$, $i\in\N$, be the eigenpairs of $\intersop^{I'}_{I'}$ with nonzero eigenvalues.
Since $\eigenfn_1,\eigenfn_2,\ldots$ is an orthonormal basis for $\nullsp\intersop^{I'}_{I}\,{}^{\perp}$, it suffices to show that 
\begin{align}
	\label{e.l.naimark.1.goal}
	\Ipp{ \eigenfn_i\ , \intersop^{I'}_{I}\, \eigenfn_{i} }
	=
	\Ipp{ \eigenfn_i\ ,\Yopu_{I'}\,\proj_{I}\,\Yopu^{*}_{I'}\, \eigenfn_{i} }
	=
	\Norm{ \proj_{I}\,\Yopu^{*}_{I'}\, \eigenfn_{i} }^2
	\ .
\end{align}

To prove \eqref{e.l.naimark.1.goal}, decomposing $I'$ into internally disjoint $I_1,I,I_3$, we begin by claiming that
\begin{align}
	\label{e.l.naimark.1.}
	\proj_{I}\, \embedd_{I_1\,I\,I_3}
	=
	\embedd_{I}\,
	\embedd^{I_1\,I\,I_3\,*}_{I}\ .
\end{align}
To see why this claim holds, consider the coordinatewise embedding
\begin{align}
	\label{e.l.naimark.1..}
	\embedd_{I}^{I_1\,I\,I_3}:
	\lsp^2(\{I\}\times \N)
	\hookrightarrow
	\lsp^2(\{I_1,I,I_3\}\times \N)\ .
\end{align}
The domain and range of $\embedd_{I}^{I_1\,I\,I_3}$ are isometrically embedded into $\hilsp$ respectively by $\embedd_{I}$ and $\embedd_{I_1\,I\,I_3}$, and these embeddings satisfy $\embedd_{I}=\embedd_{I_1\,I\,I_3}\,\embedd_{I}^{I_1\,I\,I_3}$, thanks to \eqref{e.hilsp.consistent} and \eqref{e.isomu.id1} for $\intvs=\{I\}$ and $\intvs'=\{I_1,I,I_3\}$.
The claim \eqref{e.l.naimark.1.} is straightforward to verify from these properties.
Next, on both sides of \eqref{e.l.naimark.1.}, right multiply with $\isomu^{I_1\,I\,I_3}_{I}$, use \eqref{e.hilsp.consistent} for $\intvs=\{I'\}$ and $\intvs'=\{I_1,I,I_3\}$ on the left-hand side, use \eqref{e.isomu.id2} for $\intvs=\{I'\}$ and $\intvs'=\{I_1,I,I_3\}$ on the right-hand side, and right multiply the result with $\Yop_{I'}\eigenfn_{i}$.
Doing so gives
\begin{align}
	\label{e.l.naimark.1...}
	\proj_{I}\, \Yopu_{I'}\, \eigenfn_{i}
	=
	\embedd_{I}\,
	\embedd^{I_1\,I\,I_3\,*}_{I}\,
	\isom_{I_1\,I\,I_3}\, \Ystt{I'}^{*}\, \eigenfn_{i}\ .
\end{align}
Use \eqref{e.l.polar} for $\Yop_1=\Ystt{I'}$ and $\Yop_2=\embed\Ystt{I_1}\oplus'\embed\Ystt{I}\oplus'\embed\Ystt{I_3}$ to obtain that
$
	\isom_{I_1\,I\,I_3}\, \Ystt{I'}^{*}\, \eigenfn_{i}
	=
	\embed\Ystt{I_1}^*\eigenfn_{i}\oplus\embed\Ystt{I}^*\eigenfn_{i}\oplus\embed\Ystt{I_3}^*\eigenfn_{i}
$,
and note that $\embedd^{I_1\,I\,I_3\,*}_{I}$ projects the last expression to $\embed\Ystt{I}^*\eigenfn_{i}$.
We conclude that the right-hand side of \eqref{e.l.naimark.1.goal} is equal to $\norm{\embed\Ystt{I}^*\eigenfn_{i}}^2$, which by Proposition~\ref{p.embed..} is equal to the left-hand side.

(\ref*{l.naimark.2})\
It suffices to show that, for each $i$, $\norm{\proj_{I_k}\embedd_{I'}e_i}\to 0$ $\P$-a.s.
To this end, use Part~\eqref{l.naimark.1} together with the property $\Ystt{I'}^*\eigenfn_i/\sqrt{\eigenval_i}=e_i$ (see \eqref{e.Ystt}) to obtain
\begin{align}
	\Norm{ \proj_{I_k}\,\embedd_{I'}\, e_i }^2
	=
	\frac{1}{\eigenval_i}\, \Ipp{ \eigenfn_i, \intersop^{I'}_{I_k} \eigenfn_i }
	\leq
	\frac{1}{\eigenval_i}
	\Big(\int_{\Psp_{I'}}\, \PMw{}^{\otimes 2}(\d\Path,\d\Pathh)
	\, \inters^{\Path,\Pathh}_{I'}(I_k)^2 \Big)^{1/2}\ ,
\end{align}
where we have used the Cauchy--Schwarz inequality wrt $\PMw{}^{\otimes 2}$ and that $\norm{\eigenfn_{i}}=1$ in the  inequality.
Applying the dominated convergence theorem to the last integral gives the desired result.
\end{proof}

\bibliographystyle{alpha}
\bibliography{shf}

\end{document}